\documentclass[12pt]{amsart}
\usepackage[utf8]{inputenc}
\usepackage{graphicx}
\usepackage{amssymb}
\usepackage{bbm}
\usepackage{bm}
\usepackage{mathtools}
\usepackage{amsthm}
\usepackage{amsmath}
\usepackage{amsbsy}
\usepackage{amstext}
\usepackage{amscd}
\usepackage{nccmath}
\usepackage{bbm}
\usepackage[usenames]{xcolor}
\usepackage{enumerate}
\usepackage{float}
\usepackage{enumitem}

\newcommand{\subscript}[2]{$#1 _ #2$}

\newcommand{\cP}{\mathcal{P}}

\newcommand{\cU}{\mathcal{U}}
\newcommand{\cV}{\mathcal{V}}
\newcommand{\cW}{\mathcal{W}}
\newcommand{\PS}{\mathcal{PS}}
\newcommand{\K}{\kappa}
\newcommand{\C}{k}
\newcommand{\E}{\mathrm{E}}
\newcommand{\cov}{\mathrm{Cov}}
\newcommand{\NC}{\mathit{NC}}
\newcommand{\Tr}{\mathrm{Tr}}

\newcommand{\iCycleGamma}[1]{[\![m_{#1} ]\!]}

\allowdisplaybreaks
\newcommand{\ab}{\allowbreak}

\newcommand{\thebottomline}{\renewcommand{\thefootnote}{}
  \renewcommand{\footnoterule}{}
  \phantom{M}\footnotetext{\tiny{}\hfill
    \textit{\noindent\romannumeral\day.%
\romannumeral\month.\romannumeral\year}}}

\newtheorem{theorem}{Theorem}[section]
\newtheorem{proposition}[theorem]{Proposition}
\newtheorem{corollary}[theorem]{Corollary}
\newtheorem{lemma}[theorem]{Lemma}

\theoremstyle{definition}
\newtheorem{definition}[theorem]{Definition}
\newtheorem{notation}[theorem]{Notation}
\newtheorem{example}[theorem]{Example}
\newtheorem{remark}[theorem]{Remark}
\newtheorem{notation and remark}[theorem]{Notation and remark}

\title[{Higher order free cumulants}] {Asymptotic limit of Cumulants and Higher order free cumulants of Complex Wigner Matrices}

\author[mingo]{James A. Mingo$^{(*)}$} \address{Department of Mathematics and Statistics, Queen's University, Jeffery Hall, Kingston, Ontario, K7L 3N6, Canada}
\email{mingo@mast.queensu.ca}

\author[munoz]{Daniel Munoz George$^{(\dagger)}$}
\address{Department of Mathematics and Statistics, Hong Kong University of Science and Technology, Clear Water Bay, Kowloon, Hong Kong}
\email{dmunoz@ust.hk}

\thanks{$^{(*)}$ Research supported by a Discovery Grant from
  the Natural Sciences and Engineering Research Council of
  Canada}
\thanks{$^{(\dagger)}$ Research Supported by Hong Kong RGC GRF 16305421 and NSFC 12222121.}


\begin{document}\thispagestyle{empty}

\begin{abstract}
We compute the fluctuation moments $\alpha_{m_1,\dots,m_r}$ of a Complex Wigner Matrix $X_N$ given by the limit $\lim_{N\rightarrow\infty}N^{r-2}\ab\C_r(\Tr(X_N^{m_1}),\dots, \Tr(X_N^{m_r}))$. We prove the limit exists and characterize the leading order via planar graphs that result to be trees. We prove these graphs can be counted by the set of non-crossing partitioned permutations which permit us to express the moments $\alpha_{m_1,\dots,m_r}$ in terms of simpler quantities $\K_{m_1,\dots,m_r}$ known as the higher order cumulants. As for lower order dimensions ($r\leq 3$) we observe that while the moments have a more elaborated expression the cumulants are simpler.
\end{abstract}

\maketitle

\section{Introduction and main results}\label{Section: Introduction}

A key point of Wigner's celebrated semi-circle theorem was its universality; the limit of the empirical eigenvalue distribution of a Wigner matrix depended only on the variance of the entries. The rest of the information was assumed lost to universality. In this paper we show how to recover the distribution of the entries from the higher order fluctuation moments. We show that the fluctuation moments are described by some planar graphs with certain topological properties.  This connection with planar graphs goes back to early work of 't Hooft \cite{gth}. However as the order increases, the graphs get quite complicated. In this paper we show that by utilizing higher order free cumulants, the fluctuation moments can be easily understood. The higher order free cumulants are just obtained by M\"obius inversion from the fluctuation moments, so this is quite a simple transformation. 

While it is surprising that one can recover local information from the fluctuation moments, which are manifestly global, there were two precedents. The first of course was Wigner's original work \cite{W1}. The second was the work of Khorunzhy,  Khoruzhenko, and Pastur \cite{KKP}, where they showed that the linear statistics of the empirical eigenvalue distribution can described using the first four moments of the distribution of the entries. Since \cite{KKP} was before higher order free cumulants were known, there were no further results until our recent paper \cite{MM} where we showed how to recover the first six moments from the fluctuation moments. In this paper we complete the project to get all moments.

A further impetus for this work was the recent breakthrough of Borot, Charbonnier, Garcia-Failde, Leid, and Shadrin \cite{bcgf}. Based on work of Maxim Kazarian \cite{bdks} and collaborators,  in the paper \cite{bcgf} a connection was established between the topological recursion of some matrix integrals and higher order free cumulants.  However, until \cite{MM} and this paper, these cumulants were only computed for the Gaussian case, in which case they were all $0$ except $\kappa_2$. 


In Wigner's original series of papers \cite{W1, W2}, he showed that the moments of the limiting eigenvalue distribution were given by the Catalan numbers, known for counting many planar combinatorial objects. Much later, the global fluctuation moments of a Wigner matrix, $\cov(\Tr(X^n), \Tr(X^m))$, were found by Khorunzhy, Khoruzhenko, and Pastur in \cite{KKP}.  These moments were reinterpreted using the non-crossing annular permutations introduced in \cite{MN}, and their interactions with
deterministic matrices were described in \cite{MMPS}. 
%
%

In \cite{MM} the authors adapted the techniques of \cite{MMPS} to compute the large $N$ limit of the  third order moments of a Wigner matrix, $N\C_3(\Tr(X^n), \ab \Tr(X^m), \Tr(X^p))$, and showed they can be simply described by its higher order free cumulants introduced in \cite{CMSS}, here $\C_3$ denotes the third classical cumulant as defined in Section \ref{Section: Non-crossing objects} Equation \ref{Equation: First order free cumulants}.

Besides the Wigner model, the quantities $\C_r(\Tr(X^{m_1}),\dots, \Tr(X^{m_r}))$ have been studied for other models. In \cite{MN}, Mingo and Nica study the second order case $\C_2(\Tr(X^n), \Tr(X^m))$ for the Wishart model. In \cite{PSp}, Pluma and Speicher look at $\C_r(\Tr(X^{m_1}),\dots, \Tr(X^{m_r}))$ for the SYK-model while in \cite{CMSS} they deal with unitarily invariant ensembles. In this paper we show that the normalized cumulants $N^{r-2}\C_r(\Tr(X^{m_1}),\dots,\ab \Tr(X^{m_r}))$ converge in the large $N$ limit. Moreover we characterize the leading order and find and very simple expression for the limit in terms of its higher order free cumulants (see Section \ref{Section: Non-crossing objects} Equation \ref{Equation: Higher order free cumulants}). In this context, the present work represents the first instance in which free cumulants of arbitrary order are computed for a random matrix model, beyond the GUE case, where the calculations are elementary.

For our model we consider $R$-diagonal variables. A complex random variable $x$ is called $R$-diagonal if its classical cumulants are all $0$ except,
$$\C_r(x,\overline{x},\dots,x,\overline{x}).$$
We use the notation,
$$\beta_{2n}\vcentcolon = \C_{2n}(x,\overline{x},\dots,x,\overline{x}).$$
The sequence $(\beta_n)_n$ will characterize our limit and we call this the cumulant sequence of $x$. $R$-diagonal variables has been considered in other models, for example in \cite{AM} they consider the second order cumulants as defined in Equation \ref{Equation: Higher order free cumulants} for the product of $R$-diagonal variables. The model $R$-diagonal variable  is the Gaussian complex random variable for which $\beta_n=0$ for all $n\neq 2$ and $\beta_2=1$. In this sense, our model generalizes the $GUE$ model.

\begin{definition}\label{Definition: Complex Wigner Matrix}
By a complex Wigner matrix $X_N$ we mean a self-adjoint
$N\times N$ random matrix of the form
$X_N=\frac{1}{\sqrt{N}}(x_{i,j})$ such that,

\begin{itemize}

\item[$\diamond$]
the entries are complex random variables;

\item[$\diamond$]
the matrix is self-adjoint: $x_{i,j}=\overline{x_{j,i}}$;
 
\item[$\diamond$] 
all entries on and above the diagonal are independent: 
$\{x_{i,j}\}_{i<j}\ab \cup \{x_{i,i}\}_i$ are independent;

\item[$\diamond$]
the entries above the diagonal, $\{x_{i,j}\}_{i<j}$, are
identically distributed and $R$-diagonal with cumulants sequence $(\beta_n)_n$;

\item[$\diamond$] 
the diagonal entries, $\{x_{i,i}\}_{i}$, are identically
distributed $N(0,\beta_2)$;

\item[$\diamond$] 
$\E(x_{i,j})=0$ for all $i,j$,



\end{itemize}
\end{definition}

We are interested in the quantities,
\begin{equation}\label{Equation: definition of higher order moments N}
\alpha_{m_1,\dots,m_r}^{(N)}\vcentcolon = N^{r-2}\C_r(\Tr(X^{m_1}),\dots, \Tr(X^{m_r})).
\end{equation}
and their limit,
\begin{equation}\label{Equation: definition of higher order moments}
\alpha_{m_1,\dots,m_r}\vcentcolon = \lim_{N\rightarrow\infty} N^{r-2}\C_r(\Tr(X^{m_1}),\dots, \Tr(X^{m_r})).
\end{equation}
We regard the quantities $(\alpha_{m_1,\dots,m_r})_{m_1,\dots,m_r}$ as moment sequences; we are interested in their corresponding free cumulants 
$(\K_{m_1,\dots,m_r})_{m_1,\dots,m_r}$ 
defined as in Section \ref{Section: Non-crossing objects} Definition \ref{Definition: Higher order free cumulants}. Historically, for $r\leq 3$, the cumulants describe the limit of Equation \ref{Equation: definition of higher order moments} in a very simple way. We will show that this behavior still holds for any $r$. The normalization $N^{r-2}$ is motivated by the phenomena observed in lower order cases ($r\leq 3$) and that also holds for other models such as the Wishart model. Our main result shows that the limit of Equation \ref{Equation: definition of higher order moments} exists and it provides explicit expressions for both $(\alpha_{m_1,\dots,m_r})_{m_1,\dots,m_r}$ and $(\K_{m_1,\dots,m_r})_{m_1,\dots,m_r}$. Before stating our main results let us go through a brief recap of lower order cases. Recall that the semi-circular law can be described by requiring that all free cumulants $\kappa_n = 0$ for $n>2$. In the case of the
fluctuation moments, it was shown in \cite{MMPS} that the
fluctuation moments can be described by requiring that
$\kappa_2 = \beta_2$ and $\kappa_n = 0$ for $n > 2$ and
$\kappa_{p,q} = 0$ for $(p, q) \not = (2, 2)$ and
$\kappa_{2,2} = 2 \beta_4$, where
$\{\kappa_{p,q}\}_{p,q}$ are the second order free cumulants. In \cite{MM} we show that for the third order case we require $\K_2=\beta_2, \K_{2,2}=2\beta_4, \K_{2,2,2}=4\beta_6,
\K_{2,1,1} = -2\beta_4$ and $0$ otherwise. Let us now state our main results.

\begin{notation}
For $r\geq 1$ let $\gamma_{2r}^{par}\in S_{2r}$ to be the permutation $$(1,2)\cdots (2r-1,2r).$$
For $\sigma\in \cP_2(2r)$ and $\tau\in \cP(2r)$ with $\sigma\leq \tau$ we let $\Gamma(\tau,\sigma,\gamma_{2r})$ to be the bipartite graph with white vertices the blocks of $\gamma_{2r}^{par}\vee\sigma$, black vertices the blocks of $\tau$ and edges indexed by the blocks of $\sigma$ where the block $B$ of $\sigma$ is an edge going from the black vertex that contains $B$ to the white vertex that contains $B$. For a partition $\tau\in \cP(2r)$ we say that $\tau$ \textit{doesn't admit a non-crossing partitioned permutation} if any block of $\tau$ contains the same number of even and odd numbers, $\tau\vee\gamma_{2r}^{par}=1_{2r}$, and there is no $\sigma\in \cP_2(2r)$ where each block of $\sigma$ contains one even and one odd number and such that $\Gamma(\tau,\sigma,\gamma_{2r}^{par})$ is a tree.
\end{notation}

\begin{definition} 
Let $m_1,\dots,m_r\in\mathbb{N}$ and $m=\sum_i m_i$. Let $\sigma\in S_m$ be a permutation whose cycle decomposition is a product of transpositions. Let $\gamma$ be the permutation with cycles $(1,\dots,m_1)\cdots (m_1+\cdots+m_{r-1}+1,\dots,m)$. For a cycle $B=(u,v)$ of $\sigma$ we say that $B$ is a \textit{loop block} if $u$ and $v$ are in the same cycle of $\gamma\sigma$. For a pair $(\tau,\sigma)\in \cP(m)\times \cP_2(m)$ we say that  the pair is \textit{loop free} if $\sigma\leq \tau$ and if $B$ is a loop block of $\sigma$ then $B$ is also a block of $\tau$. We denote by $\PS_{NC_2}^{loop-free}(m_1,\dots,m_r)$ to all the partitioned permutations $(\tau,\sigma)\in \PS_{NC_2}(m_1,\dots,m_r)$ where $(\tau,\sigma)$ is loop free. Here $\PS_{NC_2}(m_1,\dots,m_r)$ denotes the subset of the set of non-crossing partitioned permutations where $\sigma$ has only cycles of size $2$.
\end{definition}

\begin{notation}
Let $r\in\mathbb{N}$ and let $\PS_{\NC_2}(2^{r}):=\PS_{\NC_2}(2,\dots,2)$ be the set of non-crossing partitioned permutations on $r$ cycles each consisting on $2$ points. Let,
$$C_{2r}=\sum_{\substack{\tau\in \cP(2r) \\ \tau\text{ doesn't admit }\\ \text{a non-crossing} \\ \text{partitioned permutation}}}\prod_{D\in\tau}\beta_{|D|}.$$
\end{notation}

\begin{theorem}\label{Theorem: Main theorem 1}
Let $X_N$ be a complex Wigner matrix with moment sequence give by the limit of Equation $(\ref{Equation: definition of higher order moments})$. Let $u,v\geq 0$ be non-negative integers such that $u$ is even and let $r=u/2+v$. Let $(\delta_r)_r$ be the sequence given by the recursive equation $\delta_1=\beta_2$ and
$$\delta_r=2^{r-1}\left(\beta_{2r}+C_{2r}\right)-\sum_{\substack{(\cU,\sigma)\in \PS_{\NC_2}(2^r) \\ \cU\neq 1_{2r}}}\prod_{V\in U}\delta_{\#(\sigma|_V)}-\prod_{V\in U}2^{\#(\sigma)-1}\beta_{|V|},$$
for $r\geq 2$. Then the higher order free cumulants of $X_N$ are give by
\begin{equation}
\K_{1,\dots,1,2,\dots,2}=(-1)^{u/2}\frac{u!}{2^{u/2}(u/2)!}\delta_r,
\end{equation}
where there are exactly $v$ indices that are $2$ and $u$ indices that are $1$ and $(u,v)\neq(2,0)$ and $0$ otherwise.
\end{theorem}

\begin{theorem}\label{Theorem: Main theorem 2}
Let $X_N$ be a complex Wigner matrix. Then for any $m_1\dots,m_r$,
\begin{equation}
\alpha_{m_1,\dots,m_r}=\sum_{(\tau,\sigma)\in \PS_{\NC_2}^{loop-free}(m_1,\dots,m_r)}\K_{(\tau,\sigma)},
\end{equation}
where $\K_{(\tau,\sigma)}$ is the multiplicative extension of $\K_{m_1,\dots,m_r}$ given as in Theorem \ref{Theorem: Main theorem 1}.
\end{theorem}

\begin{corollary}\label{Theorem: Main theorem 3}
Let $X_N$ be a GUE random matrix, i.e a Wigner matrix with off diagonal entries Complex Gaussian. Then for any $m_1,\dots,m_r$,
$$\alpha_{m_1,\dots,m_r}=|\NC_2(m_1,\dots,m_r)|.$$
\end{corollary}

Recently it has been shown that higher order free cumulants have an interpretation in terms of a phenomenon, known as topological recurrence, which is a universal recurrence on the genus of a map on a surface and the number of its boundaries. In \cite{bcgf}, Borot, Charbonnier, Garcia-Failde, Leid, and Shadrin showed that this relation between higher order moments and higher order cumulants can be expressed in the form of power series connecting fluctuation moments to higher order free cumulants. This was already done in \cite{CMSS} in the second order case. In this paper we compute, for the first time the higher order cumulants of any order.

After this introduction, the organization of the paper is as  follows.
In Section \ref{Section: Non-crossing objects} we review some basic material on non-crossing partitions, permutations and partitioned permutations. We also introduce the cumulants. In Section \ref{Section: Graph theory} we introduce some graph theory and techniques that we will need to establish our main result. In Section \ref{Section: Asymptotics} we show that the large $N$ limit defined in Equation \ref{Equation: definition of higher order moments} exists and in Section \ref{Section: Leading order} we identify the leading order that correspond to the moments, this provides an expression for both the moments and the cumulants and it proves our main results.

\section{Non-crossing sets}\label{Section: Non-crossing objects}

\subsection{Non-crossing partitions and free cumulants} We begin by recalling from \cite{NS} some basic facts about
non-crossing partitions and free cumulants. We let $[n] =
\{1, 2, \dots, n\}$, a partition of $[n]$ is a disjoint collection of sets $V_1,\dots,V_k$ whose union is $[n]$. We call $V_i$ the \textit{blocks of }$\pi$ and use the notation $V_i\in \pi$. The set of partitions of $[n]$ will be denoted $\cP(n)$. $0_n$ denotes the partition of $[n]$ with all blocks singletons and $1_n$ denotes the partition of $[n]$ with only one block. We say $\pi$ is \textit{non-crossing} if we cannot find $i < j < k <l \in [n]$ such that $i$ and $k$ are in one block of $\pi$ and $j$ and $l$ are in a different block of $\pi$. We denote by $\NC(n)$ the subset of $\cP(n)$ consisting of non-crossing partitions. When all blocks of $\pi$ have size $2$ we call $\pi$ a pairing and denote to the set of pairings by $\cP_2(n)$.
       
If $\pi \in \cP(n)$ and $A \subseteq [n]$ we can form $\pi|_A$, the \textit{restriction} of $\pi$ to $A$ as follows. If the blocks of $\pi$ are $V_1, \dots, V_k$, then the blocks of $\pi|_A$ are the non-empty elements of $V_1 \cap A, \dots, V_k \cap A$. For two given partitions $\pi,\sigma$ we say that $\pi\leq \sigma$ if any block of $\pi$ is contained in a block of $\sigma$. Equipped with this partial order the sets $\cP(n)$ and $\NC(n)$ becomes a partially ordered set. The partition $\pi\vee\sigma$ denotes the smallest partition greater or equal than $\pi$ and $\sigma$.

The free cumulants of a random variable $X$ are the sequence $(\K_n)_n$ given by the recursive equations
\begin{equation}\label{Equation: First order free cumulants}
\E(X^m)=\sum_{\pi\in \NC(m)}\K_{\pi},
\end{equation}
where
$$\K_{\pi}=\prod_{B\in\pi}\K_{|B|}.$$
More generally, the mixed cumulants of a sequence of random variable $(X_i)_i$ is the sequence determined by the recursive equations
\begin{equation}\label{Equation: First order free cumulants multivariate}
\E(X_1\cdots X_m)=\sum_{\pi\in \NC(m)}\K_{\pi}(X_1,\dots,X_m),
\end{equation}
where,
$$\K_{\pi}(X_1,\dots,X_m)=\prod_{\substack{B\in\pi\\ B=\{i_1,\dots,i_k\}}}\K_k(X_{i_1}\cdots X_{i_k}).$$
For more details on non-crossing partitions and cumulants we refer to \cite[Lecture 9,Lecture 11]{NS}.

\subsection{Non-crossing permutations} As one can see from the definition, the order in $[n]$ plays an important role for elements of $\NC(n)$. In this paper it will be very important to switch our focus to non-crossing permutations. We let $S_n$ be the group of permutations of $[n]$. For a permutation $\pi\in S_n$ we let $0_{\pi}$ be the partition obtained from its cycle decomposition. We say that $\pi \in S_n$ is \textit{non-crossing}, if $0_{\pi}$ is non-crossing. When there is no risk of confusion we just write $\pi$ to denote both the permutation and the partition $0_{\pi}$. Using a theorem of Biane we can describe easily the permutations that give non-crossing partitions.

For a partition or a permutation, $\pi$, we let $\#(\pi)$
denote the number of blocks of $\pi$ if it is a partition
and the number of cycles of $\pi$ if it is a permutation. We
let $\gamma_n = (1, 2, \dots, n) \in S_n$ be the permutation
with one cycle and the elements in increasing order. We
compose our permutations from right to left: $\pi\gamma_n$
means perform $\gamma_n$ first then $\pi$.  Biane's rule
is that for $\pi \in S_n$ we have $\#(\pi) +
\#(\pi^{-1}\gamma_n) \leq n + 1$ with equality only if $\pi$
is non-crossing, for more details and references one can see \cite{B} and \cite[Ch.~5]{MSBook}. 

Another way to write Biane's rule is to use a \textit{length} function on $S_n$. For a permutation $\pi \in S_n$ we let $|\pi|$ be the minimal number of transpositions required to write $\pi$ as a product of transpositions. One can then check that $\#(\pi) + |\pi| = n$. One immediately has the triangle
inequality $|\pi\sigma| \leq |\pi| + |\sigma|$ for any $\pi,
\sigma \in S_n$. Then Biane's inequality becomes $|\gamma_n|
\leq |\pi| + |\pi^{-1} \gamma_n|$ with equality only if
$\pi$ is non-crossing. Because of this we shall say that
$\pi$ is non-crossing \textit{relative} to $\gamma_n$. In \cite[Equation 2.9]{MN} they prove that for any two permutations $\pi,\gamma$ the following inequality holds.
\begin{equation}\label{Inequality: Mingo and Nica inequality}
\#(\pi)+\#(\gamma)+\#(\pi^{-1}\gamma) \leq n+2\#(\pi\vee\gamma).
\end{equation}
Inspired in all these facts we define the notion of a non-crossing permutation relative to any other permutation $\gamma$

\begin{definition}\label{Definition: Non-crossing permutation relative to gamma arbitrary}
Let $\gamma\in S_n$ and $\pi\in S_n$. We say that $\pi$ is non-crossing relative to $\gamma$ if,
\begin{enumerate}
    \item $\pi\vee\gamma=1_n$, and,
    \item $\#(\pi)+\#(\gamma)+\#(\pi^{-1}\gamma) = n+2.$
\end{enumerate}
We denote to the set of non-crossing permutations relative to $\gamma$ by $S_{NC}(\gamma)$. When the first condition is not satisfied but we still have Equality in \ref{Inequality: Mingo and Nica inequality}, i.e,
$$\#(\pi)+\#(\gamma)+\#(\pi^{-1}\gamma) = n+2\#(\pi\vee\gamma).$$
Then we say that $\pi$ is non-connecting non-crossing and denote to the set of these permutations by $S_{NC}^{nc}(\gamma)$.
\end{definition}

In this paper we will be concerned with permutations on a
\textit{multi-annulus}. In terms of permutations this means
we consider a permutation $\gamma$ with one or more cycles. If $m_1, m_2, \dots, m_r$ are positive integers, $n = m_1 + \cdots + m_r$, we let $\gamma_{m_1,m_2, \dots, m_r}$ be the
permutation in $S_n$ with $r$ cycles, the $i^{th}$ cycle
being $[\![m_i ]\!] := (m_1\ab + \cdots + m_{i-1} +1, \dots,
m_1 + \dots + m_i)$. When $r = 1$ we are in the case of
$\NC(n)$. When $r > 1$ we are in the case of a
\textit{r-annulus}. In \cite[Def.~3.5]{MN}, conditions were
given, similar to the ones for non-crossing partitions
above, that ensured that when $r = 2$ we can draw two
circles representing the cycles of $\gamma_{m_1, m_2}$ and
then arrange the cycles of $\pi$ so that the cycles do not
cross, assuming that there is at least one cycle of $\pi$
that meets both cycles of $\gamma_{m_1, m_2}$, i.e $\pi \vee
\gamma_{m_1, m_2} = 1_{m_1 + m_2}$.  In \cite{MN}, it was shown
that the non-crossing condition was equivalent to the metric
condition: $|\pi| + |\pi^{-1}\gamma_{m_1, m_2}| =
|\gamma_{m_1, m_2}| + 2$.

\begin{notation}
We let for $r \geq 1$ and $n = m_1 + \cdots + m_r$ and
$\gamma_{m_1, \dots, m_r}$,
$$S_{\NC}(m_1,\dots,m_r)\vcentcolon = S_{NC}(\gamma_{m_1,\dots,m_r}),$$
and,
$$S_{\NC}^{nc}(m_1,\dots,m_r)\vcentcolon = S_{NC}^{nc}(\gamma_{m_1,\dots,m_r}).$$
\end{notation}

Observe that it might be possible than more that one permutation $\pi$ and $\pi^\prime$ correspond to the same partition $0_{\pi}$. In this sense turning a permutation into a partition is well defined however the reverse operation depends on the cycle choice. There is however one case where the cycle choice is unique, that is when all the cycles are of size at most $2$.

\begin{notation}
We let for $r\geq 1$ and $n=m_1+\cdots +m_r$ and $\gamma=\gamma_{m_1,\dots,m_r}$,
\begin{multline*}
\NC_2(m_1, \dots, m_r) \\ = \{ \pi \in S_\NC(m_1, \dots,
m_r) \mid \textrm{\ all cycles of\ } \pi \textrm{\ have 2
  elements\ }\},
\end{multline*}
and,
\begin{multline*}
\NC_2^{nc}(m_1, \dots, m_r) \\ = \{ \pi \in S_\NC^{nc}(m_1, \dots,
m_r) \mid \textrm{\ all cycles of\ } \pi \textrm{\ have 2
  elements\ }\}.
\end{multline*}
A cycle of $\pi\in \NC_2(m_1,\dots,m_r)$ consisting of two elements from distinct cycles of $\gamma$ is called a \textit{through string}.
\end{notation}

\begin{definition}\label{Definition: Restriction of permutation}
For a permutation $\pi\in S_n$ and a subset $M \subset [n]$. Let $\pi|_M$ be the permutation on the set of points $M$ defined by the first return map; that is,
$$\pi|_{M}(m)=\pi^l(m),$$
where $l = \min\{k \geq 1: \pi^k(m)\in M\}.$ It is well known that $\pi^{n!}=id$ and therefore such a minimum exist. Moreover if $m_1,m_2\in M$ are such that $\pi|_{M}(m_1)=\pi|_{M}(m_2)$ then $\pi^{m_1^{min}}(m_1)=\pi^{m_2^{min}}(m_2)$. If $m_1^{min}=m_2^{min}$ then $m_1=m_2$, if $m_1^{min}<m_2^{min}$ then we apply the inverse function, $\pi^{-1}$, $m_1^{min}$ times to get $m_1=\pi^{m_2^{min}-m_1^{min}}(m_2)$, however this means that $\pi^{m_2^{min}-m_1^{min}}(m_2) \in M$ with $0<m_2^{min}-m_1^{min}<m_2^{min}$ which is not possible, so it must be $m_1^{min}=m_2^{min}$ and $m_1=m_2$. This proves that $\pi|_{M}$ is well defined.
\end{definition}

\subsection{Non-crossing partitioned permutations and higher order free cumulants} 

\begin{notation}
Let $\cP(n)$ be the partitions of $[n] = \{1 ,2 ,
\dots,\ab n\}$ and $S_n$ be the permutations of $[n]$.  For
$\cU \in \cP(n)$ we let $\#(\cU)$ denote the number of
blocks of $\cU$ and $|\cU| = n - \#(\cU)$, the
\textit{length} of $\cU$.

For $\pi \in S_n$ recall that $|\pi| = n - \#(\pi)$ and $|\pi|$ is
the minimal number of transpositions needed to write $\pi$
as a product of transpositions. Since multiplying a
permutation by the transposition, $(i, j)$, reduces the
number of cycles by 1, if $i$ and $j$ were in different
cycles of the permutation, and increases it by $1$ if $i$
and $j$ are in the same cycle of the permutation, we see
that the length of a permutation and the length of the
partition obtained from it cycles are the same.

If $\pi \in S_n$ and $\cU \in \cP(n)$ we write $\pi \leq
\cU$ to mean that every cycle of $\pi$ is contained in some
block of $\cU$.  We call the pair $(\cU, \pi)$ a
\textit{partitioned permutation}. We denote by $\PS(n)$ to the set of partitioned permutations. 

We let $|(\cU, \pi)| = 2|\cU| - |\pi|$ and call this the \textit{length} of $(\cU, \pi)$. If we let $p = |\cU| - |\pi| = \#(\pi) - \#(\cU)$, then $p$ is the number of cycles of $\pi$ joined by $\cU$. For example, when $p = 1$, this means that one block of $\cU$ contains two cycles of $\pi$ and all other blocks of $\cU$ contain only one cycle of $\pi$.  If $p = 2$ this means that either one block of $\cU$ contains three cycles of $\pi$ or two blocks of $\cU$ each contain 2 cycles of $\pi$.

\end{notation}

Two crucial properties of the length function are
\cite[Lemma 2.2]{MSS} and \cite[Prop. 5.6]{CMSS}.

\begin{proposition}[Triangle Inequality]
Given partitioned permutations $(\cV, \pi)$ and $(\cU,
\sigma)$ we have
\[
|(\cV \vee \cU, \pi\sigma)| \leq |(\cV, \pi)| + |(\cU, \sigma)|.
\]
\end{proposition}

Given two partitioned permutations $(\pi,\mathcal{V}),(\mathcal{W},\sigma)\in \PS(n)$, their product is defined as,
\begin{align*} \lefteqn{
(\pi,\mathcal{V})\cdot(\mathcal{W},\sigma) }\\
&= \left\{ \begin{array}{lcc}
             (\mathcal{V} \vee \mathcal{W},\pi\sigma) &   \text{if }|(\mathcal{V} \vee \mathcal{W},\pi\sigma)| = |(\mathcal{V},\pi)|+|(\mathcal{W},\sigma)|, \\
             0 &  \mathrm{otherwise} \\
             \end{array}
   \right.
\end{align*}

\begin{proposition}[Equality in the triangle inequality] 
  \label{Proposition: Equality in triangle}
Let $(\cV, \pi)$ and $(\cW, \ab\pi^{-1} \gamma)$ be
partitioned permutations of $[n]$. The equation,
$$(\cV,\pi)\cdot (\cW,\pi^{-1}\gamma) = (\cV\vee\cW,\gamma)$$
is equivalent to the conjunction of the following four conditions:
\begin{enumerate}[label=(\roman*)]

\item
$|\pi| + |\pi^{-1}\gamma| + |\gamma| = 2 |\pi \vee \gamma|$

\item
$|\cV \vee \gamma| - |\pi \vee \gamma | = |\cV| - |\pi|$

\item
$|\cW \vee \pi| - |\pi \vee \gamma| = |\cW| - |\pi^{-1}\gamma|$

\item
$|\cV| - |\pi| + |\cW| - |\pi^{-1}\gamma| =
|\cV \vee \cW| - |\pi \vee \gamma|$

\end{enumerate}
In particular if we have equality in the triangle inequality:
\[
|(\cV, \pi)| + |(\cW, \pi^{-1}\gamma)| = |(\cV \vee \cW,
\gamma)|.
\]
Then $(i),(ii),(iii)$ and $(iv)$ hold.

\end{proposition}

\begin{remark}
Let us recall from \cite[Notation 5.13]{CMSS} what each of
these four conditions means.  First, $(i)$ means $\pi\in S_{\NC}(\gamma)\cup S_{\NC}^{nc}(\gamma)$. Condition $(ii)$ describes the way that $\cV$ can be a union of cycles of $\pi$.

Since $\pi \leq \cV$, we must have that each block of $\cV$
is a union of cycles of $\pi$; however $(ii)$ says that if
$c_1$ and $c_2$ are cycles of $\pi$ and are in the same block of
$\cV$ then these cycles may not meet the same cycle of
$\gamma$, i.e. they lie in different blocks of $\pi \vee
\gamma$. Condition $(iii)$ says that the same holds for the
pair $(\cW, \pi^{-1}\gamma)$. Condition $(iv)$ says that the
sum of the number of joins of the cycles of $\pi$ made by
$\cV$ and the number of joins of the cycles of
$\pi^{-1}\gamma$ made by $\cW$ equals the number of joins of
the blocks of $\pi \vee \gamma$ made by $\cV \vee \cW$.
\end{remark}

For this paper we deal with the set of non-crossing partitioned permutations introduced in \cite{CMSS}. Let us remind the reader  of the definition of non-crossing for partitioned permutations.

\begin{definition}\label{Definition: Non-crossing partitioned permutations}\phantom{I}
\begin{enumerate}
\item For $(\mathcal{U},\gamma)\in \PS(n)$ fixed. We say that $(\mathcal{V},\pi)\in \PS(n)$ is $(\mathcal{U},\gamma)$-non crossing if,
$$(\mathcal{V},\pi)\cdot (0_{\pi^{-1}\gamma},\pi^{-1}\gamma) = (\mathcal{U},\gamma).$$
The set of $(\mathcal{U},\gamma)$-non crossing partitioned permutations will be denote by $\PS_{\NC}(\mathcal{U},\gamma)$.
\item We say that $(\mathcal{V},\pi)\in \PS(n)$ is a \textit{non-crossing partitioned permutation} in the $m_1,\dots,m_r$-annulus if $(\mathcal{V},\pi)$ is $(1_n,\gamma_{m_1,\dots,m_r})$-non crossing with $\gamma_{m_1,\dots,m_r}$ defined as before and $n=m_1+\cdots+m_r$. The set of non-crossing partitioned permutations in the $m_1,\dots,m_r$-annulus will be denote by $\PS_{\NC}(m_1,\dots,m_r)$.
\item We denote by $\PS_{\NC_2}(m_1,\dots,m_r)$ to all the non-crossing partitioned permutations in $\PS_{\NC}(m_1,\dots,m_r)$ such that $\sigma$ is a pairing and by $\PS_{\NC_{2,1}}(m_1,\dots,m_r)$ to all the ones where any cycle of $\sigma$ has size at most $2$.
\end{enumerate}
\end{definition}

\begin{remark}\label{Remark: Equality of four conditions}
Now let us consider what happens in Proposition
\ref{Proposition: Equality in triangle} when $\cW =
0_{\pi^{-1}\gamma}$ and $\cV \vee \gamma = 1_n$. First $(i)$
does not involve the partitions $\cV$ and $\cW$, so nothing changes
here. Next, $(ii)$ becomes $|1_n| - |\pi \vee \gamma| =
|\cV| - |\pi|$. Finally, $(iii)$ and $(iv)$ are tautologies
when $\cW = 0_{\pi^{-1}\gamma}$.
\end{remark}

\begin{remark}\label{Remark and definition: The graph Gamma(V,pi)}
Suppose $(\cV, \pi)$ is a partitioned permutation in $\PS_{\NC}(\ab m_1,\dots,m_r)$ and let $\gamma=\gamma_{m_1,\dots,m_r}$.  Let us create an unoriented bipartite graph, $\Gamma(\cV, \pi, \gamma)$. We let the edge set $E = \{ e_c \mid c$ is a cycle of $\pi\}$ be the cycles of $\pi$ and the vertex set $V$ be the union of the blocks of $\cV$ (the \textit{black} vertices) and the blocks of $\pi \vee \gamma$ (the \textit{white} vertices). Each edge $e_c$ connects the block of $\cV$ containing $c$ to the block of $\pi \vee \gamma$ containing $c$.
\end{remark}

Thanks to Remark \ref{Remark: Equality of four conditions} we get the following result which characterizes the non-crossing partitioned permutations via its graph $\Gamma(\cV,\pi,\gamma)$.

\begin{lemma}\label{Lemma: Tree iff Non-crossing partitioned permutation}
Let $(\cV,\pi)\in \PS(m)$ be a partitioned permutation such that $\pi\in S_{\NC}(m_1,\dots,m_r)\cup S_{\NC}^{nc}(m_1,\dots,m_r)$. Then,
$$(\cV,\pi)\in \PS_{\NC}(m_1,\dots,m_r)$$ 
if and only if the graph $\Gamma(\cV, \pi,\gamma)$ is a tree.
\end{lemma}
\begin{proof}
Observe that $\cV\vee\gamma=1_n$ if and only if $\Gamma(\cV,\pi,\gamma)$ is connected. On the other hand the condition $(ii)$ of Proposition \ref{Proposition: Equality in triangle} is already satisfied as $\pi\in S_{\NC}(m_1,\dots,m_r)\cup S_{\NC}^{nc}(m_1,\dots,m_r)$. So, thanks to Proposition \ref{Proposition: Equality in triangle} we just have to check that $|V|-|E|=1$ if and only if $(ii)$ of Proposition \ref{Proposition: Equality in triangle} holds. This follows immediately as,
$$|V|-|E|=\#(\pi\vee\gamma)+\#(\cV)-\#(\pi)=n-|\pi\vee\gamma|-|\cV|+|\pi|.$$
\end{proof}

Let us finish this section with the introduction of the higher order free cumulants defined in \cite{CMSS}.

\begin{definition}\label{Definition: Higher order free cumulants}
For each $r\in\mathbb{N}$ let $(\alpha_{m_1,\dots,m_r})_{m_1,\dots,m_r=1}^{\infty}$ be a sequence indexed by $r$ subscripts. We call this a moment sequence of order $r$. Given the moment sequences of orders at most $r$:
$$(\alpha_m)_{m=1}^{\infty},(\alpha_{m_1,m_2})_{m_1,m_2=1}^{\infty},\dots, (\alpha_{m_1,\dots,m_r})_{m_1,\dots,m_r=1}^{\infty}$$ 
we define the free cumulants sequences 
$$(\K_m)_{m=1}^{\infty},(\K_{m_1,m_2})_{m_1,m_2=1}^{\infty},\dots, (\K_{m_1,\dots,m_r})_{m_1,\dots,m_r=1}^{\infty}$$ 
associated to these moment sequences with the recursive equations:
\begin{equation}\label{Equation: Higher order free cumulants}
\alpha_{m_1,\dots,m_t}=\sum_{(\mathcal{U},\pi)\in \PS_{NC}(m_1,\dots,m_t)}\K_{(\mathcal{U},\pi)}
\end{equation}
for $t=1,\dots,r$. Where $\K_{(\mathcal{U},\pi)}$ is called the \textit{multiplicative extension } of the cumulants sequence and it is defined as follows:
$$\K_{(\mathcal{U},\pi)}=\prod_{\substack{B\text{ block of }\mathcal{U} \\ V_1,\dots,V_i \text{ cycles of } \pi \text{ with }V_i\subset B}}\K_{|V_1|,\dots,|V_i|}$$
The numbers $\K_{m_1,\dots,m_r}$ are called the \textit{free cumulants of order} $r$ and the sequence $(\K_{m_1,\dots,m_r})_{m_1,\dots,m_r=1}^{\infty}$ is called the \textit{free cumulant sequence of order} $r$.
\end{definition}

When $r=1$ we recover the free cumulants as defined in Equation \ref{Equation: First order free cumulants}. There is also a multivariate version of the higher order free cumulants analogous to its first order case defined in Equation \ref{Equation: First order free cumulants multivariate}, however, for the purpose of this work we will not go in further. We rather refer to \cite[\S 7]{CMSS}.

\section{Graph theory}\label{Section: Graph theory}

The main tools to prove our results rely on graph theory. So for this section we will introduce all the concepts required throughout the paper. By a unoriented graph or simply graph we mean a pair $(V,E)$ where $V$ is a non-empty set (the vertices) and $E$ is a subset of $V\times V$ where if $(a,b)\in E$ then we say that the edge $(a,b)$  connects the vertices $a$ and $b$. In a graph we do not put an orientation on the edges, so we will write $\{a,b\}$ instead of $(a,b)$. When the pair $\{a,b\}$ has an order then we call $(V,E)$ an \textit{oriented graph}, for which we now adopt the notation $(a,b)$. In this case $(a,b)\in E$ does not imply $(b,a)\in E$. In our oriented graphs we will allow multiple edges and loops, so, the same pair of vertices might have more than one edge going from one vertex to the other and vice versa. Thus given two vertices $a,b$ we may have more than one edge of the form $(a,b)$. To distinguish in between two edges joining the same pair of vertices we consider a label $\delta: E\rightarrow \mathbb{N}$ which is an injective function that assigns a number to each edges. The edge $e_u$ denotes the edge whose label is $u$. The image of this label $\delta(E)$ will be denoted by $E_{\delta}$. When there are no multiple edges we call the graph \textit{simple}.

\begin{definition}\label{Definition: Quotient graph}
Let $G=(V,E,\delta)$ be a simple oriented graph and let $\pi\in \cP(V)$ be a partition of $V$. The \textit{quotient graph} $G^{\pi}=(V^{\pi},E^{\pi},\delta^\pi)$ is the labelled oriented graph obtained after identifying vertices in the same block. For $u\in V$, $[u]_{\pi}$ denotes the block of $\pi$ that contains $u$, the vertices of $G^{\pi}$ consist of the blocks of $\pi$. If $e=(u,v)$ is an edge of $G$ then $([u]_{\pi},[v]_{\pi})$ will be the corresponding edge of $G^{\pi}$ which we will denote by $e^{\pi}$. The label of the quotient graph is defined by $\delta^{\pi}(e^{\pi})\vcentcolon=\delta(e)$. In $G^{\pi}$ we may have multiple edges; if $e_1=([u_1]_\pi,[v_1]_\pi)$ and $e_2=([u_2]_\pi,[v_2]_\pi)$ are such that $[u_1]_\pi=[u_2]_\pi$ and $[v_1]_\pi=[v_2]_\pi$ or $[u_1]_\pi=[v_2]_\pi$ and $[u_2]_\pi=[v_1]_\pi$ then we say that $e_1$ and $e_2$ connect the same pair of vertices. The partition $\overline{\pi} \in \cP(E_\delta)$ is the partition where $i$ and $j$ are in the same block if and only if $e_i^{\pi}$ and $e_j^{\pi}$ connect the same pair of vertices. The labels $\delta^\pi$ and $\delta$ are essentially the same so to keep it simple we just write $G^{\pi}=(V^{\pi},E^{\pi},\delta)$. See Figure \ref{Figure: Big example} part A) for an example of a quotient graph.
\end{definition}

\begin{notation}\label{Notation: Elementarization}
Let $G=(V,E)$ be an oriented graph. The \textit{elementarization} of $G$ denoted by $\overline{G}$ is the simple unoriented graph obtained after identifying edges of $G$ connecting the same pair of vertices without considering the orientation of the edges. The vertices of $\overline{G}$ remains the same while the edges of $G$ consist of equivalence classes of $E$ where if $e\in E$ connects $u$ and $v$ (regardless of $e=(u,v)$ or $e=(v,u)$) then $\overline{e}$ is the edge of $\overline{G}$ that consists of all edges connecting $u$ and $v$. We will write $\overline{G}=(V,\overline{E})$. See Figure \ref{Figure: Big example} part B) for an example of the elementarization of a graph.
\end{notation}

\begin{remark}
If $G=(V,E,\delta)$ is a simple oriented graph and $\pi\in \cP(V)$ then $\overline{G^{\pi}}=(V^{\pi},\overline{E^{\pi}},\delta)$ is a graph whose vertices can be identified by the blocks of $\pi$ and whose edges can be identified by the blocks of $\overline{\pi}$. More precisely, $B$ is a block of $\overline{\pi}$ if and only if $\{e_b:b\in B\}$ is an edge of $\overline{G^{\pi}}$.
\end{remark}

\subsection{Quotient graphs by identification of the edges} Observe that a quotient graph is obtained after identifying vertices, we would like to do something similar via identification of the edges. Since we are working with oriented graphs then the identification of the edges must indicate in which orientation we will identify two edges. This motivates the following definition.

\begin{definition}\label{Definition: Identification of edges}
Let $G=(V,E,\delta)$ be an oriented labelled graph and let $\overline{G}=(V,\overline{E})$ be its elementarization. Let $\overline{e_1},\dots,\overline{e_n}$ be a collection of distinct edges of $\overline{E}$. We write $\overline{e_i}=\{u_i,v_i\}$ for $i=1,\dots,n$. 
\begin{enumerate}
\item An \textit{identification of the edges} $\overline{e_1},\dots,\overline{e_n}$ is an equivalence relation $\sim_{\overline{e_1},\dots,\overline{e_n}}$ on $A=\{u_1,\dots,u_n,v_1,\dots,v_n\}$ of the following form. If some of the edges $\overline{e_i}$ is a loop then $a \sim_{\overline{e_1},\dots,\overline{e_n}}b$ for any $a,b\in A$. Otherwise we consider $A_1$ and $A_2$ to be ordered sets of the form,
$$A_1=\{x_1,\dots,x_n\},$$
$$A_2=\{y_1,\dots,y_n\},$$
where for each $i$ either $x_i=u_i$ and $y_i=v_i$ or $x_i=v_i$ and $y_i=u_i$. We then define the equivalence relation $\sim_{\overline{e_1},\dots,\overline{e_n}}$ be given by $a\sim_{\overline{e_1},\dots,\overline{e_n}} b$ for any $a,b\in A_1$ and $a\sim_{\overline{e_1},\dots,\overline{e_n}} b$ for any $a,b\in A_2$ whenever $A_1\cap A_2=\emptyset$, or, $a\sim_{\overline{e_1},\dots,\overline{e_n}} b$ for any $a,b\in A$ whenever $A_1\cap A_2\neq \emptyset$.
\item For $\sim_{\overline{e_1},\dots,\overline{e_n}}$ fixed and $\overline{e_{i_1}},\cdots,\overline{e_{i_k}}$ a subset of $\overline{e_1},\dots,\overline{e_n}$ we define the restriction of $\sim_{\overline{e_1},\dots,\overline{e_n}}$ to $\overline{e_{i_1}},\cdots,\overline{e_{i_k}}$ as the equivalence relation on $D=\{u_{i_1},\dots,u_{i_k},v_{i_1},\dots,v_{i_k}\}$ given by the restriction of $\sim_{\overline{e_1},\dots,\overline{e_n}}$ to $D$.
\item An oriented partition of the edges is a pair $(\tau,\sim)$ where $\tau\in \cP(\overline{E})$ is a partition of the edges and $\sim=\{\sim_B : B\in\tau\}$ is a collection of equivalence relations for each block of $\tau$. Here if $B=\{\overline{e_1},\dots,\overline{e_n}\}$ is a block of $\tau$ then $\sim_B \vcentcolon = \sim_{\overline{e_1},\dots,\overline{e_n}}$ is an equivalence relation as before.
\item For two oriented partition of the edges $(\tau,\sim)$ and $(\tau^\prime,\sim^\prime)$ we say that $(\tau^\prime,\sim^\prime)\leq (\tau,\sim)$ if $\tau^\prime \leq \tau$ and for any block $B^\prime$ of $\tau^\prime$ that is contained in the block $B$ of $\tau$ we have that $\sim^\prime_{B^\prime}$ is the restriction of $\sim_B$ to $B^{\prime}$.
\end{enumerate} 
When a single edge $\overline{e_1}=(u,v)$ is a block of $\tau$ then $\sim_{\overline{e_1}}$ is a trivial relation in the sense that $u\sim_{\overline{e_1}} v$ if and only if $u=v$. An example of an oriented partition of the edges can be seen in Example \ref{Example: Oriented partition of the edges}.
\end{definition}

\begin{remark}
The objective of Definition \ref{Definition: Identification of edges} is to identify the edges $e_1,\dots,e_n$ of the graph $G$. For this reason we do not consider edges that are already connecting the same of vertices. The equivalence relation $\sim_{\overline{e_1},\dots,\overline{e_n}}$ will tell us whether $e_i=(u_i,v_i)$ and $e_j=(u_j,v_j)$ will be identified by merging $u_i$ with $u_j$ and $v_i$ with $v_j$ or the other way around.
\end{remark}

Given an oriented partition of the edges $(\tau,\sim)$ we can define a quotient graph where edges in the same block of $\tau$ are identified according to the orientation given by $\sim$. The formal definition is given right away.

\begin{definition}\label{Definition: First definition of Quotient along edges}
Let $G=(V,E,\delta)$ be a labelled oriented graph and let $\overline{G}=(V,\overline{E})$ be its elementarization. Let $(\tau,\sim)$ be an oriented partition of the edges. We let $\underset{\sim}{\tau}\in \cP(V)$ be the partition given by $u$ and $v$ are in the same block of $\underset{\sim}{\tau}$ if $u \sim_{B} v$ for some $B\in \tau$. The quotient graph $G^{\underset{\sim}{\tau}}$ will be called a \textit{a quotient graph along edges}. See Example \ref{Example: Oriented partition of the edges} for a quotient graph along the edges.
\end{definition}

\begin{example}\label{Example: Oriented partition of the edges}
Let us consider the labelled oriented graph $G=(V,E,\ab \delta)$ with vertices,
$$V=\{v_1,v_2,v_3,v_4,v_5,v_6\},$$
and edges,
$$E=\{e_1,e_2,e_3,e_4,e_5,e_6,e_7,e_8\},$$
given by $e_1=(v_3,v_1),e_2=(v_3,v_2),e_3=(v_3,v_2),e_4=(v_2,v_3),e_5=(v_3,v_4),e_6=(v_4,v_4),e_7=(v_5,v_6)$ and $e_8=(v_6,v_5)$, see Figure \ref{Figure: Oriented graph} part a). The elementarization of this graph has edges,
$\overline{e_1}=\{e_1\}, \overline{e_2}=\{e_2,e_4,e_3\}, \overline{e_5}=\{e_6\},\overline{e_6}=\{e_6\}$ and $\overline{e_7}=\{e_7,e_8\}$. Let us consider the partition $\tau$ of the edges $\overline{E}$ with blocks,
$$\{\overline{e_1}\},\{\overline{e_2},\overline{e_5},\overline{e_7}\},\{\overline{e_6}\}.$$
Let us consider two equivalence relations $\sim_1$ and $\sim_2$ as follows. Let us write, $\overline{e_2}=\{v_2,v_3\},\overline{e_5}=\{v_3,v_4\}$ and $\overline{e_7}=\{v_5,v_6\}$. Then for the block $\{\overline{e_2},\overline{e_5},\overline{e_7}\}$ of $\tau$ we let $\sim_1$ to be given by the sets,
$$A=\{v_2,v_3,v_5\}\text{ and }B=\{v_3,v_4,v_6\}.$$
In this case according to the definition we have that $a \sim_1 b$ for any $a,b\in \{v_2,v_3,v_4,v_5,v_6\}$ and therefore $\underset{\sim_1}{\tau}=\{v_1\},\{v_2,v_3,v_4,v_5,v_6\}$. The quotient graph along the edges of this oriented partition of the edges can be seen in Figure \ref{Figure: Oriented graph} part b). On the other hand, for $\sim_2$ we let,
$$A=\{v_2,v_4,v_5\}\text{ and }B=\{v_3,v_3,v_6\}.$$
Thus $a\sim_1 b$ for any $a,b\in A$ and $a\sim_1 b$ for any $a,b\in B$ and therefore $\underset{\sim_2}{\tau}=\{v_1\},\{v_2,v_4,v_5\},\{v_3,v_6\}$. The quotient graph along the edges of this oriented partition of the edges can be seen in Figure \ref{Figure: Oriented graph} part b).
\end{example}

\begin{figure}
    \centering
    \includegraphics[width=0.9\textwidth]{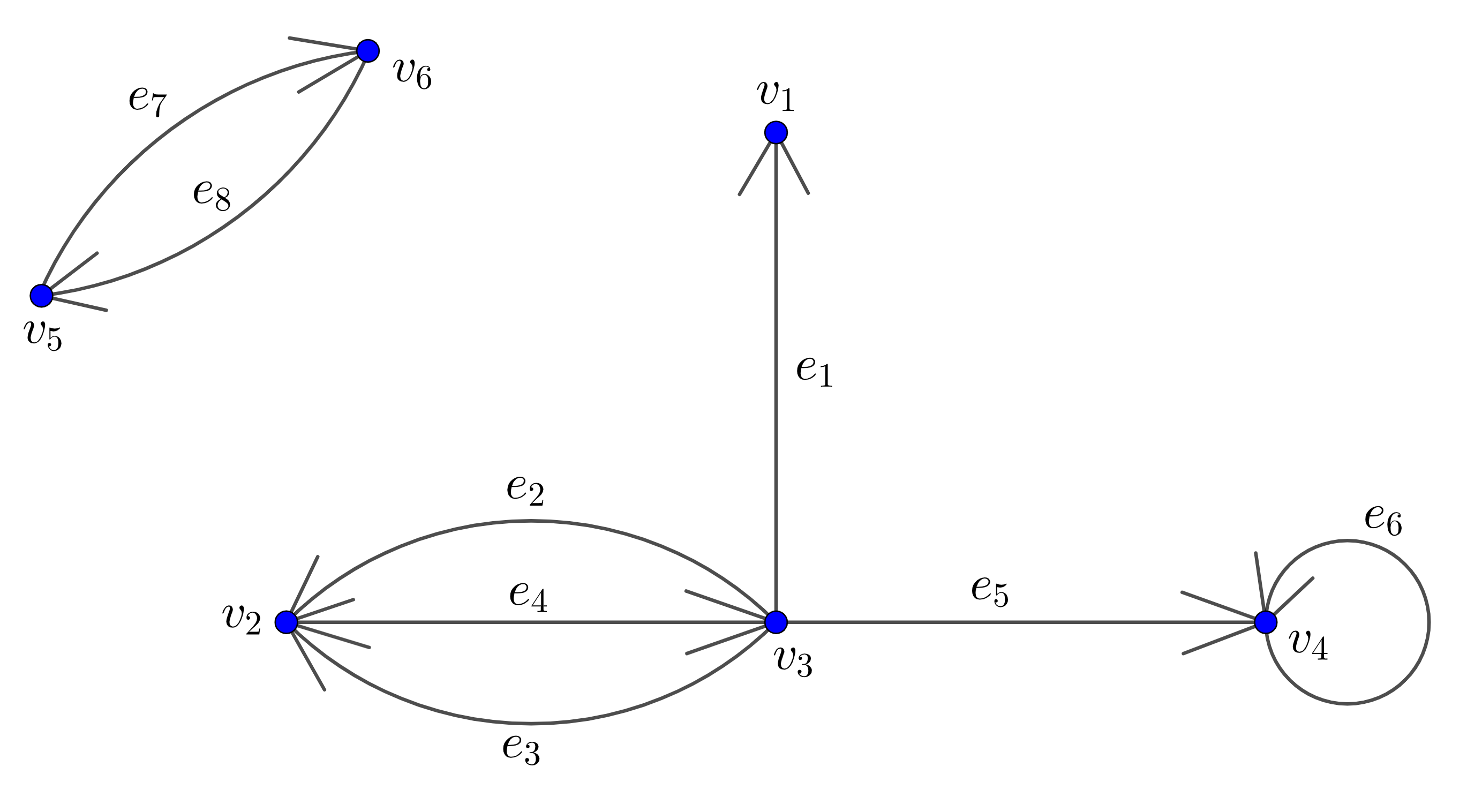}
    \text{a) The oriented graph of Example \ref{Example: Oriented partition of the edges}.}\\
    \includegraphics[width=0.45\textwidth]{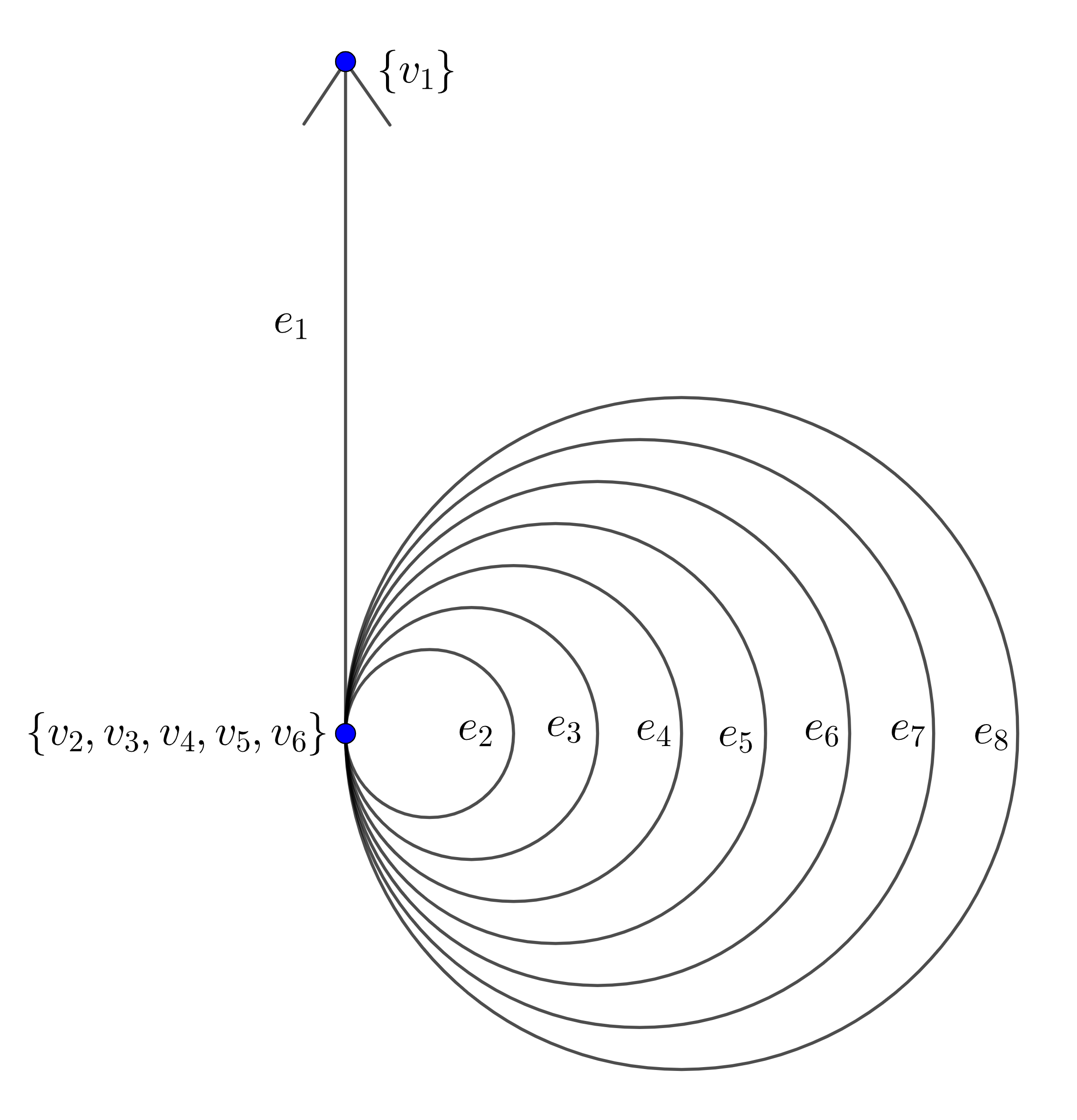}
    \includegraphics[width=0.45\textwidth]{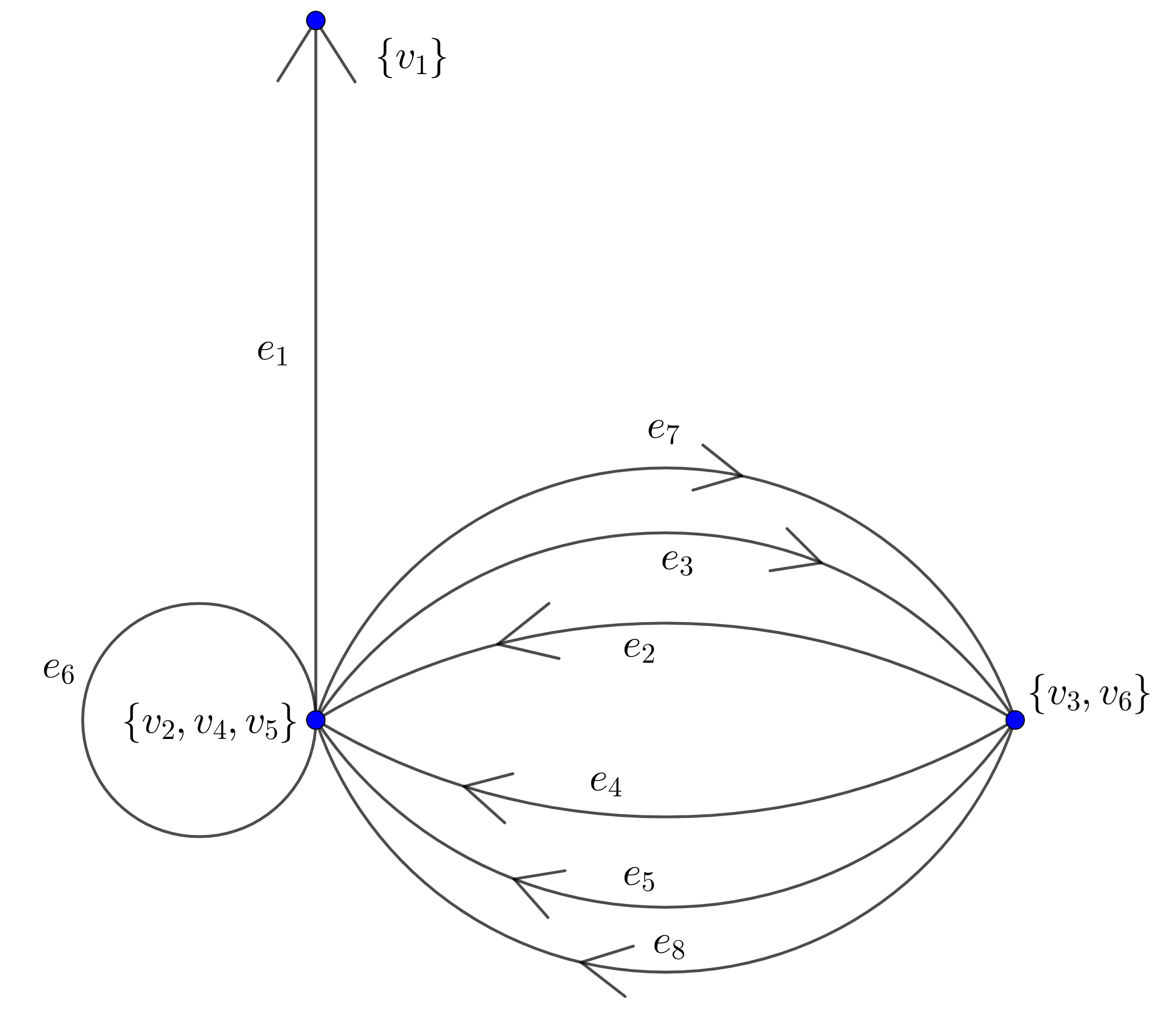}
    \text{b) Left the quotient graph along the edges corresponding to $(\tau,\sim_1)$.} 
    \text{Right The quotient graph along the edges corresponding to $(\tau,\sim_2)$.}\\
    \caption{The oriented graphs of Example \ref{Example: Oriented partition of the edges}.}
    \label{Figure: Oriented graph}
\end{figure}

\begin{definition}\label{Definition: Graphs with partitions on the edges without multiplicity}
Let $G=(V,E,\delta)$ be a labelled oriented graph and let $\overline{G}=(V,\overline{E})$ be its elementarization. Let $\tau\in \cP(\overline{E})$ be a partition of the edges. We define the graph $G(\tau)$ to be the bipartite graph with,
\begin{enumerate}
    \item set of white vertices given by the connected components of $G$,
    \item set of black vertices given by blocks of $\tau$, and,
    \item set of edges indexed by $\overline{E}$, where the edge $\overline{e}$ connects a white vertex $G_1$ with a black vertex $B$ if and only if $\overline{e}\in B$ and $\overline{e}$ is an edge of $\overline{G_1}$. An example can be seen in Figure \ref{Figure: The graph G(tau)}
\end{enumerate}
\end{definition}

\begin{figure}
    \centering
    \includegraphics[width=\textwidth]{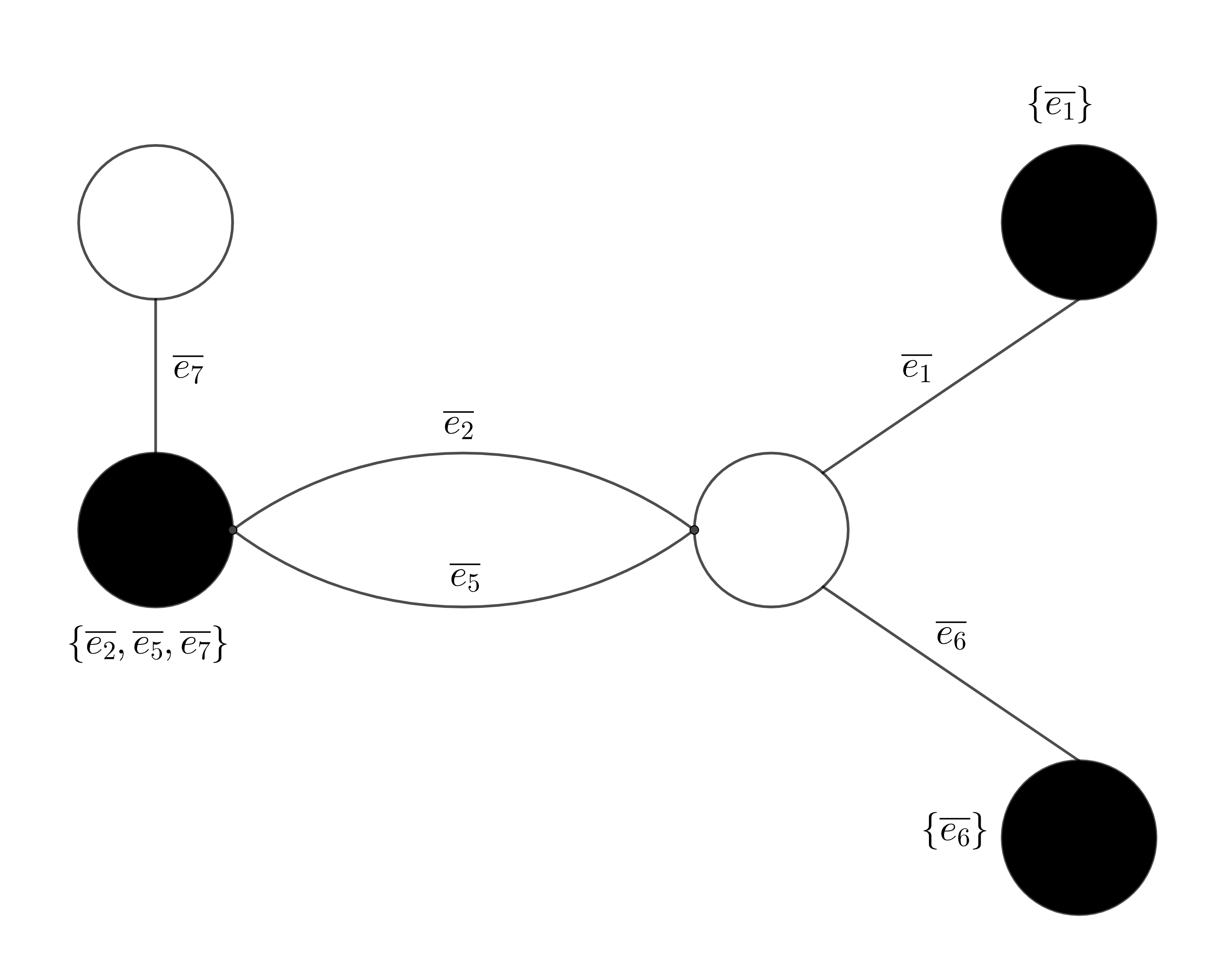}
    \caption{The graph $G(\tau)$ corresponding to $G=(V,E)$ and $\tau$ given as in Example \ref{Example: Oriented partition of the edges}.}
    \label{Figure: The graph G(tau)}
\end{figure}

\begin{lemma}\label{Lemma: Underline partition respects order}
Let $G=(V,E,\delta)$ be a labelled oriented graph and let $\overline{G}=(V,\overline{E})$ be its elementarization. Let $(\tau_1,\sim^1)$ and $(\tau_2,\sim^2)$ be two oriented partition of the edges and such that $(\tau_1,\sim^1)\leq (\tau_2,\sim^2)$. Then
$$\underset{\sim^1}{\tau_1}\leq \underset{\sim^2}{\tau_2}.$$
\end{lemma}
\begin{proof}
Let $u,v\in V$ be such that are in the same block of $\underset{\sim^1}{\tau_1}$. Then there exist $u_2,\dots,u_n \in V$ and blocks $B_1,\dots,B_n$ of $\tau_1$ such that,
$$u =u_1 \sim^1_{B_1} u_2 \sim^1_{B_2} u_3 \cdots u_n \sim^1_{B_n} u_{n+1}=v.$$
As $u_1 \sim^1_{B_1} u_{2}$ there exist edges $\overline{e_1}=\{u_1,v_1\}$ and $\overline{e_2}=\{u_2,v_2\}$ of $\overline{E}$
such that $\overline{e_1}$ and $\overline{e_2}$ are in the same block $B_1\in \tau_1$. Since $\tau_1\leq \tau_2$ then $B_1$ is contained in a block of $\tau_2$, say, $D_1$. Therefore $\overline{e_1}$ and $\overline{e_2}$ are also in the same block $D_1$ of $\tau_2$. Hence, as $(\tau_1,\sim^1)\leq (\tau_2,\sim^2)$ then $u_1 \sim^2_{D_1} u_2$. Similarly we prove that for any $i$, $u_i \sim^2_{D_i} u_{i+1}$ for some block $D_i$ of $\tau_2$. Thus $u=u_1$ and $v=u_{n+1}$ are in the same block of $\underset{\sim^2}{\tau_2}$.
\end{proof}

\begin{lemma}\label{Lemma: Equality when tree}
Let $G=(V,E,\delta)$ be a labelled oriented graph and let $\overline{G}=(V,\overline{E})$ be its elementarization. Let $(\tau,\sim)$ be an oriented partition of the edges such that if $\overline{e}$ is a loop of $\overline{G}$ then $\{\overline{e}\}$ is a singleton of $\tau$. If $G(\tau)$ is a tree then the following satisfy.
\begin{enumerate}[label=(\roman*)]
    \item $\#(\underset{\sim}{\tau})=|V|-2(\#(\text{connected components of }G)-1)$.
    \item For any $e_i,e_j\in E$ no loops of $G$; $\overline{e_i}$ and $\overline{e_j}$ are in the same block of $\tau$ if and only if $e_i^{\underset{\sim}{\tau}}$ and $e_j^{\underset{\sim}{\tau}}$ connect the same pair of vertices of $G^{\underset{\sim}{\tau}}$.
    \item No block of $\underset{\sim}{\tau}$ contains two vertices of $G$ in the same connected component of $G$.
\end{enumerate}
\end{lemma}
\begin{proof}
Before getting the proof observe that $(ii)$ is trivially satisfied if $e_i$ and $e_j$ are such that $\overline{e_i}=\overline{e_j}$ so we may assume always that $\overline{e_i}\neq\overline{e_j}$.

We will prove all claims by induction on the number of connected components of $G$. If $G$ is connected then any block of $\tau$ must be of size $1$ otherwise there will be a black vertex with more than one edge connected to the same white vertex of $G(\tau)$ which is impossible as $G(\tau)$ is a tree. If any block of $\tau$ has size $1$ then $\sim_B$ is a trivial relation for any $B\in \tau$ so $\underset{\sim}{\tau}$ connects no vertices of $V$, i.e $\underset{\sim}{\tau}=1_V$. Thus,
$$\#(\underset{\sim}{\tau})=|V|=|V|-2(1-1),$$
which proves $i$. $(ii)$ is satisfied automatically as there are no blocks of $\tau$ of size larger than $1$. Similarly $(iii)$ is satisfied since $\underset{\sim}{\tau}=1_V$. Now we suppose all are true if $G$ has $s$ connected components and we prove it for $G$ having $s+1$ connected components. Let $G_1,\dots,G_{s+1}$ be the connected components of $G$. Let $pruned(G)$ be the graph $G(\tau)$ after we remove all black vertices that are leaves and their adjacent edges. Since $G(\tau)$ is a tree so is $pruned(G)$ and therefore it must have a leaf that correspond to a white vertex, say $G_{s+1}$. The vertex $G_{s+1}$ is connected to the the rest of the graph $pruned(G)$ through one edge of $\overline{G}$, let us denoted that edge by $\overline{e_{s+1}}$ which must also be an edge of $\overline{G_{s+1}}$. Let $B$ be the block of $\tau$ that contains $\overline{e_{s+1}}$. Let $G^\prime = G_1\cup\cdots \cup G_s=(V^\prime,E^\prime)$ and $\tau^{\prime}$ to be $\tau$ restricted to the set of edges of $\overline{G^{\prime}}$. Observe that when we removed the leaves corresponding to black vertices then we removed all blocks of $\tau$ of size $1$, so since $G_{s+1}$ was a leaf after removing these blocks that means that $B$ is the only block of $\tau$ that contains one edge from $\overline{G_{s+1}}$ and that has size larger than $1$. This means that any block of $\tau^{\prime}$ is a block of $\tau$ except by the block $B\setminus \{\overline{e_{s+1}}\}$. We let $(\tau^\prime,\sim^{^\prime})$ be the oriented partition of the edges of the graph $\overline{G^\prime}$ given by $\sim^{\prime}_B=\sim_B$ if $B$ is a block of $\tau^\prime$ that is also a block of $\tau$ and for the block $B^\prime=B\setminus \{\overline{e_{s+1}}\}$ we let $\sim^{\prime}_{B^\prime}$ to be the equivalence relation $\sim_B$ restricted to $B^\prime$

Note that the graph $G^\prime(\tau^\prime)$ consist of the graph $G(\tau)$ after we take away the white vertex $G_{s+1}$ and all its adjacent edges and vertices except by the vertex $B$ which becomes the vertex $B\setminus \{\overline{e_{s+1}}\}$ in $G^\prime(\tau^\prime)$. The latter means that $G^{\prime}(\tau^\prime)$ is a tree where $G^\prime$ has $s$ connected components, so by induction hypothesis,
$$\#(\underset{\sim^\prime}{\tau^\prime})=|V^\prime|-2(s-1).$$
Also, for any $e_i,e_j\in \overline{E^\prime}$; $\overline{e_i}$ and $\overline{e_j}$ are in the same block of $\tau^\prime$ if and only if $e_i^{\underset{\sim^\prime}{\tau^\prime}}$ and $e_i^{\underset{\sim^\prime}{\tau^\prime}}$ connect the same pair of vertices of ${G^\prime}^{\underset{\sim^\prime}{\tau^\prime}}$. And no block of $\underset{\sim^\prime}{\tau^\prime}$ contains two vertices of $G^\prime$ in the same connected component of $G^\prime$.

By the mentioned before the blocks of $\tau$ are conformed by the block $B$, the singletons $\{\overline{e}\}$ where $\overline{e}$ is an edge of $\overline{G_{s+1}}$ distinct from $\overline{e_{s+1}}$ and the blocks of $\tau^{\prime}$ distinct of the block $B^\prime$. By definition the equivalence relations $\sim$ and $\sim^{\prime}$ are all the same except by $\sim_B$ and $\sim^{\prime}_{B^\prime}$. Let $\{\overline{e_1},\dots,\overline{e_n}\}$ be the block $B^\prime$ so that $\{\overline{e_1},\dots,\overline{e_n},\overline{e_{s+1}}\}$ is the block $B$. We write $\overline{e_i}=\{u_i,v_i\}$ for $i=1,\dots,n,s+1$. By hypothesis $u_i\neq v_i$, assume without loss of generality that $\sim_B$ is determined by the ordered sets,
$$A_1=\{u_1,\dots,u_n,u_{s+1}\},$$
$$A_2=\{v_1,\dots,v_n,v_{s+1}\},$$
where $A_1$ and $A_2$ are the sets defined as in Definition \ref{Definition: Identification of edges}. Therefore $\sim^{\prime}_{B^\prime}$ is simply determined by the ordered sets,
$$A_1^\prime=\{u_1,\dots,u_n\},$$
$$A_2^\prime=\{v_1,\dots,v_n\}.$$
This means that $\underset{\sim}{\tau}$ and $\underset{\sim^\prime}{\tau^\prime}$ identify exactly the same vertices except that in $\underset{\sim}{\tau}$ the vertices $u_{s+1}$ and $u_1$ are in the same block while in $\underset{\sim^\prime}{\tau^\prime}$ are not. The same applies to $v_{s+1}$ and $v_1$. Any other vertex of $G_{s+1}$ is a singleton of $\underset{\sim}{\tau}$ as any other edge of $\overline{G_{s+1}}$ distinct of $\overline{e_{s+1}}$ is a singleton of $\tau$. Thus,
\begin{eqnarray*}
\#(\underset{\sim}{\tau}) &=& \#(\underset{\sim^\prime}{\tau^\prime})+\#(\text{Vertices of }G_{s+1})-2 \\
&=& |V^\prime|-2s+2+\#(\text{Vertices of }G_{s+1})-2 \\
&=& |V|-2s.
\end{eqnarray*}
that completes the induction proof of $(i)$. 

Observe that the partition $\underset{\sim}{\tau}$ is just conformed by the singletons $\{v\}$ where $v$ is a vertex of $G_{s+1}$ distinct from $u_{s+1},v_{s+1}$, the blocks $A_1$ and $A_2$ and the blocks of $\underset{\sim^\prime}{\tau^\prime}$ distinct from $A_1^\prime$ and $A_2^\prime$. The latest means that the unique vertices of $G_{s+1}$ that are joined to other vertices of $G^{\prime}$ in the quotient graph $G^{\underset{\sim}{\tau}}$ are $u_{s+1}$ and $v_{s+1}$ which are joined to $A_1^\prime$ and $A_2^\prime$ respectively. This means that only the edge $\overline{e_{s+1}}$ of $\overline{G_{s+1}}$ is identified with the other edges $\overline{e_1},\dots,\overline{e_n}$ in the quotient graph $G^{\underset{\sim}{\tau}}$ which proves $(ii)$. For $(iii)$ we know by induction hypothesis that the sets $A_1^\prime$ and $A_2^\prime$ are disjoint otherwise $\underset{\sim^\prime}{\tau^\prime}$ contains $A_1^\prime \cup A_2^\prime$ in a block which means it identifies vertices of $G^{\prime}$ in the same connected component. Finally, since only the vertices $u_{s+1}$ and $v_{s+1}$ of $G_{s+1}$ are joined to other vertices in $\underset{\sim}{\tau}$ and $u_{s+1}$ and $v_{s+1}$ are joined to $A_1^\prime$ and $A_2^\prime$ respectively then $u_{s+1}$ and $v_{s+1}$ are not joined which proves the desired.
\end{proof}

\begin{lemma}\label{Lemma: Equality only if partitions are the same}
Let $G=(V,E,\delta)$ be a labelled oriented graph and let $\overline{G}=(V,\overline{E})$ be its elementarization. Let $(\tau_1,\sim^1)$ and $(\tau_2,\sim^2)$ be two oriented partitions of the edges such that $(\tau_1,\sim^1)\leq (\tau_2,\sim^2)$. Assume $(\tau_1,\sim^1)$ is such that if $\overline{e}$ is a loop of $\overline{G}$ then $\{\overline{e}\}$ is a singleton of $\tau_1$. Also assume that $G(\tau_1)$ is a tree. Then
$$\#(\underset{\sim^2}{\tau_2})\leq \#(\underset{\sim^1}{\tau_1}),$$
with equality if and only if $\tau_1=\tau_2$.
\end{lemma}
\begin{proof}
By Lemma \ref{Lemma: Underline partition respects order} it suffices to show that if $\tau_1 \neq \tau_2$ then $\#(\underset{\sim^2}{\tau_2})< \#(\underset{\sim^1}{\tau_1})$. Since $\tau_1 \neq \tau_2$ then there exist distinct blocks of $\tau_1$; $V_1$ and $V_2$ such that $V_1\cup V_2$ is contained in a block of $\tau_2$. Let $\overline{e_1}$ and $\overline{e_2}$ so that $\overline{e_1}\in V_1$ and $\overline{e_2}\in V_2$. Let $(\tau^\prime,\sim^\prime)$ be the ordered partition of the edges given by:
\begin{enumerate}
    \item Any block of $\tau_1$ is a block of $\tau^\prime$ except $V_1$ and $V_2$ for which we made $V_1\cup V_2$ into a single block of $\tau^\prime$. In this way $\tau^\prime \leq \tau_2$.
    \item Each block $V^\prime$ of $\tau^\prime$ is contained in a block $V$ of $\tau_2$ so we let $\sim^\prime_{V^\prime}$ to be the restriction of $\sim^2_V$ to $V^\prime$. 
\end{enumerate}
Thus we have
$$(\tau_1,\sim^1)\leq(\tau^\prime,\sim^\prime)\leq (\tau_2,\sim^2).$$
By Lemmas \ref{Lemma: Underline partition respects order} we have $$\underset{\sim^1}{\tau_1}\leq \underset{\sim^\prime}{\tau^\prime} \leq \underset{\sim^1}{\tau_2},$$ 
and,
$$\#(\underset{\sim^2}{\tau_2}) \leq \#(\underset{\sim^\prime}{\tau^\prime}) \leq \#(\underset{\sim^1}{\tau_1}).$$
Write $\overline{e_1}=\{u_1,v_1\}$ and $\overline{e_2}=\{u_2,v_2\}$. By Lemma \ref{Lemma: Equality when tree} part $(ii)$ we know that $e_1$ and $e_2$ connect distinct pair of vertices of $G^{\underset{\sim^1}{\tau_1}}$ however by the same definition of $\tau^\prime$ we have have that $\underset{\sim^\prime}{\tau^\prime}$ connects $u_1$ to $u_2$ and $v_1$ to $v_2$ or the other way around. In either case we have that $e_1$ and $e_2$ connect the same pair of vertices of $G^{\underset{\sim^\prime}{\tau^\prime}}$ which forces $\underset{\sim^1}{\tau_1}<\underset{\sim^\prime}{\tau^\prime}$ and $\#(\underset{\tau^\prime}{\tau^\prime})<\#(\underset{\sim^1}{\tau_1})$.
\end{proof}

\begin{lemma}\label{Lemma: Inequality with equality iff tree}
Let $G=(V,E,\delta)$ be a labelled oriented graph and let $\overline{G}=(V,\overline{E})$ be its elementarization. Let $(\tau,\sim)$ be an oriented partition of the edges such that if $\overline{e}$ is a loop of $\overline{G}$ then $\{\overline{e}\}$ is a singleton of $\tau$. Suppose $G(\tau)$ is connected. Then
$$\#(\underset{\sim}{\tau})\leq |V|-2(\#(\text{connected components of }G)-1),$$
with equality if and only if $G(\tau)$ is a tree.
\end{lemma}
\begin{proof}
Since $G(\tau)$ is connected it has a spanning tree $T$. For any block, $B$, of $\tau$ which is a black vertex of $T$, we define $B^\prime$ by removing the edges $\overline{e}$ of $B$ that correspond to edges of $G(\tau)$ that were removed from $G(\tau)$ in $T$. Let $G^\prime$ be the graph that has white vertices the connected components of $G$ and black vertices the blocks $B^{\prime}$, and edges connecting the connected component $G_1$ to $B^\prime$ if and only if $T$ has an edge connecting $G_1$ to $B$. The graph $G^\prime$ is isomorphic to $T$ in the sense that it is the same graph except that the black vertices $B$ of $T$ becomes $B^\prime$. Thus $G^\prime$ is a tree. For any edge ,$\overline{e}$, of $G(\tau)$ that was removed in $T$ we add to the graph $G^\prime$ the extra leaf consisting on the vertex $\{\overline{e}\}$ and the edge connecting the connected component of $G$ that contains $\overline{e}$ to $\{\overline{e}\}$. The resulting graph has then set of black vertices the blocks $B^\prime$ and singletons $\{\overline{e}\}$ where $\overline{e}$ is such that was removed from $B$ to get $B^\prime$. Let us denote to the set of black vertices of $G^\prime$ as $\tau^\prime$ which is a partition of $\overline{E}$ and $\tau^\prime \leq \tau$. We choice $\sim^\prime$ so that $(\tau^\prime,\sim^\prime)\leq (\tau,\sim)$. The resulting graph $G^\prime$ is still a tree and is the same as $G(\tau^\prime)$, moreover if $\overline{e}$ is a loop of $\overline{G}$ then it is a singleton of $\tau$ and then so is of $\tau^\prime$. By Lemmas \ref{Lemma: Equality only if partitions are the same} and \ref{Lemma: Equality when tree},
$$\#(\underset{\sim}{\tau}) \leq \#(\underset{\sim^\prime}{\tau^\prime})=|V|-2(\#(\text{connected components of }G)-1),$$
with equality if and only if $\tau^\prime=\tau$. Thus if we have equality we have $G(\tau)=G(\tau^\prime)$ and therefore $G(\tau)$ is a tree. Conversely if $G(\tau)$ is a tree then Lemma \ref{Lemma: Equality when tree} provides equality.
\end{proof}

Observe that for a labelled oriented graph $G=(V,E,\delta)$ with $\overline{G}=(V,\overline{E})$ being its elementarization we have only considered partitions $\tau \in \cP(\overline{E})$. Let us now introduce another graph similar to $G(\tau)$ but now instead we will consider $\tau\in \cP(E)$.

\begin{definition}\label{Definition: Graph with partitions on the edges with multiplicity}
Let $G=(V,E,\delta)$ be a labelled oriented graph and let $\overline{G}=(V,\overline{E})$ be its elementarization. Let $\rho\in \cP(E)$ and $\sigma\in \cP(E)$ be partitions of the edges such that $\sigma\leq \rho$ and if $u,v$ are in the same block of $\sigma$ then $\overline{e_u}=\overline{e_v}$. We define the graph $G(\sigma,\rho)$ to be the bipartite graph with,
\begin{enumerate}
    \item set of white vertices given by the connected components of $G$,
    \item set of black vertices given by blocks of $\rho$, and,
    \item set of edges given by blocks of $\sigma$ where the edge $B\in \sigma$ connects a white vertex $G_1$ with a black vertex $D$ if and only if $B\subset D$ and $\overline{e}$ is an edge of $\overline{G_1}$ for some $e\in B$.
\end{enumerate}
The latest is well defined as by hypothesis for any $e_u,e_v\in B$ we know $\overline{e_u}=\overline{e_v}$ and therefore the edges $e_u,e_v$ connect the same pair of vertices of $G$, in particular, they are in the same connected component of $G$.
\end{definition}

\section{Asymptotic expansion of the cumulants}\label{Section: Asymptotics}

For the rest of the paper we set $r\in\mathbb{N}$, $m_1,\dots,m_r\in\mathbb{N}$, $m=\sum_{i=1}^r m_i$ and $\gamma = \gamma_{m_1,\dots,m_r} \in S_m$ be defined as the permutation whose cycles are
$$(1,\dots,m_1)(m_1+1,\dots,m_1+m_2)\cdots (m_1+\cdots+m_{r-1}+1,\dots,m).$$
We let
$$\iCycleGamma{i} =\vcentcolon (m_1+\cdots+m_{i-1}+1,\dots,m_1+\cdots+m_i),$$
be the $i^{th}$ cycle of  $\gamma$. For each $1\leq i\leq r$ we let $T_i=(V_i,E_i,\delta_i)$ be the labelled oriented graph with set of vertices
$$V_i=\{m_1+\cdots+m_{i-1}+1,\dots,m_1+\cdots+m_i\},$$
set of edges
$$e_j \vcentcolon = (\gamma(j),j),$$
and label given by $\delta_i(e_j)=j$ for any $m_1+\cdots+m_{i-1}+1\leq j\leq m_1+\cdots+m_i$. We will let $T = T_{1,\dots, r}=(V,E,\delta)$ be the labelled oriented graph consisting on the union of the graphs $T_1,\dots,T_r$. See Figure \ref{Figure: T_4,2,2,2,2,2,2} for an example of such a graph.

\begin{figure}
    \centering
    \includegraphics[width=0.8\textwidth]{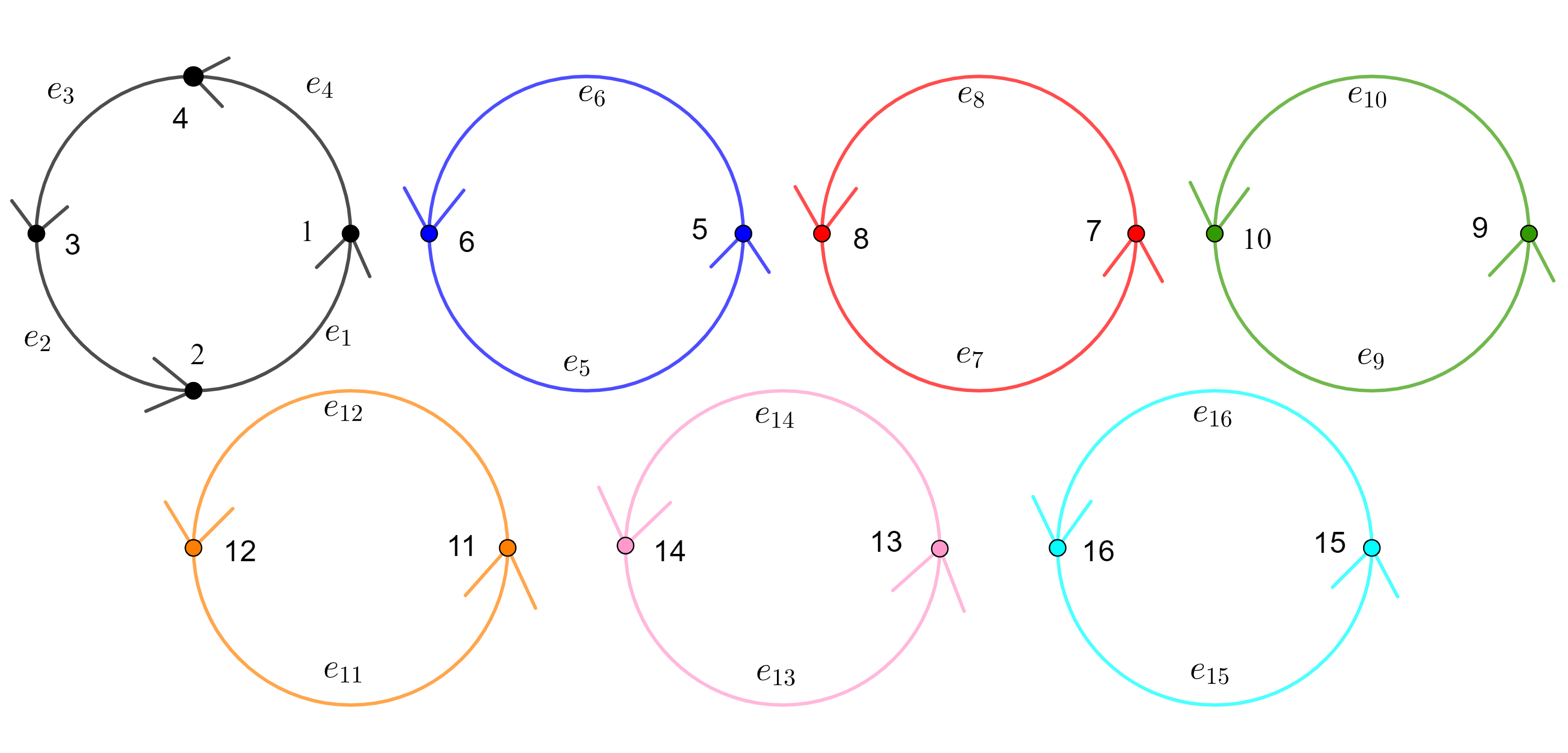}
    \caption{The graph $T$ corresponding to $m_1=4$ and $m_2=m_3=m_4=m_5=m_6=m_7=2$.}
    \label{Figure: T_4,2,2,2,2,2,2}
\end{figure}

We are interested in finding the asymptotic large limit of the quantities
\begin{eqnarray}\label{Equation: Definition of alpha}
\alpha_{m_1,\dots,m_r} \vcentcolon = \lim_{N\rightarrow \infty} N^{r-2}\C_r(\Tr(X^{m_1}),\dots, \Tr(X^{m_r})),
\end{eqnarray}
where $X=X_N$ is a Wigner Ensemble defined as in Section \ref{Section: Introduction}. For the rest of the paper we will answer to this question. In this Section we will prove that the limit of Equation \ref{Equation: Definition of alpha} exists while in Section \ref{Section: Leading order} we will find and explicit expression for the sequence $(\alpha_{m_1,\dots,m_r})_{m_1,\dots,m_r}$.
 
For each $1\leq j\leq r$, let $i_j$ be a function from $V_j$ to $[N]$. We denote by,
$$r(i_j)=\prod_{e\in E_j}x_{i_j(trg(e)),i_j(src(e))},$$
where $trg(\cdot)$ and $src(\cdot)$ denotes the target and source of an edge. It is clear by definition of $T_j$ that
$$\Tr(X^{m_j})=N^{-m_j/2}\sum_{i_j:V_j\rightarrow [N]}r(i_j).$$
This permit us to rewrite $\alpha_{m_1,\dots,m_r}^{(N)} \vcentcolon = N^{r-2}\C_r(\Tr(X^{m_1}),\dots, \Tr(X^{m_r}))$ as
\begin{eqnarray*}
\alpha_{m_1,\dots,m_r}^{(N)} &=& N^{-m/2+r-2}\sum_{i:V\rightarrow [N]} \C_r(r(i_1),\dots,r(i_r)),
\end{eqnarray*}
with $i_j$ being the restriction of $i$ to $V_j$. The latest can be rewritten as
$$\alpha_{m_1,\dots,m_r}^{(N)}=N^{-m/2+r-2}\sum_{\pi\in \cP(V)}\sum_{\substack{i:V\rightarrow [N] \\ \ker(i)=\pi}}\C_j(r(i_1),\dots,r(i_j)),$$
where $\ker(i)$ is the partition of $V$ given by $u$ and $v$ are in the same block of $\ker(i)$ if and only if $i(u)=i(v)$. Now we would like to use the permutation invariance property of our matrices. Observe that the term $K_r(r(i_1),\dots,r(i_r))$ depends only on $\ker(i)$, for a more detailed explanation see \cite[Lemma 6.2.3]{D}, thus
\begin{equation}\label{Equation: First equality of the cumulant}
\alpha_{m_1,\dots,m_r}^{(N)}=N^{-m/2+r-2}\sum_{\pi\in \cP(V)}q(N,\pi)\C_r(\pi),
\end{equation}
where $q(N,\pi)=N(N-1)\cdots (N-\#(\pi)+1)$ and 
$$K_r(\pi) \vcentcolon = \C_r(r(i_1),\dots,r(i_r)),$$
for some $i$ such that $\ker(i)=\pi$. Since the quantity $\C_r(r(i_1),\dots,r(i_r))$ depends only on $\ker(i)$ we consider a canonical function $i:V\rightarrow [N]$ with $\ker(i)=\pi$ as follows.
\begin{notation}\label{Notation: Definition of canonical i}
For a partition $\pi \in \cP(V)$ let $$i^\pi: V \rightarrow [N],$$ 
be the function that is constant on the blocks of $\pi$ defined as follows: 
\begin{enumerate}
\item We enumerate the blocks of $B$ according to its minimal element so that $B_1$ is the block that contains $1$, and for $i>1$, $B_i$ is the block that contains the smallest integer which is not contained in $\cup_{k=1}^{i-1} B_k$.
\item $i^{\pi}(b)=i$ for all $b\in B_i$.
\end{enumerate}
\end{notation}

With Notation \ref{Notation: Definition of canonical i} and Equation \ref{Equation: First equality of the cumulant} we get
\begin{equation}\label{Equation: Second equality of the cumulant}
\alpha_{m_1,\dots,m_r}^{(N)}=N^{-m/2+r-2}\sum_{\pi\in \cP(V)}q(N,\pi)\C_r(\pi),
\end{equation}
where $q(N,\pi)=N(N-1)\cdots (N-\#(\pi)+1)$ and 
$$\C_r(\pi) \vcentcolon = \C_r(r(i^{\pi}_1),\dots,r(i^{\pi}_r)).$$

From \cite[Theorem 2.2]{KS},
$$\C_r(\pi)=\sum_{\substack{\tau\in \cP(m) \\ \tau\vee\gamma=1_m}}\C_{\tau}(x_{i^{\pi}(1),i^{\pi}(\gamma(1))},\dots,x_{i^{\pi}(m),i^{\pi}(\gamma(m))}).$$
Substituting the last Equation into Equation \ref{Equation: Second equality of the cumulant} we get,
\begin{equation}\label{Equation: Third equality of the cumulant}
\alpha_{m_1,\dots,m_r}^{(N)}=\sum_{\pi\in \cP(m)}\sum_{\substack{\tau\in \cP(m)\\ \tau\vee\sigma=1_m}}N^{-m/2+r-2}q(N,\pi)\C_{\tau}(\pi),
\end{equation}
where $q(N,\pi)=N(N-1)\cdots (N-\#(\pi)+1)$ and 
$$\C_{\tau}(\pi) \vcentcolon = \C_{\tau}(x_{i^{\pi}(1),i^{\pi}(\gamma(1))},\dots,x_{i^{\pi}(m),i^{\pi}(\gamma(m))}).$$
Equation \ref{Equation: Third equality of the cumulant} provides an expansion in terms of $N$ of the quantities $\alpha_{m_1,\dots,m_r}^{(N)}$. Our next goal will be making use of this Equation to prove the existence of the large $N$ limit and finding a characterization of the leading order.

The next Lemmas prove the existence of the large $N$ limit of $\alpha_{m_1,\dots,m_r}^{(N)}$ and hence the moments and cumulants $\alpha_{m_1,\dots,m_r}$ and $\K_{m_1,\dots,m_r}$. 

\begin{notation}\label{Notation: Partition of edges and integers}
For a partition $\rho\in \cP(m)$ and $\pi\in \cP(m)$ we let $e_{\rho}^{\pi}\in \cP(E^{\pi})$ be the partition given by $\{u_1,\dots,u_n\}$ is a block of $\rho$ if and only if $\{e_{u_1}^{\pi},\dots,e_{u_n}^{\pi}\}$ is a block of $e_{\rho}^{\pi}$.
\end{notation}

\begin{lemma}\label{Lemma: Blocks of sigma determined edges of T^gamma sigma}
Let $T=(V,E,\delta)$ be the labelled oriented graph defined at the beginning of this section and let $\sigma\in \cP(m)$ be such that all its blocks have size $2$. Then
$$\sigma \leq \overline{\gamma\sigma}.$$
\end{lemma}
\begin{proof}
If $\{u,v\}$ is a block of $\sigma$ then $\gamma\sigma(u)=\gamma(v)$ and therefore,
$$[u]_{\gamma\sigma}=[\gamma(v)]_{\gamma\sigma}.$$
Similarly,
$$[v]_{\gamma\sigma}=[\gamma(u)]_{\gamma\sigma}.$$
Since $e_u=(\gamma(u),u)$ then $e_u^{\gamma\sigma}$ goes from $[\gamma(u)]_{\gamma\sigma}$ to $[u]_{\gamma\sigma}$ and in the same way $e_v^{\gamma\sigma}$ goes from $[\gamma(v)]_{\gamma\sigma}$ to $[v]_{\gamma\sigma}$. By the proved at the beginning, $\overline{e_u^{\gamma\sigma}}=\overline{e_v^{\gamma\sigma}}$ i.e. $u$ and $v$ are in the same block of $\overline{\gamma\sigma}$.
\end{proof}

\begin{lemma}\label{Lemma: Gamma sigma overline less or equal than pi}
Let $T=(V,E,\delta)$ be the labelled oriented graph defined at the beginning of this section and let $\pi\in \cP(m)$ and $\sigma\in \cP(m)$ be such that all the blocks of $\sigma$ have size $2$. If for any block $\{u,v\}\in \sigma$ the edges $e_u^{\pi}$ and $e_v^{\pi}$ connect the same pair of vertices of $T^{\pi}$ and with the opposite orientation, then
$$\gamma\sigma \leq \pi.$$
\end{lemma}
\begin{proof}
If $\{u,v\}$ is a block of $\sigma$ then $e_u^{\pi}$ goes from $[\gamma(u)]_{\pi}$ to $[u]_{\pi}$ in $T^{\pi}$. Similarly $e_v^{\pi}$ goes from $[\gamma(v)]_{\pi}$ to $[v]_{\pi}$. By hypothesis,
$$[u]_{\pi}=[\gamma(v)]_{\pi}.$$
Thus,
$$[\gamma\sigma(u)]_{\pi}=[\gamma(v)]_{\pi}=[u]_{\pi},$$
which proves that $u$ and $\gamma\sigma(u)$ are in the same block of $\pi$.
\end{proof}

\begin{notation and remark}\label{Notation: Edge associated to block}
Let $T=(V,E,\delta)$ be the graph defined at the beginning of the Section. Let $\pi \in \cP(m)$ and let $T^{\pi}=(V^{\pi},E^{\pi},\delta)$ be the quotient graph of $T$ under $\pi$. Let $\tau\in \cP(m)$ be such that $\tau\leq \overline{\pi}$. Then for any block, $B=\{i_1,\dots,i_s\}$, of $\tau$ the edges $e_{i_1}^{\pi},\dots,e_{i_s}^{\pi}$ connect all the same pair of vertices of $T^{\pi}$, i.e $\overline{e_{i_1}^{\pi}}=\cdots =\overline{e_{i_s}^{\pi}}$. Under this scenario, we use the notation $\overline{e_B^{\pi}}$ to mean $\overline{e_{i_1}^{\pi}}$. When $\overline{e_B^{\pi}}$ is a loop then all $e_i^{\pi}$ are loops of $T^{\pi}$ and therefore the orientation of each $e_i^{\pi}$ is unique. Thus, in this case by $e_B^{\pi}$ we mean $e_{i_1}^{\pi}$.
\end{notation and remark}

From now on for $\sigma\in \cP_2(m)$ and $\sigma\leq \tau$ by $T^{\gamma\sigma}(\sigma,\tau)$ we mean $T^{\gamma\sigma}(e_{\sigma}^{\gamma\sigma},e_{\tau}^{\gamma\sigma})$.

\begin{lemma}\label{Lemma: First characterization of the limit}
Let $T=(V,E,\delta)$ be the labelled oriented graph defined at the beginning of the section. Let $\pi\in \cP(m)$ and $\tau\in \cP(m)$ be such that $\tau\vee\gamma=1_m$ and
$$\C_{\tau}(x_{i^{\pi}(1),i^{\pi}(\gamma(1))},\dots,x_{i^{\pi}(m),i^{\pi}(\gamma(m))})\neq 0.$$
Then,
\begin{enumerate}
    \item $\tau\leq \overline{\pi}$.
    \item Any block of $\tau$ has even size and it contains the same number of edges in each orientation of the graph $T^{\pi}$.
    \item If $D\in \tau$ is a block of $\tau$ such that $e_D^{\pi}$ is a loop of $T^{\pi}$, then $D$ has size $2$.
    \item There exist $\sigma\in \cP_2(m)$ with $\sigma\leq \tau$ and such that if $\{u,v\}$ is a block of $\sigma$ then $e_u^{\pi}$ and $e_v^{\pi}$ connect the same pair of vertices of $T^{\pi}$ with opposite orientation.
    \item $$\#(\pi)-m/2+r-2\leq 0.$$
\end{enumerate}
Moreover we have equality in $(5)$ if and only if there exist $\sigma\in \cP_2(m)$ as in $(4)$ such that all of the following satisfy,
\begin{enumerate}[label=$($\subscript{L}{{\arabic*}}$)$]
    \item $\sigma\in \NC_2(m_1,\dots,m_r)\cup \NC_2^{nc}(m_1,\dots,m_r)$.
    \item For each block, $W$, of $\sigma\vee\gamma$, the partition $\pi|_W = \gamma|_W\sigma|_W$.

   \item The graph $T^{\pi}$ is the quotient graph along edges, all of them non-loops, of the graph $T^{\gamma\sigma}$, i.e. the graph $T^{\pi}$ is obtained by identifying non-loops edges of $T^{\gamma\sigma}$, where by identifying edges we mean merging the two pair of vertices of the edges.
        
        
    \item The graph $\overline{T^{\gamma\sigma}(\sigma,\tau)^{\overline{\pi_\tau}}}$ is a tree. Here $\overline{\pi}_\tau$ is the partition given by merging the blocks $D_1$ and $D_2$ of $\tau$ whenever they are in the same block of $\overline{\pi}$ except when $e_{D_i}^{\pi}$ is a loop of $T^{\pi}$.
\end{enumerate}
Examples of pairs $(\pi,\tau)$ for which $\#(\pi)-m/2+r-2=0$ and $\#(\pi)-m/2+r-2<0$ can be seen in Example \ref{Example: Example of limit and non-limit partition}.
\end{lemma}
\begin{proof}
If $u$ and $v$ are in the same block of $\tau$ then $e_u^{\pi}$ and $e_v^{\pi}$ must connect the same vertices of $T^{\pi}$ otherwise $x_{i^{\pi}(u),i^{\pi}(\gamma(u))}$ and $x_{i^{\pi}(v),i^{\pi}(\gamma(v))}$ are independent and then $\C_{\tau}(\pi)=0$. This proves $(1)$. Observe that any block of $e_{\overline{\pi}}^{\pi}$ must consist of the same number of edges in each orientation, otherwise for any $\tau \leq \overline{\pi}$ there is a block of $e_{\tau}^{\pi}$ that doesn't consist on the same number of edges in each orientation meaning $\C_{\tau}(\pi)=0$. The same reasoning applies to any block of $\tau$ and hence $(2)$. To prove $(3)$ it is enough to observe that if a block, $D$, of $\tau$ that contains more than two elements and is such that $e_u^{\pi}$ is a loop of $T^{\pi}$ with $u\in D$ then $\C_{\tau}(\pi)$ has a factor of the form $\C_D(x_{1,1},\dots,x_{1,1})$ with $|D|\geq 4$ which is $0$ and so is $\C_{\tau}(\pi)$. Now let us prove $(4)$. Let $\sigma\in \cP_2(m)$ be such that $\sigma\leq \tau$ and if $\{u,v\}$ is a block of $\tau$ then $e_u^{\pi}$ and $e_v^{\pi}$ connect the same pair of vertices of $T^{\pi}$ with opposite orientation (such a $\sigma$ exist thanks to $(2)$). 

Finally let us prove $(5)$. Let $W_1,\dots,W_s$ be the blocks of $\sigma\vee\gamma$. For each $1\leq i\leq s$ we let $\pi_i=\pi|_{W_i}$ and let $T_{W_i}^{\pi_i}$ be the quotient graph of $T_{W_i}$ under $\pi_i$. Here $W_i$ is a union of cycles $\iCycleGamma{i_1},\dots,\iCycleGamma{i_t}$ of $\gamma$, and therefore by $T_{W_i}$ we mean the union of the graphs $T_{i_1},\dots,T_{i_t}$. Let $\hat{\pi} = \pi_1\cup\cdots \cup \pi_s$, i.e. the partition whose any block is a block of $\pi_i$ for some $1\leq i\leq s$. Let $G=T^{\hat{\pi}}=(V^{\hat{\pi}},E^{\hat{\pi}},\delta)$ which is the same as the union of the graphs $T_{W_1}^{\pi_1},\dots,T_{W_s}^{\pi_s}$ and let $\overline{G}=(V^{\hat{\pi}},\overline{E^{\hat{\pi}}})$ be its elementarization. Observe that if $B=\{u,v\}$ is a block of $\sigma$ then by $(2)$ the edges $e_u^{\pi}$ and $e_v^{\pi}$ connect the same pair of vertices of $T^{\pi}$, which we call $C$ and $D$. Since $W_i$ is a union of blocks of $\sigma$ the block $B$ must be completely contained in some $W_i$, say $W_1$. The partition $\pi_1$ contains the blocks $C\cap W_1$ and $D\cap W_1$, hence the edges $e_u^{\hat{\pi}}$ and $e_v^{\hat{\pi}}$ will connect the same pair of vertices of $T^{\hat{\pi}}$ corresponding to $C\cap W_1$ and $D\cap W_2$. In other words we still have $\sigma\leq \overline{\hat{\pi}}$. Therefore the graph $G(\sigma,\tau)\vcentcolon = G(e_\sigma^{\hat{\pi}},e_\tau^{\hat{\pi}})$ is well defined. 


Let $\overline{\pi_\tau}$ be the partition obtained after joining two blocks of $\tau$ whenever the blocks are contained in the same block of $\overline{\pi}$, except when the blocks of $\tau$ correspond to loops of $T^{\pi}$.

Any block of $\tau$ determines a block of $e_\tau^{\hat{\pi}}$ and conversely, so we may identify the blocks of $e_\tau^{\hat{\pi}}$ with the blocks of $\tau$ and the blocks of $e_\sigma^{\hat{\pi}}$ with the blocks of $\sigma$. In this way the black vertices of $G(\sigma,\tau)$ are the blocks of $\tau$ and its edges are the blocks of $\sigma$. Note that $\hat{\pi}\leq\pi$ and therefore $\overline{\hat{\pi}}\leq \overline{\pi}$. 


We let $\hat{\tau}$ to be the partition given by joining blocks of $\tau$ whenever they are contained in the same block of $\overline{\pi_\tau}$. If $B_1,\dots,B_n$ are blocks of $\tau$ that correspond to loops of $T^{\pi}$ and such that $B_i$ are all connected to the same white vertex of $G(\sigma,\tau)$ and $B_i$ are all contained in the same block of $\overline{\pi}$ then we also let $B_1\cup \cdots \cup B_n$ to be a block of $\hat{\tau}$.

We let $\bar{\tau}\in \cP(\overline{E^{\hat{\pi}}})$ be the partition given by $\overline{e_u^{\hat{\pi}}}$ and $\overline{e_v^{\hat{\pi}}}$ are in the same block if and only if $u$ and $v$ are in the same block of $\hat{\tau}$. For each block of $\bar{\tau}$ with edges $\overline{e_1^{\hat{\pi}}},\dots,\overline{e_n^{\hat{\pi}}}$ we consider an identification of the edges $\sim_{\overline{e_1^{\hat{\pi}}},\dots,\overline{e_n^{\hat{\pi}}}}$ given by $\pi$ in the following manner: for each $u,v=1,\dots,n$ since $u$ and $v$ are in the same block of $\hat{\tau}$ then so are in the same block of $\overline{\pi}$ and therefore if $e_u=(u_1,v_1)$ and $e_v=(u_2,v_2)$ then either $u_1,u_2$ are in the same block of $\pi$ and $v_1,v_2$ are in the same block of $\pi$ or the other way around. We then define $[u_1]_{\hat{\pi}} \sim_{\overline{e_1^{\hat{\pi}}},\dots,\overline{e_n^{\hat{\pi}}}}[u_2]_{\hat{\pi}}$ and $[v_1]_{\hat{\pi}}\sim_{\overline{e_1^{\hat{\pi}}},\dots,\overline{e_n^{\hat{\pi}}}}[v_2]_{\hat{\pi}}$ in the former case or $[u_1]_{\hat{\pi}}\sim_{\overline{e_1^{\hat{\pi}}},\dots,\overline{e_n^{\hat{\pi}}}}[v_2]_{\hat{\pi}}$ and $[v_1]_{\hat{\pi}}\sim_{\overline{e_1^{\hat{\pi}}},\dots,\overline{e_n^{\hat{\pi}}}}[u_2]_{\hat{\pi}}$ in the latter case. We have then constructed and oriented partition of the edges $(\bar{\tau},\sim)$. Now we consider another partition $\hat{\tau_2}$ given by merging two blocks of $\tau$ whenever the blocks are contained in the same block of $\overline{\pi}$, i.e. $\hat{\tau_2}=\overline{\pi}$. Let us consider as before $\bar{\tau_2}\in \cP(\overline{E^{\hat{\pi}}})$ to be the partition give by $\overline{e_u^{\hat{\pi}}}$ and $\overline{e_v^{\hat{\pi}}}$ are in the same block if and only if $u$ and $v$ are in the same block of $\hat{\tau_2}$. We then construct and oriented partition of the edges $(\bar{\tau_2},\sim_2)$ as done previously. It is clear that,
$$(\bar{\tau},\sim) \leq (\bar{\tau_2},\sim_2).$$
Moreover $\underset{\sim_2}{\bar{\tau_2}}=\pi$ because $\hat{\tau_2}=\overline{\pi}$ and the identification of the edges respects the order given by $\pi$.

Observe that the graph $\overline{G(\sigma,\tau)^{\hat{\tau}}}$ is a connected graph with black vertices the blocks of $\hat{\tau}$ and exactly one edge for each block of $\overline{E^{\hat{\pi}}}$ because $G(\sigma,\tau)$ has at least one edge for each block of $\overline{E^{\hat{\pi}}}$ (since $\sigma\leq \overline{\hat{\pi}}$) then after doing the elementarization we get exactly one for each one. Further, each block of $\hat{\tau}$ determines a block of $\bar{\tau}$ and conversely. Therefore we conclude that the graph $G(\bar{\tau})$ is the same as the graph $\overline{G(\sigma,\tau)^{\hat{\tau}}}$ because in $G(\sigma,\tau)^{\hat{\tau}}$ we merge blocks of $\tau$ into the same block of $\hat{\tau}$, $\hat{\tau}$ and $\bar{\tau}$ have the same number of blocks and when we do the elementarization, $\overline{G(\sigma,\tau)^{\hat{\tau}}}$, we take away multiplicity of the edges. Moreover, let $\overline{e_1^{\hat{\pi}}}$ be a loop of $\overline{G}$ and let us write $\overline{e_1^{\hat{\pi}}}=\{e_1^{\hat{\pi}},\dots,e_n^{\hat{\pi}}\}$. Any block of $e_\tau^{\hat{\pi}}$ restricted to $\{e_1^{\hat{\pi}},\dots,e_n^{\hat{\pi}}\}$ must be of size $2$ otherwise $\C_{\tau}(\pi)=0$, so we may assume that $n=2t$ and the blocks of $\tau$ are $\{e_1^{\hat{\pi}},e_2^{\hat{\pi}}\},\dots,\{e_{2t-1}^{\hat{\pi}},e_{2t}^{\hat{\pi}}\}$. This means that the edge joining the white vertex consisting on the connected component of $G$ containing $\{e_1^{\hat{\pi}},\dots,e_n^{\hat{\pi}}\}$ and the black vertex the block $\{e_i^{\hat{\pi}},e_{i+1}^{\hat{\pi}}\}$ of $e_\tau^{\hat{\pi}}$ is a leaf of $G(\sigma,\tau)$ for $i=1,3,\dots,2k-1$. By definition of $\hat{\tau}$ all the black vertices $\{e_i^{\hat{\pi}},e_{i+1}^{\hat{\pi}}\}$ are merged in $G(\sigma,\tau)^{\hat{\tau}}$ and only them as we do not merge leaves that correspond to loops of $G$ connected to another connected component of $G$. This results in $\{e_1^{\hat{\pi}},\dots,e_n^{\hat{\pi}}\}$ being a block of $e_{\hat{\tau}}^{\hat{\pi}}$ and therefore $\{\overline{e_1^{\hat{\pi}}}\}$ being a block of $\bar{\tau}$. We conclude by Lemmas \ref{Lemma: Underline partition respects order} and \ref{Lemma: Inequality with equality iff tree},
\begin{equation}\label{Aux: Inequality 1}
\#(\pi)=\#(\underset{\sim_2}{\bar{\tau_2}})\leq \#(\underset{\sim}{\bar{\tau}}) \leq |V^{\hat{\pi}}|-2(s-1),
\end{equation}

By definition of $\hat{\tau}$, the blocks of $\bar{\tau}$ of size larger than $2$ contain edges $\overline{e_u^{\hat{\pi}}}$ where $e_u^{\pi}$ is a non-loop of $T^{\pi}$. However $e_u^{\pi}$ is a non-loop of $T^{\pi}$ if and only if $e_u^{\hat{\pi}}$ is a non-loop of $T^{\hat{\pi}}$. So we conclude that if $\#(\underset{\sim_2}{\bar{\tau_2}})=\#(\underset{\sim}{\bar{\tau}})$ then the graph $T^{\pi}$ is obtained by identifying the non-loops edges of $T^{\hat{\pi}}$ that are joining the same pair of vertices of $T^{\pi}$. Conversely if $T^{\pi}$ is obtained by identifying non-loops edges of $T^{\hat{\pi}}$ then clearly identifying all the non-loops edges of $T^{\hat{\pi}}$ that are joining the same pair of vertices of $T^{\pi}$ recovers $T^{\pi}$, i.e. $\#(\underset{\sim_2}{\bar{\tau_2}})=\#(\underset{\sim}{\bar{\tau}})$. So the first inequality becomes equality if and only if $(L_3)$ is satisfied with the minor difference that we consider $\hat{\pi}$ instead of $\gamma\sigma$.

Similarly, thanks to \ref{Lemma: Inequality with equality iff tree} the second inequality becomes equality if and only if $G(\bar{\tau})$ is a tree. This occurs if and only if $\overline{T^{\hat{\pi}}(\sigma,\tau)^{\overline{\pi_\tau}}}$ is a tree. This is condition $(L_4)$ taking $\hat{\pi}$ instead of $\gamma\sigma$.
 

On the other hand, observe that $\gamma\sigma\leq \pi$ by Lemma \ref{Lemma: Gamma sigma overline less or equal than pi}, hence for each $1\leq i\leq s$ we have $\gamma_i\sigma_i \leq \pi_i$ with $\gamma_i=\gamma|_{W_i}$ and $\sigma_i=\sigma|_{W_i}$. Thus,
$$\sum_{i=1}^s \#(\pi_i) \leq \sum_{i=1}^s \gamma_i\sigma_i =\#(\gamma\sigma)\leq m+2s-\#(\gamma)-\#(\sigma).$$
Where last inequality comes from \cite[Equation 2.9]{MN}.
We then have,
\begin{eqnarray*}\label{Aux: Inequality 2}
|V^{\hat{\pi}}|-2(s-1) &=& \sum_{i=1}^s \#(\pi_i)-2(s-1) \nonumber \\
&\leq & \sum_{i=1}^s \#(\gamma_i\sigma_i)-2(s-1) \\
&=& \#(\gamma\sigma)-2(s-1)\leq m/2-r+2.  \nonumber  
\end{eqnarray*}
With Equality if and only if,
\begin{enumerate}
    \item $\pi_i=\gamma_i\sigma_i$, equivalently $\hat{\pi}=\gamma\sigma$.
    \item $\sigma\in S_{NC}(\gamma)\cup S_{NC}^{nc}(\gamma)$
\end{enumerate}

The last two conditions are precisely $(L_2)$ and $(L_1)$ respectively. From here we conclude as whenever the four conditions are satisfied we can change $\hat{\pi}$ by $\gamma\sigma$ and viceversa.
\end{proof}

\begin{example}\label{Example: Example of limit and non-limit partition}
Let $r=10$ and $m_1=m_6=5,m_2=m_7=4,m_3=m_8=3$ and $m_4=m_5=m_9=m_{10}=1$. So that,
\begin{eqnarray*}
\gamma &=& (1,2,3,4,5)(6,7,8,9)(10,11,12)(13)(14) \\
&& (15,16,17,18,19)(20,21,22,23)(24,25,26)(27)(28).
\end{eqnarray*}
Let,
\begin{eqnarray*}
\pi &=&\{1,4,8\},\{1,3,5,9,14,15,17,19,23,28\}\\
&&\{2,6,8,12\},\{16,20,22,26\},\{7,10,11,13\},\{21,24,25,27\}.
\end{eqnarray*}
The graph $T^{\pi}$ can be seen in Figure \ref{Figure: Example of limit and non-limit partition} part A, so that,
\begin{eqnarray*}
\overline{\pi}&=&\{3,4,17,18\},\{5,14,19,28\},\{1,2,8,9\},\{15,16,22,23\}\\
&&\{6,7,11,12\},\{20,21,25,26\},\{10,13\},\{24,27\}.
\end{eqnarray*}
Let,
\begin{eqnarray*}
\tau &=& \{3,4,17,18\},\{5,14\},\{19,28\},\{1,2,8,9\},\{15,16,22,23\} \\
&& \{6,11\},\{7,12\},\{20,25\},\{21,26\},\{10,13\},\{24,27\}.
\end{eqnarray*}
So that $(\pi,\tau)$ satisfies $\#(\pi)-m/2+r-2=0$ and $\K_{\tau}(\pi)\neq 0$. The permutation $\sigma\in \cP_2(m)$ is determined at the blocks of $\tau$ of size $2$, for the blocks of size $4$ we have two options for each block. Any choice of $\sigma$ actually works, this can be seen from the proof where there is no hypothesis for the choice except that if $\{u,v\}$ is a block of $\sigma$ then $e_u^{\pi}$ and $e_v^{\pi}$ connect the same pair of vertices with the opposite orientation. Let us choose,
\begin{eqnarray*}
\sigma&=&(1,8)(2,9)(3,4)(5,14)(6,11)(7,12)(10,13)(15,22)(16,23)(17,18) \\&&(19,28)(20,25)(21,26)(24,27)
\end{eqnarray*}
Note that $\gamma\vee\sigma=\{1,\dots,14\}\{15,\dots,28\}$ and,
\begin{eqnarray*}
\gamma\sigma &=&(1,9,3,5,14)(2,6,12,8)(4)(7,10,13,11)\\
&&(15,23,17,19,28)(16,20,26,22)(18)(21,24,27,25)
\end{eqnarray*}
Thus 
$$\#(\sigma)+\#(\gamma\sigma)+\#(\gamma)=14+8+10=32=m+2\#(\gamma\vee\sigma),$$
so that $\sigma\in \NC_2^{nc}(m_1,\dots,m_10)$ and then condition $L_1$ is satisfied. If $A=\{1,\dots,14\}$ and $B=\{15,\dots,28\}$ then,
\begin{eqnarray*}
\gamma|_A\sigma|_A = (1,9,3,5,14)(2,6,12,8)(4)(7,10,13,11) = \pi|_A
\end{eqnarray*}
Similarly we can verify $\gamma|_B\sigma|_B=\pi|_B$ so that condition $L_2$ is satisfied. To verify that condition $L_3$ is satisfied observe that the blocks $B_1=\{5,14\},B_2=\{19,28\},B_3=\{10,13\}$ and $B_4=\{24,27\}$ of $\sigma$ correspond to loops $e_{B_i}^{\gamma\sigma}$ of $T^{\gamma\sigma}$. Note that only the blocks $B_1$ and $B_2$ are in the same block of $\overline{\pi}$. The loop $e_{B_1}^{\gamma\sigma}$ has vertex $[B_1]_{\gamma\sigma}=\{1,9,3,5,14\}$ while the loop $e_{B_2}^{\gamma\sigma}$ has vertex $[B_2]_{\gamma\sigma}=\{15,23,17,19,28\}$. Note that the edges $e_3^{\gamma\sigma}$ and $e_{17}^{\gamma\sigma}$ are non-loop edges of $T^{\gamma\sigma}$ such that $[B_1]_{\gamma\sigma}$ and $[B_2]_{\gamma\sigma}$ are vertices of $e_3^{\gamma\sigma}$ and $e_{17}^{\gamma\sigma}$ respectively. Finally the edges $e_3^{\pi}$ and $e_{17}^{\pi}$ connect the same pair of vertices and one of this vertices is $\{1,3,5,9,14,15,17,19,23,28\}$ that contains $[B_1]_{\gamma\sigma}$ and $[B_2]_{\gamma\sigma}$ as condition $L_3$ says. Intuitively condition $L_3$ says that the graph $T^{\pi}$ is obtained from the graph $T^{\gamma\sigma}$ by connecting components of $T^{\gamma\sigma}$ along edges which are non-loops. In this example this edges correspond to $e_3^{\gamma\sigma}$ and $e_{17}^{\gamma\sigma}$. Finally, as mentioned before, the loops of $T^{\pi}$ are $e_5^{\pi},e_{14}^{\pi},e_{19}^{\pi},e_{28}^{\pi},e_{10}^{\pi},e_{13}^{\pi},e_{24}^{\pi}$ and $e_{27}^{\pi}$. Thus,
\begin{eqnarray*}
\overline{\pi_\tau}&=& \{3,4,17,18\},\{5,14\},\{19,28\},\{1,2,8,9\},\{15,16,22,23\}\\
&&\{6,7,11,12\},\{20,21,25,26\},\{10,13\},\{24,27\}.
\end{eqnarray*}
From this it can be verified that the graph $\overline{T^{\gamma\sigma}(\sigma,\tau)^{\overline{\pi_\tau}}}$ is a tree as seen in Figure \ref{Figure: Example of limit and non-limit partition} part C. Intuitively, this condition says that the connection of the connected components of $T^{\gamma\sigma}$ along non-loops edges is done in a minimal way, that is, we identified the smallest number of edges, this condition is equivalent to the tree condition. One might be tempted to think that only conditions $L_1,L_2$ and $L_4$ are required. However the loops plays a crucial role, as a counterexample to why $L_3$ condition is required we propose the partition,
\begin{eqnarray*}
\pi^\prime &=&\{1,4,8\},\{1,3,5,9,14,15,17,19,23,28\},\{2,6,8,12\}\\
&&\{16,20,22,26\},\{7,10,11,13,21,24,25,27\}.
\end{eqnarray*}
This partition is obtained from merging the blocks $\{7,10,11,13\}$ and $\{21,24,25,27\}$ of $\pi$. In terms of the graph $T^{\pi^\prime}$ it means the loops $e_{B_3}^{\pi^\prime}$ and $e_{B_4}^{\pi^\prime}$ are connecting the same vertex. By definition of $\overline{{\pi^\prime}_\tau}$ the condition $L_4$ is still satisfied. However $\#(\pi^\prime)-m/2+r-2=-1<0$ so this partition is not in the limit. This illustrates why the connection along edges of the connected components of $T^{\gamma\sigma}$ must be done trough non-loop edges.
\end{example}

\begin{figure}
    \centering
    \includegraphics[width=0.7\textwidth]{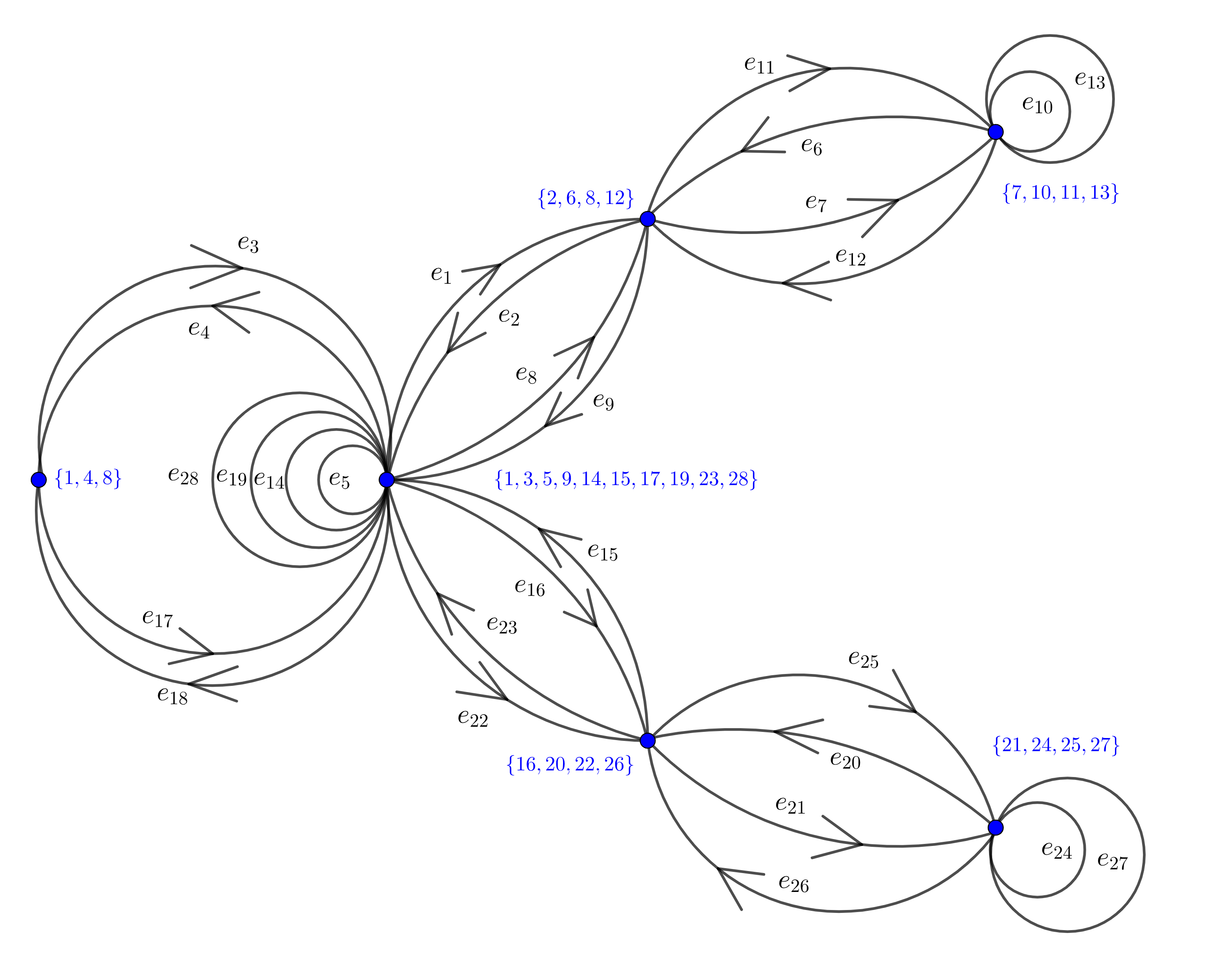}
    
{\small A) The quotient graph $T^{\pi}$.}

    \includegraphics[width=0.6\textwidth]{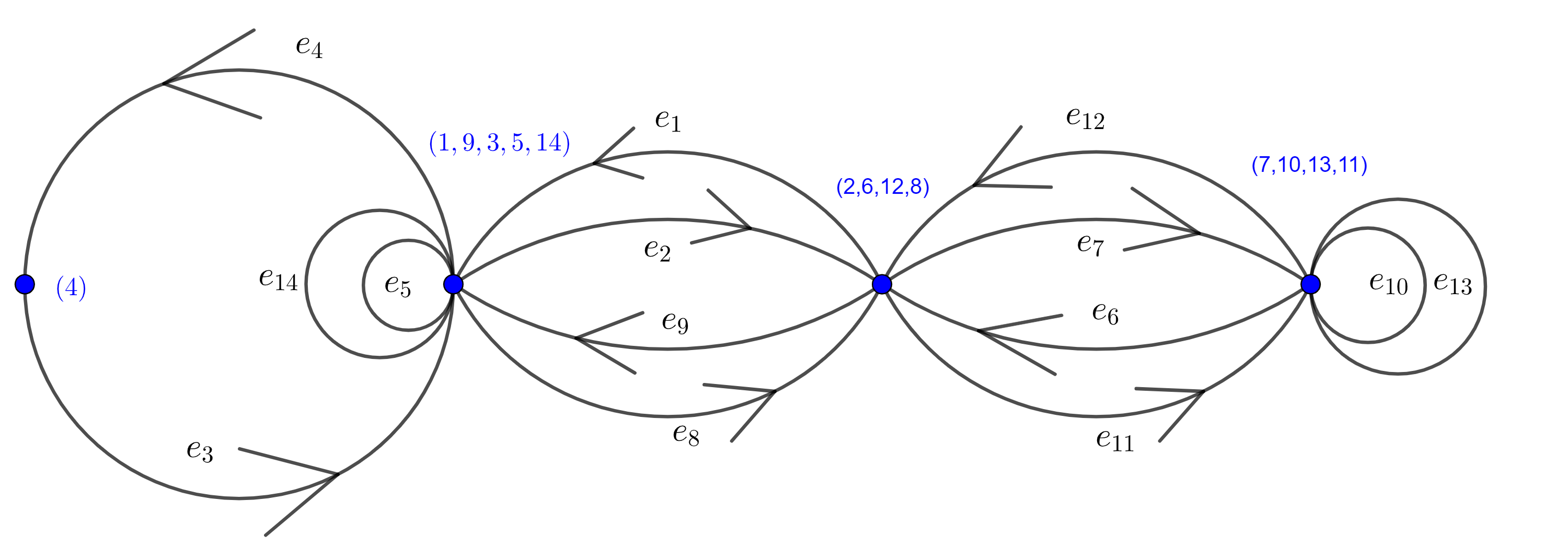}
    \includegraphics[width=0.6\textwidth]{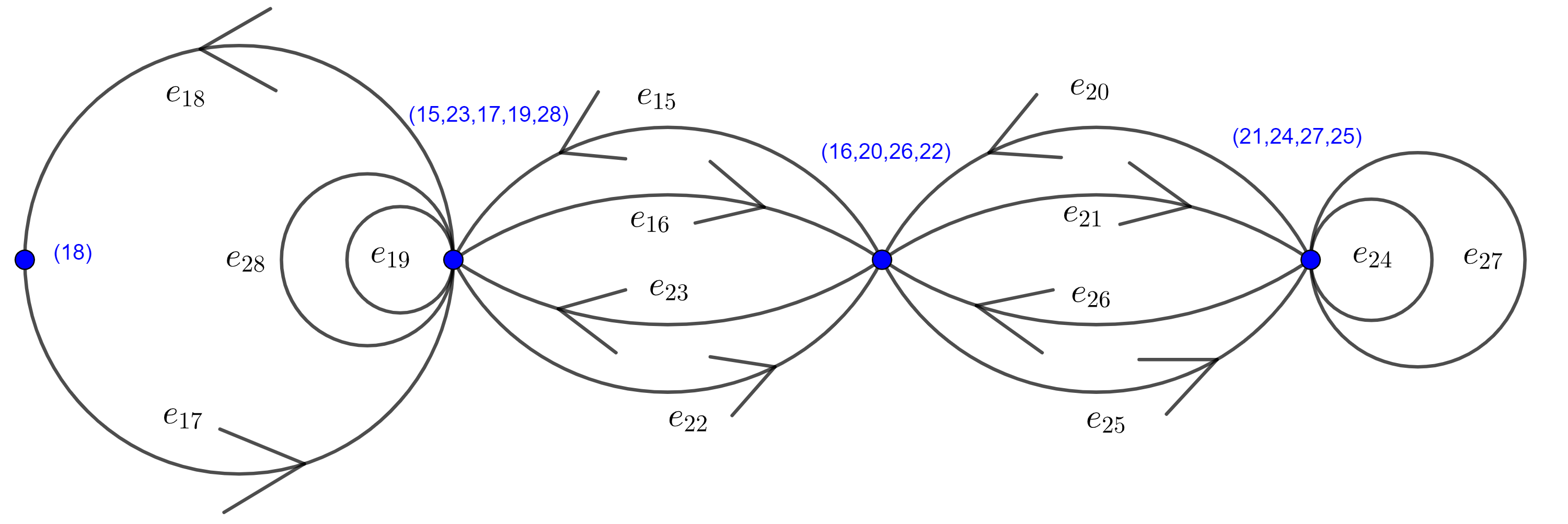}
    
{\small B) The graph $T^{\gamma\sigma}$. Note that $T^{\gamma\sigma}$ has two connected components}\\

    \includegraphics[width=0.7\textwidth]{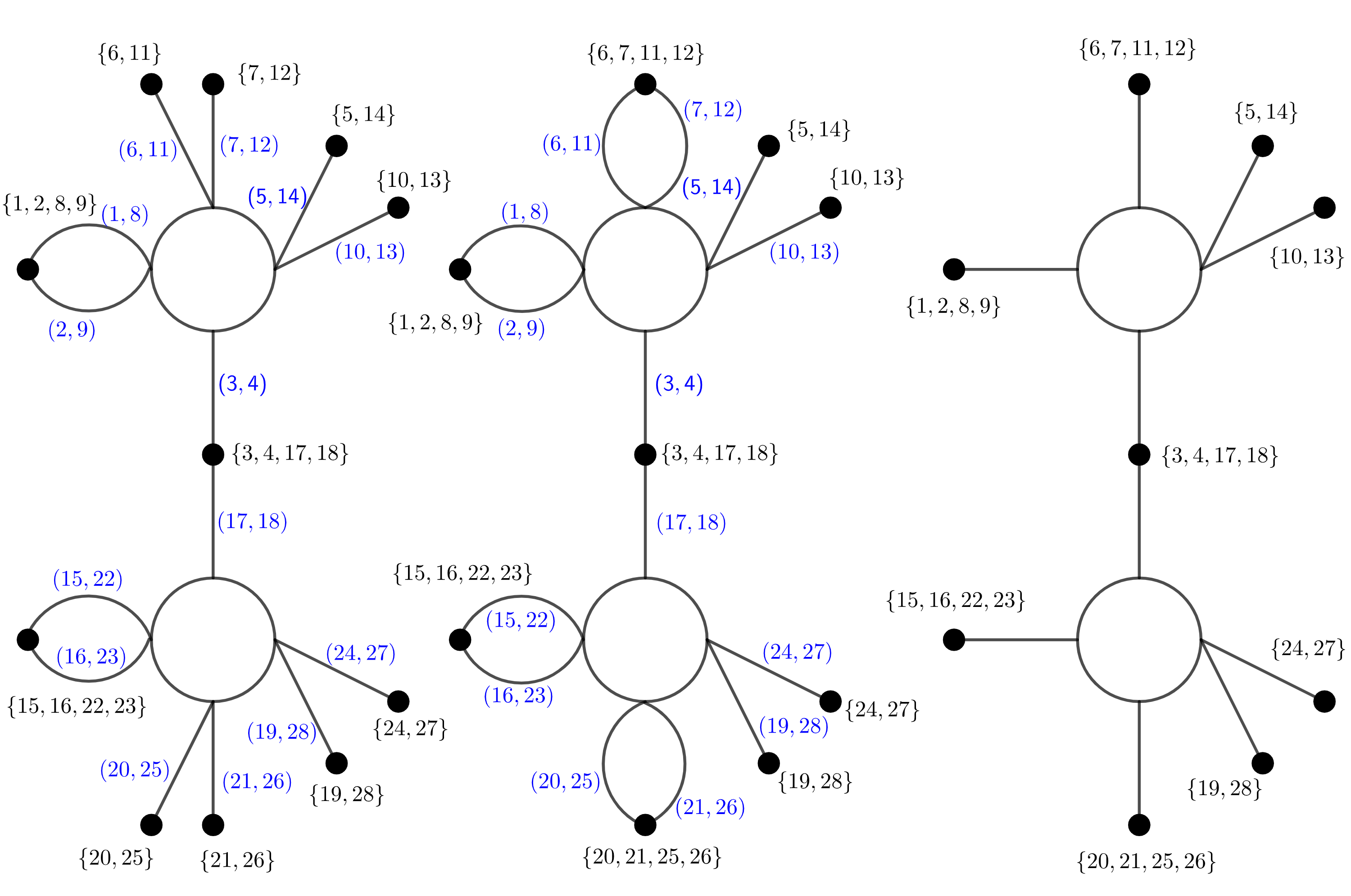}

{\small C) From left to right the graphs $T^{\gamma\sigma}(\sigma,\tau),T^{\gamma\sigma}(\sigma,\tau)^{\overline{\pi_\tau}}$ and $\overline{T^{\gamma\sigma}(\sigma,\tau)^{\overline{\pi_\tau}}}.$}\\
    
    \caption{The graphs of Example \ref{Example: Example of limit and non-limit partition}}
    \label{Figure: Example of limit and non-limit partition}
\end{figure}

\begin{remark}\label{Remark: edge type is preserved}
In the setting of Lemma \ref{Lemma: First characterization of the limit} note that $e_u^{\gamma\sigma}$ is a non-loop of $T^{\gamma\sigma}$ if and only if $e_u^{\pi}$ is a non-loop of $T^{\pi}$. This can be obtained from Lemma \ref{Lemma: Equality when tree} part $(iii)$ which guarantees that $\pi$ never merges two blocks of $\gamma\sigma$ in the same connected component and therefore any non-loop of $T^{\gamma\sigma}$ is also a non-loop of $T^{\pi}$. Conversely if $e_u^{\pi}$ is a non-loop of $T^{\pi}$ then it must be a non-loop of $T^{\gamma\sigma}$ since $\gamma\sigma\leq \pi$.
\end{remark}

\begin{corollary}\label{Corollary: Existence of moments}
For any $N \geq m/2-r+2$,
$$\alpha_{m_1,\dots,m_r}^{(N)}= \kern-1.5em
\sum_{\substack{\pi\in \cP(V) \\ \#(\pi)-m/2+r-2=0}}\sum_{\substack{\tau\in \cP(V) \\ \tau\vee\gamma =1_m}}\C_{\tau}(x_{i^{\pi}(1),i^{\pi}(\gamma(1))},\dots,x_{i^{\pi}(m),i^{\pi}(\gamma(m))})+o(1).$$
Consequently
$$\alpha_{m_1,\dots,m_r}= \kern-1em
\sum_{\substack{\pi\in \cP(V) \\ \#(\pi)-m/2+r-2=0}}\sum_{\substack{\tau\in \cP(V) \\ \tau\vee\gamma =1_m}}\C_{\tau}(x_{i^{\pi}(1),i^{\pi}(\gamma(1))},\dots,x_{i^{\pi}(m),i^{\pi}(\gamma(m))}).$$
\end{corollary}
\begin{proof}
The proof follows immediately from expression \ref{Equation: Third equality of the cumulant} and Lemma \ref{Lemma: First characterization of the limit}. We just take $N\geq m/2-r+2$ to guarantee that $q(N,\pi)$ is a polynomial on $1/N$ with leading coefficient $1$, otherwise $q(N,\pi)=0$.
\end{proof}

\section{Characterization of the leading order}\label{Section: Leading order}

In Lemma \ref{Lemma: First characterization of the limit} we characterized the leading term of $\alpha_{m_1,\dots,m_r}^{(N)}$ via the conditions $L_1,L_2,L_3$. In order to find the free cumulants $\K_{m_1,\dots,m_r}$ we would like to translate these conditions to something involving the set of non-crossing partitioned permutations. That will be the goal of this section. Let us remind that in Section \ref{Section: Asymptotics} we proved that for pair $(\pi,\tau)\in \cP(m)\times \cP(m)$ with $\C_{\tau}(\pi)\neq 0$ we have $\#(\pi)-m/2+r-2 \leq 0$ with equality if and only if there exist $\sigma\leq \tau$ that satisfies,
\begin{enumerate}[label=$($\subscript{L}{{\arabic*}}$)$]
    \item $\sigma\in \NC_2(m_1,\dots,m_r)\cup \NC_2^{nc}(m_1,\dots,m_r)$.
    \item For each block, $W$, of $\sigma\vee\gamma$, the partition $\pi|_W = \gamma|_W\sigma|_W$.
	\item The graph $T^{\pi}$ is the quotient graph along edges, all of them non-loops, of the graph $T^{\gamma\sigma}$.

    \item The graph $\overline{T^{\gamma\sigma}(\sigma,\tau)^{\overline{\pi_\tau}}}$ is a tree.
\end{enumerate}
For now on any triple $(\sigma,\tau,\pi)$ satisfying $L_1,L_2,L_3,L_4$ will be called a \textit{limit triple}. The limit triples characterize the moments $\alpha_{m_1,\dots,m_r}$. Let us now introduce the extra condition,

\begin{equation*}\tag{$L_4^\prime$}\label{Condition: L4 prime}
\text{The graph }T^{\gamma\sigma}(\sigma,\tau)\text{ is a tree}.
\end{equation*}
For a limit triple the condition \ref{Condition: L4 prime} might or might not be satisfied. So we incorporate the following notation.

\begin{notation}\label{Notation: Definition of types of limit triples}
\begin{enumerate}
    \item We say that the limit triple is of \textit{non-crossing type} if it satisfies $L_1,L_2,L_3,L_4,L_4^{\prime}$ and is of \textit{crossing type} if it only satisfies $L_1,L_2,L_3,L_4$.
    \item For $(\pi,\tau)\in \cP(m)\times \cP(m)$ with $\C_{\tau}(\pi)\neq 0$, we defined the disjoint sets,
    $$\mathcal{A}=\{(\pi,\tau): \exists \sigma 
        \text{ s.t. }(\sigma,\tau,\pi)\text{ is a non-crossing limit triple}\},$$
    and,
    \begin{multline*}
        \mathcal{B}=\{(\pi,\tau): \exists \sigma
        \text{ s.t. }(\sigma,\tau,\pi)\text{ is limit triple but there is no }\sigma \\ \text{ with }(\sigma,\tau,\pi)\text{ being a non-crossing limit triple}\}.
    \end{multline*}
\end{enumerate}
\end{notation}

An example of a limit triple $(\sigma,\tau,\pi)$ can be seen below in Example \ref{Example: Big example}.

\begin{example}\label{Example: Big example}
Let us consider $m_1=4$ and $m_2=m_3=m_4=m_5=m_6=m_7=2$. Let
$$\pi=\{1,12,14,16\},\{2,4,6,8,10,11,13,15\},\{3,5,7,9\},$$ 
so that 
$$\overline{\pi}=\{1,4,11,12,13,14,15,16\},\{2,3,5,6,7,8,9,10\}.$$
Let
$$\tau=\{1,11,13,16\},\{4,12,14,15\},\{2,5,7,10\},\{3,6,8,9\}.$$
See Figure \ref{Figure: Big example} par A) to see the quotient graph $T^{\pi}$. It can be easily verified that $\#(\pi)-m/2+r-2=0$ and $\C_{\tau}(\pi)\neq 0$. We take,
$$\sigma=(1,11)(13,16)(4,12)(14,15)(2,5)(7,10)(3,6)(8,9).$$
In Figure \ref{Figure: Big example} part C) we can see the graph $T^{\gamma\sigma}(\sigma,\tau)$. Inside each white vertex we draw the corresponding connected component of $T^{\gamma\sigma}$ so it can be seen that condition $L_2$ is satisfied. In Figure \ref{Figure: Big example} part D) we can see the graph $\Gamma(\tau,\sigma,\gamma)$ which is the same as $T^{\gamma\sigma}(\sigma,\tau)$ but the white vertices are indexed by the connected components of $\sigma\vee\gamma$ instead of the connected components of $T^{\gamma\sigma}$. Inside each white vertex we draw the corresponding $\sigma|_W$ for each block $W$ of $\sigma\vee\gamma$ so it can be seen that $\sigma\in \NC_2^{nc}(4,2,2,2,2,2,2)$. The condition $L_3$ is automatically satisfied as there are no loops in this example and therefore in this example $\overline{\pi}=\overline{\pi_\tau}$. Note that $L_4^\prime$ is not satisfied and it can be verified there is no $\sigma$ that satisfies $L_4^\prime$ so this correspond to a case where $(\pi,\tau)\in\mathcal{B}$. In Figure \ref{Figure: Big example} part E) we draw the graph $T^{\gamma\sigma}(\sigma,\tau)^{\overline{\pi_\tau}}$ which as it can be seen it is a tree after we forget multiplicity of the edges so that it satisfies condition $L_4$. In this last Figure we draw inside each white vertex the corresponding $\sigma|_W$ for each $W$ block of $\sigma\vee\gamma$ instead of the connected components of  $T^{\gamma\sigma}$.
\end{example}

\begin{figure}
    \centering
    \includegraphics[width=0.9\textwidth]{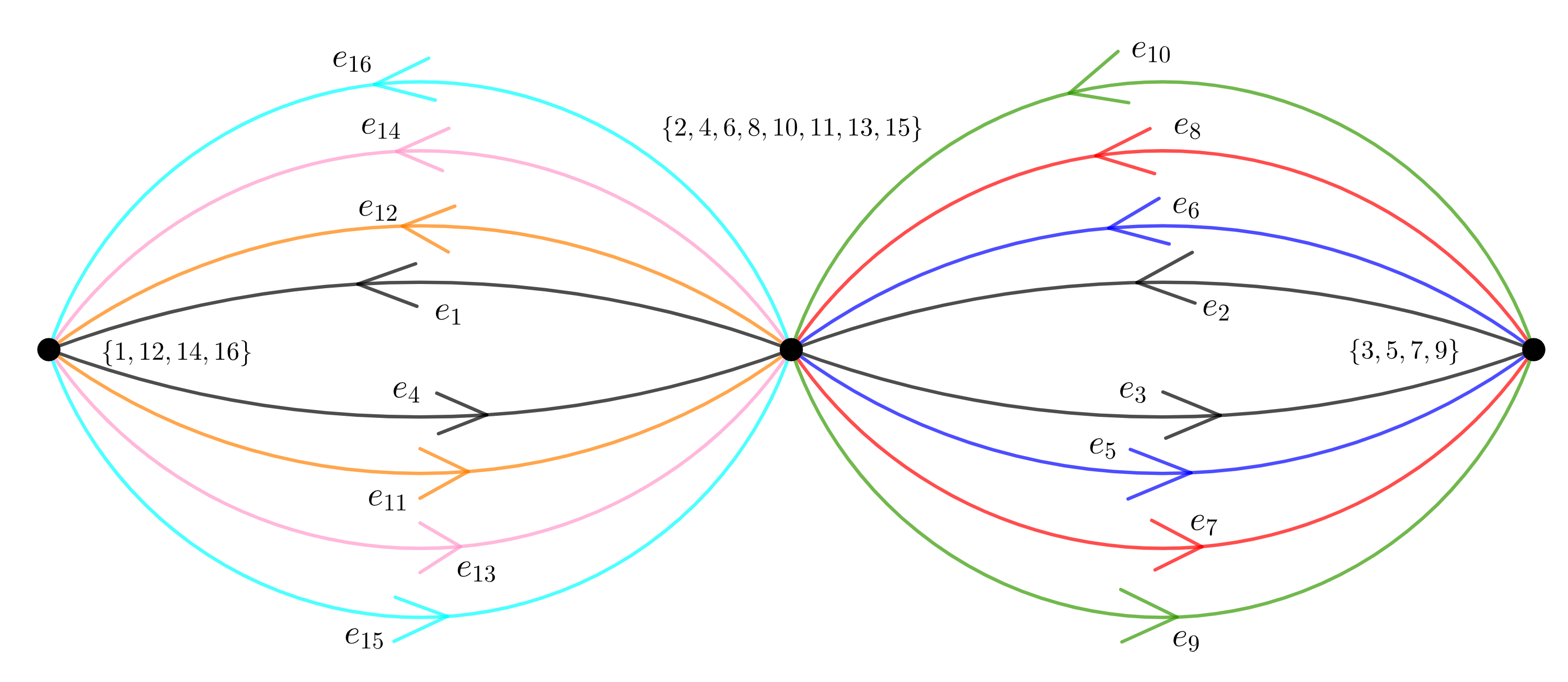}
    
A) The quotient graph $T^{\pi}$ corresponding to example \ref{Example: Big example}.

    \includegraphics[width=0.9\textwidth]{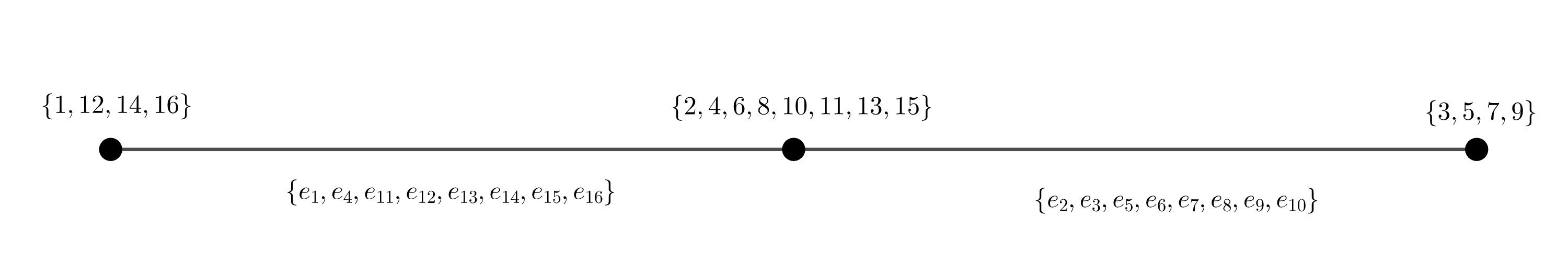}
    
B) The elementarization graph $\overline{T^{\pi}}$. The vertices are the blocks of $\pi$ while the edges are the equivalence classes $\overline{e_1}=\{e_1,e_4,e_{11},e_{12},e_{13}\ab ,e_{14},$ $e_{15},e_{16}\}$ and $\overline{e_2}=\{e_2,e_3,e_5,e_6,e_7,e_8,e_9,e_{10}\}$. Hence $\overline{\pi}=\{1,4,11,12,$
    \text{$13,14,15,16\},\{2,3,5,6,7,8,9,10\}$.}\\

    \includegraphics[width=0.9\textwidth]{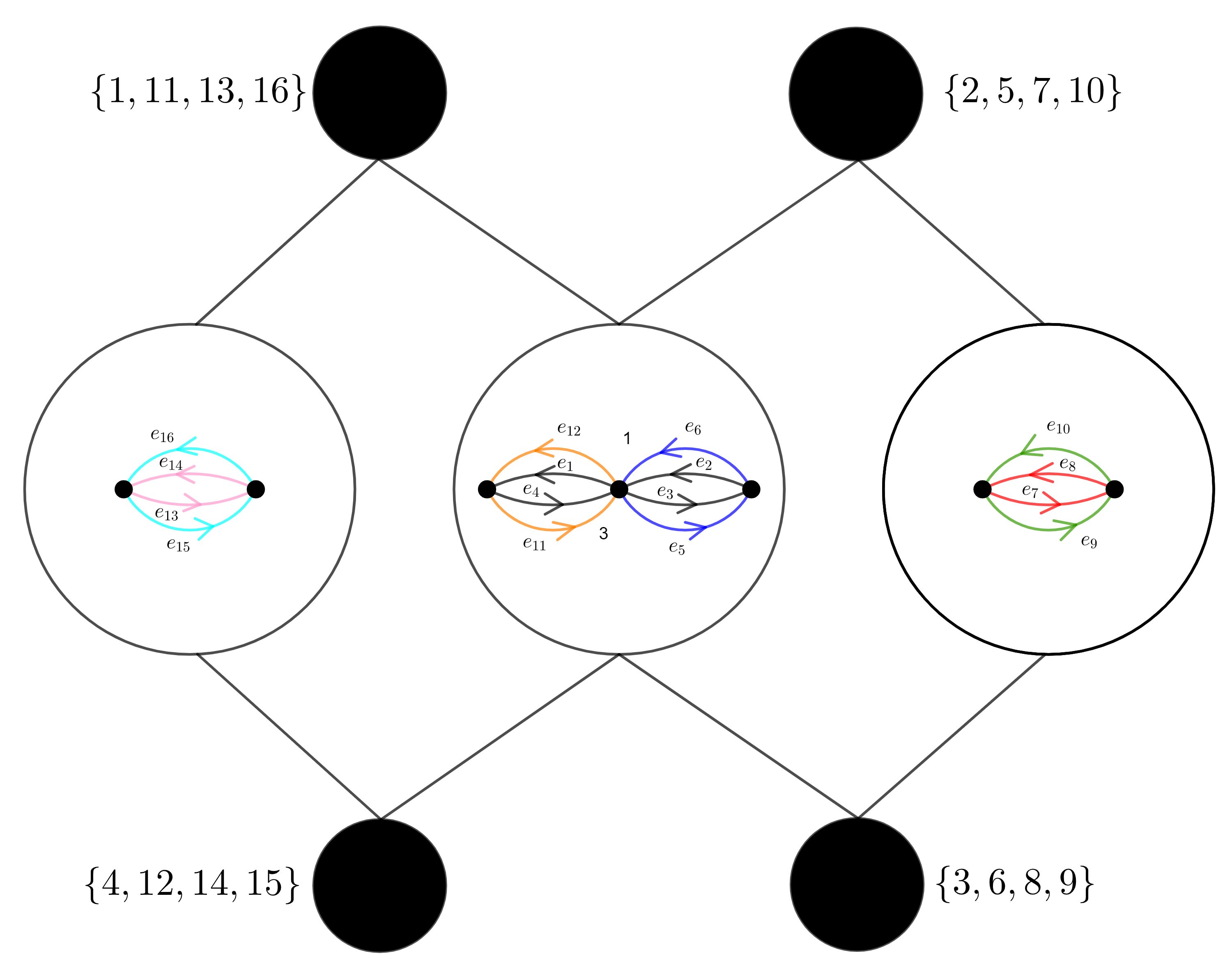}
    \text{C) The graph $T^{\gamma\sigma}(\sigma,\tau)$.}\\
    
    \caption{The graphs of Example \ref{Example: Big example}}
    \label{Figure: Big example}
\end{figure}

\addtocounter{figure}{-1}

\begin{figure}
    \centering
    \includegraphics[width=0.9\textwidth]{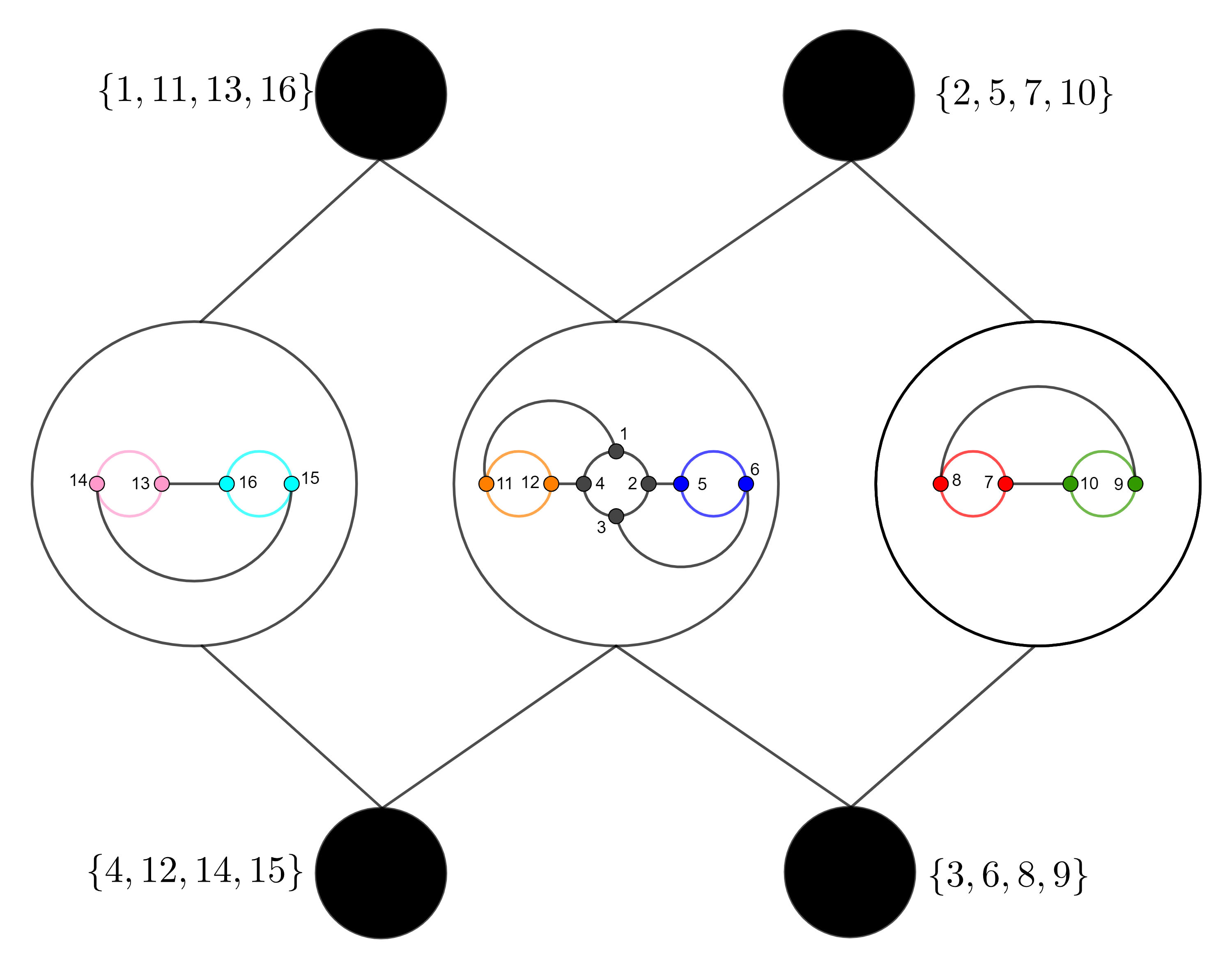}
    \text{D) The graph $\Gamma(\tau,\sigma,\gamma)$.}\\

    \includegraphics[width=0.9\textwidth]{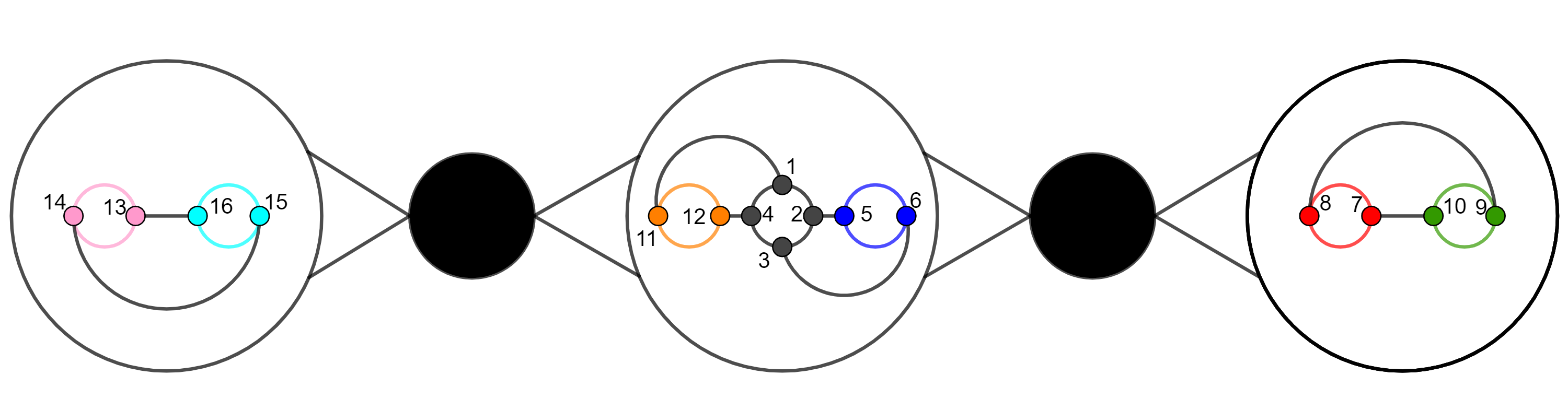}
    \text{E) The graph $T^{\gamma\sigma}(\sigma,\tau)^{\overline{\pi_\tau}}$.}\\
    
    \caption{The graphs of Example \ref{Example: Big example}}
\end{figure}

From Corollary \ref{Corollary: Existence of moments} and Lemma \ref{Lemma: First characterization of the limit} we get the following Proposition.

\begin{proposition}\label{Proposition: First expression of Alpha}
For any $m_1,\dots,m_r\in\mathbb{N}$,
$$\alpha_{m_1,\dots,m_r}=\sum_{(\pi,\tau)\in \mathcal{A}}\C_{\tau}(\pi)+\sum_{(\pi,\tau)\in \mathcal{B}}\C_{\tau}(\pi).$$
\end{proposition}

The motivation on calling a limit triple of crossing and non-crossing type will be explained during the rest of the paper. However as a short motivation we may think of the non-crossing limit triples as the ones that are characterized by the non-crossing partitioned permutations.

\subsection{Non-crossing limit triples}\label{Subsection: Non-crossing limit triples}

\begin{definition}\label{Definition: Non-crossing partitioned permutations loop free}
Let $\sigma\in \cP_2(m)$. For a block $B\in \sigma$ we say that $B$ is a \textit{loop block} if $e_B^{\gamma\sigma}$ is a loop of $T^{\gamma\sigma}$. Here by $e_B^{\gamma\sigma}$ we mean the one defined as in notation \ref{Notation: Edge associated to block} which is well defined thanks to Lemma \ref{Lemma: Blocks of sigma determined edges of T^gamma sigma}. For a pair $(\tau,\sigma)\in \cP(m)\times \cP_2(m)$ we say that  the pair is \textit{loop free} if $\sigma\leq \tau$ and if $B$ is a loop block of $\sigma$ then $B$ is also a block of $\tau$. We denote by $\PS_{NC_2}^{loop-free}(m_1,\dots,m_r)$ to all the partitioned permutations $(\tau,\sigma)\in \PS_{NC_2}(m_1,\dots,m_r)$ where $(\tau,\sigma)$ is loop free.
\end{definition}

Our goal is to prove that each $(\tau,\sigma)\in \PS_{NC_2}^{loop-free}(m_1,\dots,m_r)$ corresponds to exactly $2^{\#(\sigma)-\#(\tau)}$ elements in $\mathcal{A}$ and conversely. Before doing this let us prove the following Proposition that characterizes a loop block.

\begin{definition}\label{Definition: Cutting through string}
Let $\sigma\in \NC_2(m_1,\dots,m_r)$ and let $B=\{u,v\}\in \sigma$ be a through string. Let $\sigma^\prime \in S_m$ be the permutation $\sigma^\prime =\sigma(u,v)$. We say that $B$ is \textit{cutting} if $\#(\gamma\vee\sigma^\prime)=2$. Equivalently removing the cycle $B$ disconnect $\gamma\vee\sigma$ into two blocks. See Figure \ref{Figure: Cutting string} for an example.
\end{definition}

\begin{figure}
    \centering
    \includegraphics[width=0.9\textwidth]{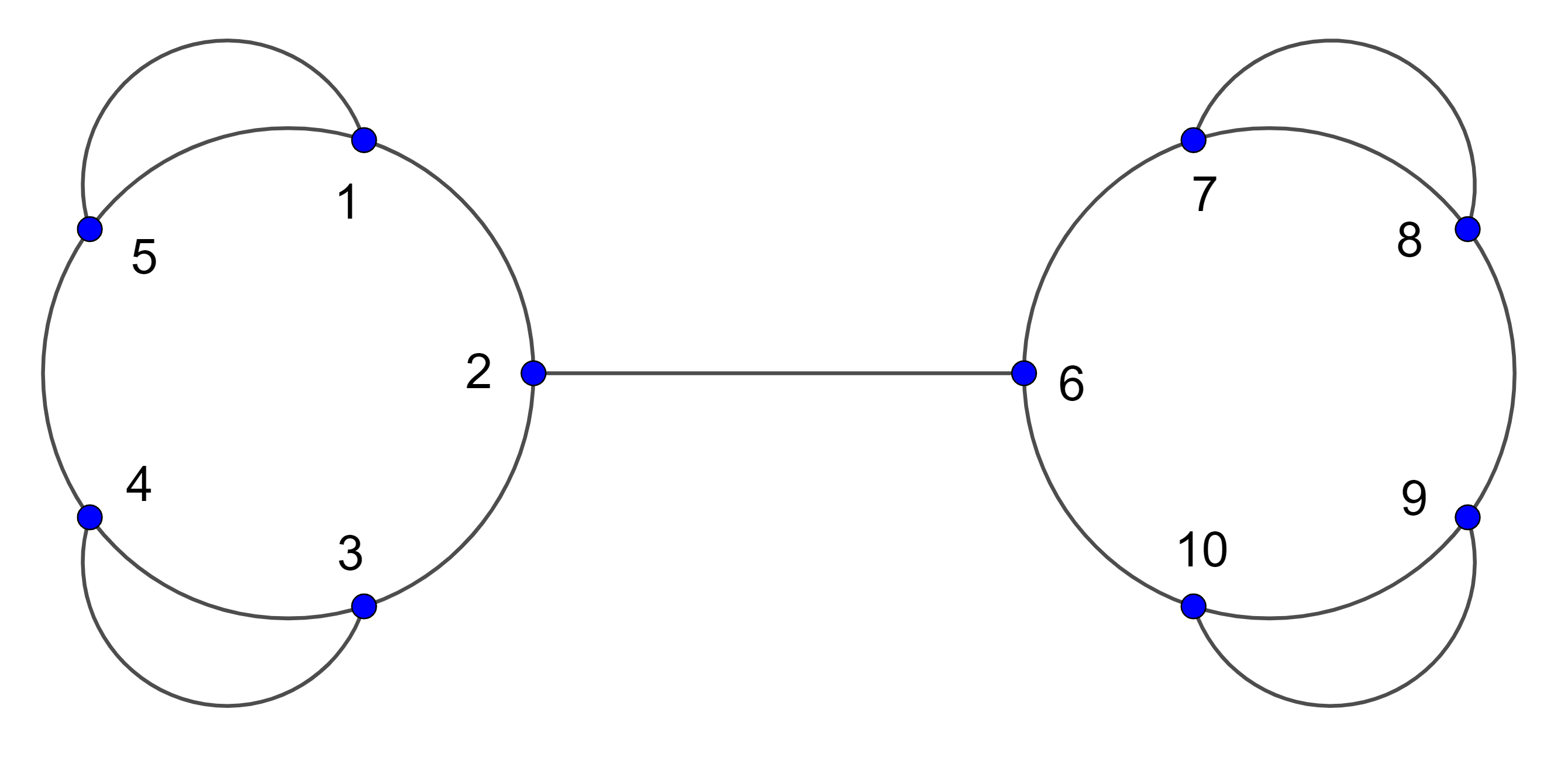}
    \caption{A non-crossing pairing $\sigma\in \NC_2(5,5)$ corresponding to $\sigma=(1,5)(2,6)(3,4)(7,8)(9,10)$ where the through string $(2,6)$ is cutting.}
    \label{Figure: Cutting string}
\end{figure}

\begin{remark}
Under the notation of Definition \ref{Definition: Cutting through string} it can be proved that $\sigma^\prime$ is still non-crossing. The proof consists of a simple calculation of the equality in Equation \ref{Inequality: Mingo and Nica inequality}.
\end{remark}

\begin{proposition}\label{Proposition: loop edge iff cutting through string}
Let $\sigma\in \NC_2(m_1,\dots,m_r)$ and let $B\in \sigma$ be a through string of $\sigma$. Then $B$ is cutting if and only if $B$ is a loop block.
\end{proposition}
\begin{proof}
Assume first that $B$ is cutting then $\#(\gamma\vee\sigma^\prime)=2$. Let $C$ and $D$ be the blocks of $\gamma\vee\sigma^\prime$, one of the blocks contains $u$ and the other contains $v$ otherwise $C$ and $D$ are also blocks of $\gamma\vee\sigma$ which is not possible. If $u$ and $v$ are in the same block of $\gamma\sigma^\prime$ then so are in the same block of $\gamma\vee\gamma\sigma^\prime$. However $\gamma\vee\gamma\sigma^\prime = \gamma\vee\sigma^\prime$, therefore $u$ and $v$ would be in the same block of $\gamma\vee\sigma^\prime$ which is a contradiction so it must be $u$ and $v$ are in disjoint blocks of $\gamma\sigma^\prime$ and then $u$ and $v$ are in the same block of $\gamma\sigma = \gamma\sigma^\prime(a,b)$. The latest means $e_B^{\gamma\sigma}$ is a loop of $T^{\gamma\sigma}$. Conversely suppose $e_B^{\gamma\sigma}$ is a loop of $T^{\gamma\sigma}$ then $u,v,\gamma(u)$ and $\gamma(v)$ are all in the same block of $\gamma\sigma$. Thus,
\begin{eqnarray*}
m+2\#(\gamma\vee\sigma^\prime) &\geq & \#(\sigma^\prime)+\#(\gamma)+ \#(\gamma\sigma^\prime) \\
&=& \#(\sigma)+1+\#(\gamma)+\#(\gamma\sigma(u,v)) \\
&=& \#(\sigma)+1+\#(\gamma)+\#(\gamma\sigma)+1 \\
&=& \#(\sigma)+\#(\gamma)+\#(\gamma\sigma)+2= m+4.
\end{eqnarray*}
Hence, $\#(\gamma\vee\sigma^\prime) \geq 2$ which forces $\#(\gamma\vee\sigma^\prime)=2$.
\end{proof}

\begin{lemma}\label{Lemma: Sigma is unique is non-crossing partitioned permutation free-loop type}
Let $(\tau,\sigma_1)$ and $(\tau,\sigma_2)$ be two non-crossing partitioned permutations such that,
$$(\tau,\sigma_1),(\tau,\sigma_2)\in \PS_{NC_2}^{loop-free}(m_1,\dots,m_r).$$
Then $\sigma_1=\sigma_2$.
\end{lemma}
\begin{proof}
We prove this by induction over $\#(\gamma\vee\sigma_1)$. Throughout the proof we will have in mind that both $\Gamma(\tau,\sigma_1)$ and $\Gamma(\tau,\sigma_2)$ are trees thanks to Lemma \ref{Lemma: Tree iff Non-crossing partitioned permutation}. If $\#(\gamma\vee\sigma_1)=1$ then any block of $\sigma$ is also a block of $\tau$ otherwise $\Gamma(\tau,\sigma)$ is not a tree. Since $\sigma_2\leq \tau$ then any block of $\sigma_2$ must also be a block of $\sigma_1$ hence $\sigma_1=\sigma_2$. Suppose it is true for $\#(\gamma\vee\sigma_1)=n$ then we prove it for $\#(\gamma\vee\sigma_1)=n+1$. Let $\Gamma^\prime$ be the graph $\Gamma(\tau,\sigma)$ after we remove all leaves (and their adjacent edges) that correspond to blocks of $\tau$ that are also blocks of $\sigma$. $\Gamma^\prime$ is also a tree so it must have a leaf which is a white vertex corresponding to a block of $\gamma\vee\sigma_1$. Let $G$ be this white vertex. Let $D$ be the unique black vertex that is adjacent to $G$ and let $B$ the edge that connects $G$ and $D$. The edge $B$ correspond to a block of $\sigma_1$, moreover $B$ is the only block of $\sigma_1$ contained in $G$ that is not a block of $\tau$. Any other block of $\sigma_1$ contained in $G$ is also a block of $\tau$ and hence a block of $\sigma_2$. We write $B=\{u,v\}$, we claim that $\{u,v\}$ must be also a block of $\sigma_2$. Indeed, suppose $\{u,v\}$ is not a block of $\sigma_2$ then there exist $a,b\in D$ such that $\{a,u\}$ and $\{b,v\}$ are blocks of $\sigma_2$. By Proposition \ref{Proposition: loop edge iff cutting through string} we know that removing $(u,v)$ from $\sigma_1$ doesn't disconnect the block $G$, this means that if $\sigma_1^\prime$ is the permutation whose cycles are the cycles of $\sigma_1$ contained in $G$ except $(u,v)$ which are both cycles of size $1$ and $\gamma^\prime$ is the permutation consisting on the cycles of $\gamma$ contained in $G$ then $\gamma^\prime\vee\sigma_1^\prime=G$. However we said that any other block of $\sigma_1$ distinct from $(u,v)$ contained in $G$ is also a block of $\sigma_2$ then $G$ must be contained in the connected component of $\gamma\vee\sigma_2$ that contains $u$ and $v$ which we call $G_2$. This means that $G_2$ is connected to the black vertex $D$ by the edges $\{a,u\}$ and $\{b,v\}$ which is not possible as $\Gamma(\tau,\sigma_2)$ is a tree. Then it must be $\{u,v\}$ is a block of $\sigma_2$. The latest also proves $G$ is a block of $\gamma\vee\sigma_2$. 

Let $C=[m]\setminus G$. Let $\hat{\gamma}=\gamma|_C$ which consists on the cycles of $\gamma$ that are not in $G$. Let $\hat{\sigma_i}=\sigma_i|_C$ for $i=1,2$ and let $\hat{\tau}=\tau|_C$. Let $\Gamma(\hat{\tau},\hat{\sigma_1},\gamma|_C)$ be defined as in Remark \ref{Remark and definition: The graph Gamma(V,pi)}. Note that the graph $\Gamma(\hat{\tau},\hat{\sigma_1})$ is precisely the graph $\Gamma(\tau,\sigma_1)$ but removing the white vertex $G$ and all its adjacent edges and vertices except $D$ which becomes $D\setminus \{u,v\}$. This means $\Gamma(\hat{\tau},\hat{\sigma_1})$ is a tree and hence by Lemma \ref{Lemma: Tree iff Non-crossing partitioned permutation} we have that $(\hat{\tau},\hat{\sigma_1})$ is a non-crossing partitioned permutation with respect to $(1_C,\gamma|_C)$. Moreover if $B$ is block of $\hat{\sigma_1}$ that is also a block of $\hat{\tau}$ then $B$ is a block of $\sigma_1$ that is also a block of $\tau$ and therefore it is a loop block of $T^{\gamma\sigma_1}$. Therefore, $B$ must also be a loop block of $T|_C^{\hat{\gamma}\hat{\sigma_1}}$, where $T|_C$ denotes the restriction of $T$ to the cycles of $\hat{\gamma}$. The latest follows from the fact that $T^{\gamma\sigma_1}$ and $T|_C^{\hat{\gamma}\hat{\sigma_1}}$ are the same except by the connected component $T|_G^{\gamma|_G{\sigma_1}|_G}$ which is in the former but not in the latter graph. That proves that,
$$(\hat{\tau},\hat{\sigma_1})\in \PS_{\NC_2}^{loop-free}(1_C,\gamma|_C).$$
Analogously we conclude,
$$(\hat{\tau},\hat{\sigma_2})\in \PS_{\NC_2}^{loop-free}(1_C,\gamma|_C).$$
But $\#(\gamma|_C\vee\hat{\sigma_1})=n$ as we are just removing the block $G$ of $\gamma\vee\sigma_1$. Then by induction hypothesis $\hat{\sigma_1}=\hat{\sigma_2}$. This proves $\sigma_1=\sigma_2$ as it was proved before that the restriction of $\sigma_1$ and $\sigma_2$ to $G$ are the same.
\end{proof}

\begin{notation}\label{Notation: Wideltilde partition}
Let $\sigma\in S_{\NC}(\gamma)$ be a non-crossing pairing of $\gamma$ and let $T^{\gamma\sigma}= \ab (V^{\gamma\sigma},E^{\gamma\sigma},\delta)$ be the quotient of $T$ under $\gamma\sigma$. Let $\rho\in \cP(m)$ be a partition such that $\overline{\gamma\sigma}\leq \rho$. We define the partition $\widetilde{\rho}\in \cP(\overline{E^{\gamma\sigma}})$ given by $\overline{e_u^{\gamma\sigma}}$ and $\overline{e_v^{\gamma\sigma}}$ are in the same block of $\widetilde{\rho}$ if and only if $e_u^{\gamma\sigma}$ and $e_v^{\gamma\sigma}$ are in the same block of $e_\rho^{\gamma\sigma}$. For an example see \ref{Example: example of rho sim}.
\end{notation}

\begin{example}\label{Example: example of rho sim}
Let $m_1,m_2,m_3,m_4,m_5,m_6,m_7$ and $\sigma$ be as in Example \ref{Example: Big example}. We have that,
$$\gamma\sigma=(1,12)(2,6,4,11)(3,5)(7,9)(8,10)(13,15)(14,16),$$
and,
$$\overline{\gamma\sigma}=\{1,4,11,12\},\{2,3,5,6\},\{7,8,9,10\},\{13,14,15,16\}.$$
Let $\rho\in \cP(16)$ with blocks,
$$\{1,2,3,4,5,6,11,12\},\{7,8,9,10\},\{13,14,15,16\}.$$
Then $\widetilde{\rho}$ is the partition with blocks,
$$\{\overline{e_1^{\gamma\sigma}},\overline{e_2^{\gamma\sigma}}\},\{\overline{e_7^{\gamma\sigma}}\},\{\overline{e_{13}^{\gamma\sigma}}\}.$$
\end{example}

\begin{remark}
In the context of Notation \ref{Notation: Wideltilde partition} observe that $\#(\rho)=\#(\widetilde{\rho})$.
\end{remark}

We are ready to prove the main goal of this subsection.

\begin{proposition}\label{Proposition: first contribution of alpha}
For any $m_1,\dots,m_r\in\mathbb{N}$,
$$\sum_{(\tau,\pi)\in \mathcal{A}}\C_{\tau}(\pi)=\sum_{(\tau,\sigma)\in \PS_{NC_2}^{loop-free}(m_1,\dots,m_r)}2^{\#(\sigma)-\#(\tau)}\prod_{B\in\tau}\beta_{|B|}.$$
\end{proposition}

\begin{proof}

The proof of this proposition is a bit long so we divide it into steps to make simpler for the reader.

\textbf{Step 1. We prove that each $(\tau,\sigma)\in \PS_{\NC_2}^{loop-free}(m_1,\dots,m_r)$ corresponds to $2^{\#(\sigma)-\#(\tau)}$ pairs $(\pi,\tau)\in \mathcal{A}$.}

Let $(\tau,\sigma)\in \PS_{NC_2}^{loop-free}(m_1,\dots,m_r)$
and let $G=T^{\gamma\sigma}= \ab (V^{\gamma\sigma},E^{\gamma\sigma}\kern-0.25em, \delta).$ We wan to construct an oriented partition of the edges $\overline{E^{\gamma\sigma}}$ determined by $(\tau,\sigma)$. If two blocks $B_1,B_2\in \sigma$ are such that $e_{B_1}^{\gamma\sigma}$ and $e_{B_2}^{\gamma\sigma}$ connect the same pair of vertices of $T^{\gamma\sigma}$ then $B_1$ and $B_2$ are both edges of $T^{\gamma\sigma}(\sigma,\tau)$ that are adjacent to the same white vertex. Let $D_1$ and $D_2$ the black vertices of $T^{\gamma\sigma}(\sigma,\tau)$ that are adjacent to $B_1$ and $B_2$ respectively. $D_1$ and $D_2$ correspond to blocks of $\tau$, which we merge. Let $\hat{\tau}$ be the resulting partition after merging all blocks of $\tau$ described as before. The resulting partition satisfies $\overline{\gamma\sigma} \leq \hat{\tau}$ so it makes sense to consider $\widetilde{\tau} \vcentcolon = \widetilde{\hat{\tau}}$ as in notation \ref{Notation: Wideltilde partition}. 

On the other hand, observe that $\Gamma(\tau,\sigma,\gamma)$ is a tree by Lemma \ref{Lemma: Tree iff Non-crossing partitioned permutation} and then so is $G(\sigma,\tau)$ as both graphs are isomorphic with the only difference that the white vertices in the former are blocks of $\gamma\vee\sigma$ and the white vertices of the latter are connected components of $T^{\gamma\sigma}$ which are exactly the same as each block of $\gamma\vee\sigma$ determines a connected components of $T^{\gamma\sigma}$ and conversely (see for example Figure \ref{Figure: Big example} part C) and D). Each time we merge two blocks of $\tau$ into the same block of $\hat{\tau}$ we join two black vertices of $T^{\gamma\sigma}(\sigma,\tau)$ and the edges $B_1$ and $B_2$ are identified, therefore, $\overline{T^{\gamma\sigma}(\sigma,\tau)^{\hat{\tau}}}$ is still a tree. In the graph $\overline{T^{\gamma\sigma}(\sigma,\tau)^{\hat{\tau}}}$ we have exactly one edge $\overline{e_B^{\gamma\sigma}}$ for each block $B$ of $\overline{\gamma\sigma}$. Thus $T^{\gamma\sigma}(\tilde{\tau})=\overline{T^{\gamma\sigma}(\sigma,\tau)^{\hat{\tau}}}$.


Since $\overline{T^{\gamma\sigma}(\sigma,\tau)^{\hat{\tau}}}$ is a tree then any block $\{\overline{e_{B_1}^{\gamma\sigma}},\dots,\overline{e_{B_n}^{\gamma\sigma}}\}$ of $\widetilde{\tau}$ with $n>1$ must be such that $\overline{e_{B_i}^{\gamma\sigma}}$ and $\overline{e_{B_j}^{\gamma\sigma}}$ belong to distinct connected components of $\overline{T^{\gamma\sigma}}$. And since any block of $\tau$ that contains more than one block of $\sigma$ is such that $e_B^{\gamma\sigma}$ is not a loop of $T^{\gamma\sigma}$ for $B\in \sigma$ then $\{\overline{e_{B_1}^{\gamma\sigma}},\dots,\overline{e_{B_n}^{\gamma\sigma}}\}$ are all non-loops of $\overline{T^{\gamma\sigma}}$. The latter means that for each block of $\widetilde{\tau}$ of the form $\{\overline{e_{B_1}^{\gamma\sigma}},\dots,\overline{e_{B_n}^{\gamma\sigma}}\}$ we can choose $2^{n-1}$ distinct possible orientation of the edges (by choosing the $2^{n-1}$ distinct pair of sets $A$ and $B$ as in Definition \ref{Definition: Identification of edges}). We conclude that for the partition of the edges $\widetilde{\tau}\in \cP(\overline{E^{\gamma\sigma}})$ we can choose as many distinct orientations as,
$$\prod_{B\in\widetilde{\tau}}2^{\#(\text{incident edges of }B)-1}.$$
Where in above product we consider even the blocks of $\widetilde{\tau}$ of size $1$ as $2^{1-1}=1$. The total number of edges is precisely $\#(\overline{\gamma\sigma})$ so the total number of distinct orientations is,
$$2^{\#(\overline{\gamma\sigma})-\#(\widetilde{\tau})}=2^{\#(\overline{\gamma\sigma})-\#(\hat{\tau})}.$$
Finally as observed before every time we merge two blocks of $\tau$ into the same block of $\hat{\tau}$ we are identifying exactly two edges, thus,
$$2^{\#(\overline{\gamma\sigma})-\#(\hat{\tau})}=2^{\#(\sigma)-\#(\tau)}.$$

Let $\pi\in \cP(V^{\gamma\sigma})$ be defined as $\pi=\underset{\sim_p}{\widetilde{\tau}}$ for some $\sim_p$ of the $2^{\#(\sigma)-\#(\tau)}$ possible orientation of the edges. Each partition $\pi\in \cP(V^{\gamma\sigma})$ determines a partition $\pi^\prime\in\cP(m)$ by letting $A_1\cup \cdots \cup A_n$ be a block of $\pi^\prime$ whenever $A_1,\dots,A_n$ are in the same block of $\pi$. Each distinct choice for the orientation of the edges produces a distinct partition $\pi^\prime$ as the identification of the vertices is distinct. As an abbuse of notation, let us call to $\pi^\prime$ simply $\pi$.

We claim that $(\pi,\tau)\in \mathcal{A}$. To prove this it is enough to check that $(\sigma,\tau,\pi)$ satisfies $L_1,L_2,L_3$ and $L_4$ as $L_4^\prime$ is is already satisfied. $L_1$ follows immediately as $\sigma$ is non-crossing by hypothesis. Let $W$ be a block of $\gamma\vee\sigma$, Let us recall that $\pi$ is obtained after we merge blocks of $T^{\gamma\sigma}$, therefore by construction $\pi|_W$ is the partition $\gamma|_W\sigma|_W$, thus $L_2$. Condition $L_3$ follows immediately as by construction we only identified edges that are non-loops. Now we verify condition $L_4$. First of all note that $T^{\gamma\sigma}(\tilde{\tau})=\overline{T^{\gamma\sigma}(\sigma,\tau)^{\hat{\tau}}}$ because when we take elementarization of the graph we take away multiplicity of the edges and both $\tilde{\tau}$ and $\hat{\tau}$ have the same number of blocks. As $\overline{T^{\gamma\sigma}(\sigma,\tau)^{\hat{\tau}}}$ is a tree then so is $T^{\gamma\sigma}(\tilde{\tau})$ so from Lemma \ref{Lemma: Equality when tree} part $(ii)$ it follows that for any $e_u^{\gamma\sigma}$ and $e_v^{\gamma\sigma}$ no loops of $T^{\gamma\sigma}$ it is satisfied that $e_u^{\pi}$ and $e_v^{\pi}$ connect the same pair of vertices of $T^{\pi}$ if and only if $\overline{e_u^{\gamma\sigma}}$ and $\overline{e_u^{\gamma\sigma}}$ are in the same block of $\tilde{\tau}$. The latest means the blocks of $\hat{\tau}$ (except the blocks that are union of blocks, $B$, of $\tau$ that are also blocks of $\sigma$ and that $e_B^{\gamma\sigma}$ correspond to loops) coincide with the blocks of $\overline{\pi_\tau}$. In other words, we have that $T^{\gamma\sigma}(\sigma,\tau)^{\overline{\pi_\tau}}$ is the graph $T^{\gamma\sigma}(\sigma,\tau)^{\hat{\tau}}$ with the only difference that in the former graph the blocks, $B$, of $\tau$ that are also blocks of $\sigma$ and that $e_B^{\gamma\sigma}$ correspond to loops are all leaves while in the later graph we merge some of these leaves adjacent to the same white vertex into a single leaf with multiple edges. But we just said $\overline{T^{\gamma\sigma}(\sigma,\tau)^{\hat{\tau}}}$ is a tree and then so is $\overline{T^{\gamma\sigma}(\sigma,\tau)^{\overline{\pi_\tau}}}$. This finishes the first step. Before moving on let us prove that $\#(\pi^\prime)-m/2+r-2=0$, this is not required in this step of the proof however it will be invoke later on. This fact follows directly from Lemma \ref{Lemma: Equality when tree} part $(i)$,

$$\#(\pi^\prime)=\#(\pi)=\#(\gamma\sigma)-2(\#(\gamma\vee\sigma)-1),$$

then $\#(\pi^\prime)-m/2+r-2=0$ since $\sigma$ is non-crossing permutation.

\textbf{Step 2. We prove that for each pair $(\pi,\tau)\in A$ there exist a unique $(\tau,\sigma)\in \PS_{\NC_2}^{loop-free}(m_1,\dots,m_r)$ so that one of the $2^{\#(\sigma)-\#(\tau)}$ pairs in $\mathcal{A}$ that corresponds to $(\tau,\sigma)$ is precisely $(\pi,\tau)$.}

Let $(\pi,\tau)\in\mathcal{A}$ then there exist a non-crossing pairing $\sigma$ such that $(\sigma,\tau,\pi)$ is a non-crossing limit triple. The graph $T^{\gamma\sigma}(\sigma,\tau)$ is a tree and then so is $\Gamma(\tau,\sigma)$, therefore by Lemma \ref{Lemma: Tree iff Non-crossing partitioned permutation} we have $(\tau,\sigma)\in \PS_{\NC_2}(m_1,\dots,m_r)$. Moreover if $B=\{u,v\}\in \sigma$ is such that $e_B^{\gamma\sigma}$ is a loop of $T^{\gamma\sigma}$ then it is also a loop of $T^{\pi}$. If $B$ is not a block of $\tau$ then $\C_{\tau}(\pi)=0$ as it has a factor of the form $\C_{n}(x_{1,1},\dots,x_{1,1})$ with $n\geq 4$ which is $0$. Thus $B$ must be a block of $\tau$, i.e. $(\tau,\sigma)\in \PS_{\NC_2}^{loop-free}(m_1,\dots,m_r)$, moreover, $\sigma$ is unique thanks to Lemma \ref{Lemma: Sigma is unique is non-crossing partitioned permutation free-loop type}.

We aim to prove that under the construction done before with $(\tau,\sigma)$ we obtain $\pi$ for some of the $2^{\#(\sigma)-\#(\tau)}$ distinct orientation of the edges. By hypothesis conditions $L_1,L_2,L_3,L_4,L_4^\prime$ are satisfied. Let $\hat{\tau}\in \cP(m)$ and $\widetilde{\tau}\in \cP(\overline{E^{\gamma\sigma}})$ be as before and let $\widetilde{\overline{\pi}}\in \cP(\overline{E^{\gamma\sigma}})$ defined as in Notation \ref{Notation: Wideltilde partition} which makes sense as $\gamma\sigma\leq \pi$ and then $\overline{\gamma\sigma}\leq \overline{\pi}$. If two blocks of $\sigma$, say $B_1$ and $B_2$ that are such that $e_{B_1}^{\gamma\sigma}$ and $e_{B_2}^{\gamma\sigma}$ connect the same pair of vertices of $T^{\gamma\sigma}$ then they connect the same pair of vertices of $T^{\pi}$ as $\gamma\sigma\leq \pi$. The latest means that if two blocks of $\tau$ are joined into the same block of $\hat{\tau}$ then they must be joined into the same block of $\overline{\pi}$. Therefore, $\hat{\tau}\leq \overline{\pi}$, thus $\widetilde{\tau}\leq \widetilde{\overline{\pi}}$. By $L_2$ we know $\pi|_W=\gamma|_W\sigma|_W$ for each block of $\gamma\vee\sigma$ then it follows that for a suitable orientation of the edges, $\sim_p$, determined by $\pi$ we have $\underset{\sim_p}{\widetilde{\overline{\pi}}}=\pi$ as $\widetilde{\overline{\pi}}$ precisely identifies the edges of $T^{\gamma\sigma}$ that are identified in $\overline{\pi}$. Let $\sim_1$ be such that $(\widetilde{\tau},\sim_1)\leq (\sim_p,\widetilde{\overline{\pi}})$ which can be chosen by just taking the restriction of $\sim_p$ to each block of $\widetilde{\tau}$. Therefore, $\underset{\sim_1}{\widetilde{\tau}}\leq \underset{\sim_p}{\widetilde{\overline{\pi}}}=\pi$. Let $\pi^\prime$ be the partition obtained from $(\widetilde{\tau},\sim_1)$ as in the first step, thus $\pi^\prime\leq \pi$. We proved at the end of the previous step that $\#(\pi^\prime)-m/2+r-2=0$. However $\#(\pi)-m/2+r-2=0$ so it must be $\pi^\prime=\pi$ as desired.

\textbf{Step 3. We conclude.}
From Proposition \ref{Proposition: First expression of Alpha} and the proved in the first two steps we can change the first sum indexed by $\mathcal{A}$ by a sum indexed by $\PS_{NC_2}^{loop-free}(m_1,\ab\dots,m_r)$ multiplying each element of the second set by $2^{\#(\sigma)-\#(\tau)}$. On the other hand, if $(\pi,\tau)\in \mathcal{A}$ and it correspond to $(\sigma,\tau,\pi)\in \PS_{NC_2}^{loop-free}\ab(m_1,\dots,m_r)$ then the contribution $\C_{\tau}(\pi)$ is given by,
$$\C_{\tau}(\pi)=\prod_{B\in\tau}\beta_{|B|},$$
because for each block $B$ of $\tau$ that consist of $\{u_1,v_1\},\dots,\{u_1,v_n\}$ blocks of $\sigma$ with $n>1$ we have that the edges $e_{u_1},\dots,e_{u_n},e_{v_1},\dots,e_{v_n}$ are all joining the same pair of vertices of $T^{\pi}$ with the same number of edges in each orientation, so the contribution of this block is,
$$\C_{|B|}(x_{1,2},x_{2,1},\dots,x_{1,2},x_{2,1})=\beta_{|B|}.$$
Note that we can assume the pair of vertices that the edges join are distinct as otherwise we have $\C_{|B|}(x_{1,1},\dots,x_{1,1})=0$ and the so is $\C_{\tau}(\pi)$. For the block of $\tau$ with size $2$ we simply have a contribution of the form $\C_2(x_{1,2},x_{2,1})$ or $\C_2(x_{1,1},x_{1,1})$, in either case equals $\beta_2$.

\end{proof}

\begin{corollary}\label{Corollary: First expression of cumulants for non-crossing objects only}
For any $m_1,\dots,m_r$,
$$\alpha_{m_1,\dots,m_r}=\sum_{(\tau,\sigma)\in \PS_{NC_2}^{loop-free}(m_1,\dots,m_r)} \widetilde{\K}_{(\tau,\sigma)} +\sum_{(\tau,\pi)\in \mathcal{B}}\C_{\tau}(\pi).$$
Where $\widetilde{\K}_{(\tau,\sigma)}$ is the multiplicative extension of $(\widetilde{\K}_{n_1,\dots,n_r})_{n_1,\dots,n_r}$ defined as,
\begin{equation}
\widetilde{\K}_{n_1,\dots,n_r}= \left\{ \begin{array}{lcc} 
2^{r-1}\beta_{2r} & if & r\geq 1 \text{ and } n_1=\cdots = n_r=2 \\ \\ 
0 & otherwise & 
\end{array} 
\right.
\end{equation}
\end{corollary}
\begin{proof}
It is enough to note that for $(\tau,\sigma)\in \PS_{\NC_2}^{loop-free}(m_1,\dots,m_r)$ we have,
$$\widetilde{\K}_{(\tau,\sigma)}=2^{\#(\sigma)-\#(\tau)}\prod_{B\in\tau}\beta_{|B|}.$$
\end{proof}

\subsection{Crossing limit triples}\label{Subsection: Crossing limit triples}

Now we would like to study the crossing-limit triples. For a limit-triple $(\sigma,\tau,\pi)$ and a block $B\in \overline{\pi_\tau}$ we denote by $T^{\gamma\sigma}(\sigma,\tau)_B$ to the restriction of the graph $T^{\gamma\sigma}(\sigma,\tau)$ to the set of black vertices $B^\prime\in \tau$ with $B^\prime \subset B$ and their adjacent edges and white vertices. Observe that this makes sense as $\tau\leq\overline{\pi}$ and therefore each block of $\overline{\pi_\tau}$ is a union of blocks of $\tau$ (where the union might consist of a single block). Intuitively, the graph $T^{\gamma\sigma}(\sigma,\tau)$ is a tree if and only if each subgraph $T^{\gamma\sigma}(\sigma,\tau)_B$ is a tree for any block $B\in\overline{\pi_\tau}$. We aim to prove this and therefore to study the whole graph $T^{\gamma\sigma}(\sigma,\tau)$ it will be enough to study each graph $T^{\gamma\sigma}(\sigma,\tau)_B$ separately.

\begin{lemma}\label{Lemma: Splitting edges produces two new non-crossing pairings}
Let $\sigma\in \NC_2(m_1,\dots,m_r)$ and let $(a,b),(c,d)\in\sigma$ be through strings of $\sigma$ such that $e_a^{\gamma\sigma}$ and $e_c^{\gamma\sigma}$ connect the same pair of distinct vertices of $T^{\gamma\sigma}$ and with the same orientation. Let $\pi=\gamma\sigma$ and let,
$$\sigma^{switch}=\sigma(a,b)(c,d)(a,d)(b,c).$$ 
Then all the following satisfy.
\begin{enumerate}[label=(\roman*)]
    \item $\gamma\vee\sigma^{switch}$ has two blocks $A$ and $B$, so that $\{a,d\}\subset A$ and $\{b,c\}\subset B$.
    \item $\sigma^{switch}|_A\in \NC_2(\gamma|_A)$ and $\sigma^{switch}|_B\in \NC_2(\gamma|_B)$.
    \item $\pi|_A=\gamma|_A\sigma^{switch}|_A$ and $\pi|_B=\gamma|_B\sigma^{switch}|_B$.
    \item If $e_u^{\gamma\sigma}$ and $e_v^{\gamma\sigma}$ are no loops of $T^{\gamma\sigma}$ connecting the same pair of vertices of $T^{\pi}$ which are distinct to the vertices connected by $e_a^{\gamma\sigma}$, then, $e_u^{\gamma\sigma^{switch}}$ and $e_v^{\gamma\sigma^{switch}}$ connect the same pair of vertices of $T^{\gamma\sigma^{switch}}$.
\end{enumerate}
\end{lemma}
\begin{proof}
Since $e_a^{\gamma\sigma}$ and $e_b^{\gamma\sigma}$ connect the same pair of vertices of $\pi=\gamma\sigma$ and with the same orientation then there exist cycles $C_1$ and $C_2$ of $\gamma\sigma$ of the form,
$$C_1=(x_1,\dots,x_s,c,x_{s+1},\dots,x_{s+t},a),$$
$$C_2=(y_1,\dots,y_u,d,y_{u+1},\dots,x_{u+v},b).$$
These cycles of $\gamma\sigma$ correspond to the blocks of $\pi$ that $e_a^{\gamma\sigma}$ and $e_b^{\gamma\sigma}$ connect. Let $C_3,\dots,C_p$ be the rest of cycles of $\gamma\sigma$. It is easy to note that $C_3,\dots,C_p$ are all cycles of $\gamma\sigma^{switch}$ and $\gamma\sigma^{switch}$ has cycles,
$$C_{1,1}=(x_1,\dots,x_s,c),$$
$$C_{1,2}=(x_{s+1},\dots,x_{s+t},a),$$
$$C_{2,1}=(y_1,\dots,y_u,d),$$
$$C_{2,2}=(y_{u+1},\dots,y_{u+v},b).$$
The latest means $\#(\gamma\sigma^{switch})=\#(\gamma\sigma)+2$, hence,
\begin{eqnarray*}
m+2\#(\gamma\vee\sigma^{switch}) & \geq & \#(\gamma)+\#(\sigma^{switch})+\#(\gamma\sigma^{switch}) \\
&=& \#(\gamma)+\#(\sigma)+\#(\gamma\sigma)+2 \\
&=& m+4.
\end{eqnarray*}
Therefore $\#(\gamma\vee\sigma^{switch})\geq 2$. On the other hand, let $\sigma^\prime=\sigma(a,b)(c,d)$, it is clear $\sigma^\prime\leq \sigma^{switch}$, thus,
\begin{eqnarray}\label{Aux: Inequality 3}
2\leq \#(\gamma\vee\sigma^{switch}) \leq \#(\gamma\vee\sigma^\prime).
\end{eqnarray}
However recall that $(a,b)$ is non-cutting as $e_a^{\gamma\sigma}$ is not a loop of $T^{\gamma\sigma}$, therefore $\#(\gamma\vee\sigma(a,b))=\#(\gamma\vee\sigma)=1$. Note that $\sigma^\prime=\sigma(a,b)(c,d)$, since removing the through string $(c,d)$ can increase the number of blocks of $\gamma\vee\sigma(a,b)$ at most by $1$ then $\#(\gamma\vee\sigma^\prime)\leq \#(\gamma\vee\sigma(a,b))+1=2$. Combining this with Inequality \ref{Aux: Inequality 3} we get,
$$\#(\gamma\vee\sigma^{switch})=2.$$
Let $A$ and $B$ be the blocks of $\gamma\vee\sigma^{switch}$, since $\{a,d\}$ is a block of $\sigma^{switch}$ then it must be contained in either $A$ or $B$, suppose it is contained in $A$. Similarly $\{b,c\}$ must be contained in either $A$ or $B$. Suppose it is contained in $A$, any other block $D$ of $\sigma^{switch}$ is also a block of $\sigma$ and therefore it is contained in either $A$ or $B$. Thus as $\{a,b,c,d\}\subset A$ it follows that any block of $\sigma$ is contained in either $A$ or $B$ which proves $\gamma\vee\sigma$ must be contained in the partition whose blocks are $A$ and $B$ which is a contradiction as $\gamma\vee\sigma=1_m$. So it must be $\{c,d\}\subset B$, this proves $(i)$. To prove $(ii)$ observe that by \cite[Equation 2.9]{MN},
$$\#(\sigma^{switch}|_A)+\#(\gamma|_A)+\#(\gamma|_A\sigma^{switch}|_A)\leq |A|+2,$$
$$\#(\sigma^{switch}|_B)+\#(\gamma|_B)+\#(\gamma|_B\sigma^{switch}|_B)\leq |B|+2,$$
with equality if and only if $\sigma^{switch}|_A\in \NC_2(\gamma|_A)$ and $\sigma^{switch}|_B\in \NC_2(\gamma|_B)$. Summing over two Inequalities yields,
$$\#(\sigma^{switch})+\#(\gamma)+\#(\gamma\sigma^{switch})\leq m+4.$$
But we proved $\#(\sigma^{switch})+\#(\gamma)+\#(\gamma\sigma^{switch})=m+4$ so it must be equality in both inequalities which proves $(ii)$. To prove $(iii)$ we use that $\gamma\vee\sigma^{switch}=\gamma\vee\gamma\sigma^{switch}$, so $A$ and $B$ are both union of blocks of $\gamma\sigma^{switch}$. Since $\{a,d\}\subset A$ then $C_{1,2}\subset A$ and $C_{2,1}\subset A$. Similarly $C_{1,1}\subset B$ and $C_{2,2}\subset B$. Suppose without loss of generality that $C_3,\dots,C_q \subset A$ with $3\leq q\leq p$ and $C_{q+1},\dots,C_p\subset B$. Thus,
$$A=C_{1,2}\cup C_{2,1}\cup C_3 \cup \cdots \cup C_q.$$
$$B=C_{1,1}\cup C_{2,2}\cup C_{q+1} \cup \cdots \cup C_p.$$
Therefore, we have that $\gamma|_A\sigma^{switch}|_A$ has cycle decomposition $C_{1,2}C_{2,1}\ab C_3\cdots C_q$. However $\pi=\gamma\sigma=C_1C_2\cdots C_p$, thus $\pi|_A=C_{1,2}C_{2,1}C_3\cdots C_q$ which proves $\pi|_A=\gamma|_A\sigma^{switch}|_A$. Similarly we prove $\pi|_B=\gamma|_B\sigma^{switch}|_B$. It remains to prove $(iv)$. Suppose $e_u^{\gamma\sigma}$ and $e_v^{\gamma\sigma}$ connect the same pair of vertices which are different from the vertices connected by $e_a^{\gamma\sigma}$. The vertices of $T^{\gamma\sigma}$ are the cycles of $\gamma\sigma$, so let us suppose without loss of generality that $e_u^{\gamma\sigma}$ connect the cycles $C_3$ and $C_4$. These cycles are also cycles of $\gamma\sigma^{switch}$ so $e_u^{\gamma\sigma^{switch}}$ and $e_v^{\gamma\sigma^{switch}}$ still connect the same pair of vertices of $T^{\gamma\sigma^{switch}}$. It remains the case when one the vertices is either $C_1$ or $C_2$ and the other is one of $C_3,\dots,C_p$ (by hypothesis $e_u^{\gamma\sigma}$ cannot connect $C_1$ and $C_2$). Suppose without loss of generality that $e_u^{\gamma\sigma}$ connect $C_1$ and $C_3$. Let $u^\prime$ be such that $\{u,u^\prime\}$ is a block of $\sigma$. Then we either have $u\in C_1$ and $u^\prime\in C_3$ or $u^\prime \in C_1$ and $u\in C_3$, let us assume without loss of generality we are in the former case. Since $u\neq a,b,c,d$ then $\{u,u^\prime\}$ is also a block of $\sigma^\prime$ and therefore $\{u,u^\prime\}$ is contained in either $A$ or $B$ since, $u^\prime\in C_3$ then it must be $\{u,u^\prime\}\subset A$ and therefore $u\in C_{1,2}$. This proves $e_u^{\gamma\sigma^{switch}}$ connect the vertices $C_{1,2}$ and $C_3$ of $T^{\gamma\sigma^{switch}}$. Similarly let $v^\prime$ be such that $\{v,v^\prime\}$ is a block of $\sigma$. Since $e_v^{\gamma\sigma}$ connect the same pair of vertices of $e_{u}^{\gamma\sigma}$ then it must be that $e_v^{\gamma\sigma}$ connects $C_1$ and $C_3$. Proceeding as before we prove that $e_v^{\gamma\sigma^{switch}}$ connects the vertices $C_{1,2}$ and $C_3$ of $T^{\gamma\sigma^{switch}}$ which proves $(iv)$.
\end{proof}

\begin{lemma}\label{Lemma: Subgraph of T^gamma sigma}
Let $(\sigma,\tau,\pi)$ be a limit triple and let $B\in \overline{\pi_\tau}$ be a block of $\overline{\pi_\tau}$. The graph $T^{\gamma\sigma}(\sigma,\tau)_B$ is connected.
\end{lemma}
\begin{proof}
Since $\gamma\vee\tau=1_m$ then $T^{\gamma\sigma}(\sigma,\tau)$ is connected and thus it has a spanning tree, which we call $T$.

Let $B_1,B_2\in\sigma$ be such that $e_{B_1}^{\gamma\sigma}$ and $e_{B_2}^{\gamma\sigma}$ connect the same pair of vertices of $T^{\gamma\sigma}$ and $B_1$ and $B_2$ are both edges of $T$. Then $B_1$ and $B_2$ are both edges of $T^{\gamma\sigma}(\sigma,\tau)$ adjacent to the same withe vertex. Let $D_1$ and $D_2$ be the black vertices of $T^{\gamma\sigma}(\sigma,\tau)$ (and then of $T$) that are adjacent to $B_1$ and $B_2$ respectively. $D_1$ and $D_2$ correspond to blocks of $\tau$ which we merge. Let $\hat{\tau}$ be the resulting partition after merging all blocks of $\tau$ described as before. Note that $\overline{T^{\hat{\tau}}}$ is still a tree as each time we merge two black vertices of $T$ into the same block of $\hat{\tau}$ we identify two edges. The process to get $\hat{\tau}$ is essentially the same process we did in Subsection \ref{Subsection: Non-crossing limit triples} however in this case we might be missing some edges of $T^{\gamma\sigma}(\sigma,\tau)$ and therefore the resulting partition $\hat{\tau}$ not necessarily satisfies $\overline{\gamma\sigma}\leq \hat{\tau}$ in the sense that any block of $\hat{\tau}$ is a union of blocks of $\overline{\gamma\sigma}$ but there might be blocks of $\overline{\gamma\sigma}$ that are in no block of $\hat{\tau}$. For those blocks of $\overline{\gamma\sigma}$ we let them be blocks of $\hat{\tau}$, in this way now it is satisfied $\overline{\gamma\sigma}\leq \hat{\tau}$. For these blocks of $\overline{\gamma\sigma}$ we also let them be leaves of $\overline{T^{\hat{\tau}}}$ in the sense that if $B$ is some of these blocks we let $B$ be a black vertex of $\overline{T^{\hat{\tau}}}$ joined to the white vertex that contains $e_u$ with $u\in B$ (this is well defined as for $e_u$ and $e_v$ with $u,v\in B$ we have that both $e_u$ and $e_v$ belong to the same connected component of $T^{\gamma\sigma}$). Let $\widetilde{\tau} = \widetilde{\hat{\tau}}$. As proved in Subsection \ref{Subsection: Non-crossing limit triples} the partition $\widetilde{\tau}\in \cP(\overline{E^{\gamma\sigma}})$ is given by $\overline{e_A^{\gamma\sigma}}$ and $\overline{e_B^{\gamma\sim}}$ are in the same block if and only if $\overline{e_A^{\gamma\sigma}}$ and $\overline{e_B^{\gamma\sigma}}$ are adjacent to the same black vertex of $\overline{T^{\hat{\tau}}}$. For $u,v\in [m]$, if $\overline{e_u^{\gamma\sigma}}$ and $\overline{e_v^{\gamma\sigma}}$ are in the same block of $\widetilde{\tau}$ then $e_u^{\gamma\sigma}$ and $e_v^{\gamma\sigma}$ join the same pair of vertices of $T^{\pi}$ so we can choose an appropriate orientation of the edges $\sim_p$ so that $\underset{\sim_p}{\widetilde{\tau}}\leq \pi$. On the other hand note that $\overline{T^{\hat{\tau}}}$ is a graph whose white vertices are the connected components of $T^{\gamma\sigma}$, the black vertices are the blocks of $\hat{\tau}$ that can be identified with blocks of $\widetilde{\tau}$ and edges indexed by $\overline{E^{\gamma\sigma}}$ where $\overline{e}\in \overline{E^{\gamma\sigma}}$ connects a white vertex $W$ with a black vertex $b$ if $\overline{e}\in b$ and $\overline{e}\in \overline{W}$, i.e. $\overline{T^{\hat{\tau}}}=T^{\gamma\sigma}(\widetilde{\tau})$ which is a tree. Moreover, if $\{u,v\}$ is a block of $\sigma$ such that $e_u^{\gamma\sigma}$ is a loop of $T^{\gamma\sigma}$, then $\{u,v\}$ must also be a block of $\tau$ otherwise $\C_{\tau}(\pi)=0$. This means that if $\overline{e_u^{\gamma\sigma}}$ is a loop of $\overline{T^{\gamma\sigma}}$ then $\{\overline{e_u^{\gamma\sigma}}\}$ must be a block of $\widetilde{\tau}$. It follows by Lemma \ref{Lemma: Equality when tree},
$$\#(\underset{\sim_p}{\widetilde{\tau}})=\#(\gamma\sigma)-2(\#(\gamma\vee\sigma)-1)=m-\#(\sigma)-\#(\gamma)+2=m/2-r+2.$$
But we also know $\#(\pi)=m/2-r+2$ so it must be $\underset{\sim_p}{\widetilde{\tau}}=\pi$. Suppose that for some of the blocks $B\in \overline{\pi_\tau}$ the graph $T^{\gamma\sigma}(\sigma,\tau)_B$ is disconnected. The block $B$ must correspond to a block for which $\overline{e_B}$ is not a loop of $T^{\pi}$, otherwise by definition of $\overline{\pi_\tau}$, $B$ has size $2$ and then $T^{\gamma\sigma}(\sigma,\tau)_B$ is connected. Note that $\hat{\tau}$ only joins blocks of $\tau$ that are contained in the same block of $\overline{\pi}$. Also note that two blocks of $\tau$ contained in a block $B\in \overline{\pi}$ that are in the same block of $\hat{\tau}$ must be in the same connected component of $T^{\gamma\sigma}(\sigma,\tau)_B$ because $\hat{\tau}$ merges blocks adjacent to the same white vertex. For this observations we conclude that there must exist two blocks $D_1$ and $D_2$ of $\hat{\tau}$ that are in distinct connected components of $T^{\gamma\sigma}(\sigma,\tau)_B$. Moreover if $u\in D_1$ and $v\in D_2$ then $e_u^{\pi}$ and $e_v^{\pi}$ are no loops of $T^{\pi}$ joining the same pair of vertices. The blocks $D_1$ and $D_2$ correspond to distinct blocks of $\widetilde{\tau}$ and therefore by Lemma \ref{Lemma: Equality when tree} part $(ii)$ we know that $\overline{e_u^{\underset{\sim_p}{\widetilde{\tau}}}}=\overline{e_u^{\pi}}$ and $\overline{e_v^{\pi}}=\overline{e_v^{\underset{\sim_p}{\widetilde{\tau}}}}$ join distinct pair of vertices. However we took $B\in \overline{\pi_\tau}$ which means any two edges join the same pair of vertices which is a contradiction.
\end{proof}

\begin{lemma}\label{Lemma: Subgraphs of T^gamma sigma shared at most one white vertex}
Let $(\sigma,\tau,\pi)$ be a limit triple and let $B,D\in \overline{\pi_\tau}$ be distinct blocks of $\overline{\pi_\tau}$. The graphs $T^{\gamma\sigma}(\sigma,\tau)_B$ and $T^{\gamma\sigma}(\sigma,\tau)_D$ shared at most one vertex which must be white.
\end{lemma}
\begin{proof}
First observe that the black vertices of these graphs correspond to blocks of $\tau$ contained in distinct blocks of $\overline{\pi}$ so they cannot have any black vertex in common. Suppose the graphs shared two white vertices, $W_1$ and $W_2$ which correspond to connected components of $T^{\gamma\sigma}$. By Lemma \ref{Lemma: Subgraph of T^gamma sigma} we know each graph is connected and therefore in each graph there exist a path going from $W_1$ to $W_2$. Each path is completely contained in the graphs. The path contained in the graph $T^{\gamma\sigma}(\sigma,\tau)_B$ contains only edges $e_V$ where $V\in \sigma$ is a block contained in $B$, similarly the edges of the other path correspond to blocks only contained in $D$. In the quotient $\overline{T^{\gamma\sigma}(\sigma,\tau)^{\overline{\pi_\tau}}}$ we get then two paths going from $A$ to $B$ which are distinct as the original paths only contain edges corresponding to blocks completely contained in each $B$ and $D$ and therefore they cannot be the same in the quotient otherwise $B=D$. This is a contradiction as that implies $\overline{T^{\gamma\sigma}(\sigma,\tau)^{\overline{\pi_\tau}}}$ has a cycle from $W_1$ to $W_2$.
\end{proof}

Thanks to Lemmas \ref{Lemma: Subgraph of T^gamma sigma} and \ref{Lemma: Subgraphs of T^gamma sigma shared at most one white vertex} we know that the graph $T^{\gamma\sigma}(\tau,\sigma)$ is a tree if and only if each subgraph $T^{\gamma\sigma}(\tau,\sigma)_B$ is a tree for any $B\in\overline{\pi_\tau}$. For a non-crossing limit triple $(\sigma,\tau,\pi)$ we have that $T^{\gamma\sigma}(\tau,\sigma)_B$ is a tree for any $B\in\overline{\pi_\tau}$ however for crossing limit triples at least one of the graphs $T^{\gamma\sigma}(\tau,\sigma)_B$ is not a tree. This will be the main tool to characterize the non-crossing limit triples.

\begin{remark}\label{Remark: Loops blocks are already trees}
For a block $B\in\overline{\pi_\tau}$ for which $e_B^{\pi}$ is a loop of $T^{\pi}$, we know that $T^{\gamma\sigma}(\sigma,\tau)_B$ is already a tree as it consists of a single white and black vertex joined by a unique edge.
\end{remark}

\begin{proposition}\label{Proposition: Characterization of subgraphs corresponding to crossing limit triples}
Let $(\pi,\tau)\in \mathcal{B}$ then there exists 
$$\sigma\in \NC_2(m_1,\dots,m_r)\cup \NC_2^{nc}(m_1,\dots,m_r),$$ such that $(\sigma,\tau,\pi)$ is a limit triple and there is $B\in\overline{\pi_\tau}$ such that \allowbreak
$T^{\gamma\sigma}(\sigma,\tau)_B$ is not a tree.
\end{proposition}
\begin{proof}
We know there exist $\sigma$ so that $(\sigma,\tau,\pi)$ is a limit triple. However if $T^{\gamma\sigma}(\sigma,\tau)_B$ is a tree for any $B\in \overline{\pi_\tau}$ then $T^{\gamma\sigma}(\sigma,\tau)$ would be a tree which is a contradiction as that would imply $(\pi,\tau)\in \mathcal{A}$.
\end{proof}

In the last Subsection we got a unique $(\tau,\sigma)$ for each $(\pi,\tau)\in\mathcal{A}$. Now we would like to associate to each pair $(\pi,\tau)\in \mathcal{B}$ a canonical choice of $\sigma$ so that $(\sigma,\tau,\pi)$ is a limit triple and there is $B\in\overline{\pi_\tau}$ so that $T^{\gamma\sigma}(\sigma,\tau)_B$ is not a tree. We know such a $\sigma$ must exist thanks to Proposition \ref{Proposition: Characterization of subgraphs corresponding to crossing limit triples}.

\begin{definition}\label{Definition: definition of canonical sigma}
Let $(\pi,\tau)\in \mathcal{B}$. We let $\sigma^c \in \NC_2(m_1,\dots,m_r)\cup \NC_2^{nc}(m_1,\dots,m_r),$ be defined as follows. Let $B\in\overline{\pi_\tau}$. If there exist $\sigma$ such that $T^{\gamma\sigma}(\sigma,\tau)_B$ is a tree then we let $\sigma^c_B = \sigma_B$. Otherwise we let $\sigma^c_B$ to be $\sigma_B$ for any choice of $\sigma$ so that $(\sigma,\tau,\pi)$ is a limit triple. The latest is well defined because $\sigma\leq \tau\leq \overline{\pi_\tau}$ so it makes sense to define $\sigma$ restricted to each block of $\overline{\pi_\tau}$. Thanks to Lemma \ref{Lemma: Sigma is unique is non-crossing partitioned permutation free-loop type} the permutation $\sigma^c$ is uniquely defined at the blocks $B$ of $\overline{\pi_\tau}$ so that $T^{\gamma\sigma}(\sigma,\tau)_B$ is a tree. For any other block we might have more than once choice. See example \ref{Example: canonical sigma}.  
\end{definition}

\begin{example}\label{Example: canonical sigma}
Let us consider $m_1=4$ and $m_2=m_3=m_4=m_5=m_6=m_7=2$. Let
$$\pi=\{1,12,14,16\},\{2,4,6,8,10,11,13,15\},\{3,5,7,9\},$$ 
so that 
$$\overline{\pi}=\{1,4,11,12,13,14,15,16\},\{2,3,5,6,7,8,9,10\}.$$
Let us consider two examples for $\tau$. Let,
$$\tau_1=\{1,11,13,16\},\{4,12,14,15\},\{2,5,7,10\},\{3,6,8,9\},$$
and,
$$\tau_2=\{1,4,13,16\},\{11,12,14,15\},\{2,5,7,10\},\{3,6,8,9\}.$$
Since there are no loops we have $\overline{\pi_{\tau_1}}=\overline{\pi_{\tau_2}}=\overline{\pi}$. In example \ref{Example: Big example} we said that there is no $\sigma\leq \tau_1$ so that $T^{\gamma\sigma}(\sigma,\tau)$ is a tree, moreover, there is no $\sigma$ so that $T^{\gamma\sigma}(\sigma,\tau)_B$ is a tree for $B\in\overline{\pi}$. In this case for the two blocks of $\overline{\pi}$ we have more than one choice of $\sigma^c$ restricted to each block. For example we can choose,
$$\sigma^c=(1,11)(13,16)(4,12)(14,15)(2,5)(7,10)(3,6)(8,9),$$
or,
$$\sigma^c=(1,13)(11,16)(4,12)(14,15)(2,5)(7,10)(3,6)(8,9).$$
An example of $T^{\gamma\sigma^c}(\sigma^c,\tau_1)$ can be seen in Figure \ref{Figure: Big example} part C) for the first choice of $\sigma^c$. On the other hand, observe that the blocks $\{2,5,7,10\}$ and $\{3,6,8,9\}$ of $\tau_2$ are also blocks of $\tau_1$, so for the block $B_1=\{2,3,5,6,7, \ab 8, \ab 9,10\}$ of $\overline{\pi}$ there is no $\sigma$ such that $T^{\gamma\sigma}(\sigma,\tau)_{B_1}$ is a tree so for this block we can choose $\sigma^c|_{B_1}$ as before, for example,
$$\sigma^c|_{B_1}=(2,5)(7,10)(3,6)(8,9).$$
However for the blocks $\{1,4,13,16\}$ and $\{11,12,14,15\}$ of $\tau_2$ that are contained in the block $B_2=\{1,4,11,12,13,14,15,16\}$ of $\overline{\pi}$ we can choose,
$\sigma^c_{B_2}=(1,4)(13,16)(11,12)(14,15).$
It can be easily checked that $T^{\gamma\sigma^c}(\sigma^c,\tau_2)_{B_2}$ is a tree.
\end{example}

\begin{definition}
Let $(\pi,\tau)\in\mathcal{B}$ and let $(\sigma^c,\tau,\pi)$ be a crossing limit triple with $\sigma^c$ be as in definition \ref{Definition: definition of canonical sigma}. We let $\rho \vcentcolon = \rho_{\tau,\pi}$ be the partition of $[m]$ given as follows:
\begin{enumerate}
    \item For $B\in\overline{\pi_\tau}$ we let $B$ be a block of $\rho$ if $T^{\gamma\sigma^c}(\sigma^c,\tau)_B$ is not a tree.
    \item For $B\in\overline{\pi_\tau}$ we let the blocks of $\tau$ contained in $B$ to be blocks of $\rho$ if $T^{\gamma\sigma^c}(\sigma^c,\tau)_B$ is a tree.
\end{enumerate}
Proposition \ref{Proposition: Characterization of subgraphs corresponding to crossing limit triples} ensures $\rho$ exist and it satisfies,
$$\tau \leq \rho \leq \overline{\pi}.$$
For an example see \ref{Example: Example of rho}.
\end{definition}

\begin{example}\label{Example: Example of rho}
Let $(\pi,\tau_1)$ and $(\pi,\tau_2)$ be as in Example \ref{Example: canonical sigma}. As mentioned in Example \ref{Example: canonical sigma} there is no $\sigma\leq \tau^1$ so that $T^{\gamma\sigma}(\sigma,\tau)_B$ is a tree for some $B\in\overline{\pi}$. Thus,
$$\rho_{\tau_1,\pi}=\overline{\pi}=\{1,4,11,12,13,14,15,16\},\{2,3,5,6,7,8,9,10\}.$$
On the other hand, in Example \ref{Example: canonical sigma} we saw that there is $\sigma\leq \tau_2$ so that $T^{\gamma\sigma}(\sigma,\tau)_{B_2}$ is a tree with $B_2=\{1,4,11,12,13,14,15,16\}\in \overline{\pi}$. For the other block $B_1$ of $\overline{\pi}$ we said there is no $\sigma$ so that $T^{\gamma\sigma}(\sigma,\tau)_{B_1}$ is a tree. Hence,
$$\rho_{\tau_2,\pi}=\{1,4,13,16\},\{11,12,14,15\},\{2,3,5,6,7,8,9,10\}.$$
\end{example}


\begin{definition}\label{Definition: Graphs that count the extra terms modified}
Let $n\in \mathbb{N}$ be even. We say that $\pi\in \cP(n)$ respects the parity if any block of $\pi$ has the same number of even and odds numbers. Let $\tau\in \cP(n)$ be a partition that respects the parity. Let $\gamma_n^{par}\in S_n$ to be the permutation with cycle decomposition,
$$(1,2)\cdots (n-1,n).$$
We say that $\tau$ doesn't admit a non-crossing partitioned permutation if $\tau\vee\sigma_n^{par}=1_{n}$ and there is no $\sigma\in \cP_2(n)$ that respects the parity with $\sigma \leq \tau$ and such that $\Gamma(\tau,\sigma,\gamma_n^{par})$ is a tree.
\end{definition}

\begin{definition}\label{Definition: Equivalence class of blocks of partitioned permutation}
Let $(\tau,\sigma)\in \PS_{\NC_2}(m_1,\dots,m_r)$ and let $(a,b)$ and $(c,d)$ be two cycles of $\sigma$.
\begin{enumerate}
    \item We say that $(a,b)\sim (c,d)$ if either $[a]_{\gamma\sigma}=[c]_{\gamma\sigma}$ and $[b]_{\gamma\sigma}=[d]_{\gamma\sigma}$ or $[a]_{\gamma\sigma}=[d]_{\gamma\sigma}$ and $[b]_{\gamma\sigma}=[c]_{\gamma\sigma}$.
    \item For two blocks $B_1,B_2\in \tau$ we say that they are related if there exist $(a,b),(c,d)\in\sigma$ with $(a,b)\subset B_1$ and $(c,d)\subset B_2$ and such that $(a,b)\sim (c,d)$. For a block $B\in \tau$ we denote by $D_{B}$ to the set of blocks of $\tau$ distinct from $B$ that are related to $B$.
    \item We say that $(\tau,\sigma)$ is \textit{maximal} if $D_B=\emptyset$ for some $B\in\tau$, in this case we call $B$ a \textit{maximal block}.
\end{enumerate}
An example of a maximal and a no maximal non-crossing partitioned permutation can be seen in Example \ref{Example: Example of maximal partitioned permutation}.
\end{definition}

\begin{example}\label{Example: Example of maximal partitioned permutation}
Let $m_1=m_2=m_3=m_4=2$. And let $(\tau_1,\sigma_1)$ and $(\tau_2,\sigma_2)$ be given by,
$$\sigma_1=(1,4)(2,3)(5,6)(7,8),$$
$$\tau_1=\{1,4\},\{2,3,5,6,7,8\},$$
$$\sigma_2=(1,2)(3,4)(5,6)(7,8),$$
$$\tau_2=\{1,2,3,4,5,6,7,8\}.$$
Both partitioned permutations are in $\PS_{\NC_2}(2,2,2,2)$, however $(\tau_2,\sigma_2)$ is maximal as $D_B=\emptyset$ for the block $B=\{1,2,3,4,5,6,7,8\}$ while $(\tau_1,\sigma_1)$ is not maximal because $(1,4)\sim (2,3)$ and therefore $D_{\{1,4\}}=\{2,3,5,6,7\} \neq \emptyset$. 
\end{example}

\begin{lemma}\label{Lemma: Crossing part contribution}
For any $m_1,\dots,m_r\in \mathbb{N},$
\begin{equation*}
\sum_{(\tau,\pi)\in \mathcal{B}}\C_{\tau}(\pi)=\sum_{\substack{(\rho,\sigma^{new})\in \PS_{\NC_2}^{loop-free}(m_1,\dots,m_r)}}\widehat{\K}_{(\rho,\sigma^{new})}.
\end{equation*}
Where,
\begin{multline*}
\widehat{\K}_{(\rho,\sigma^{new})}=2^{\#(\sigma^{new})-\#(\rho)} \\
\mbox{} \times
\left[\prod_{\substack{B\in\rho \\ |B|=2n}}\left(\beta_{|B|}+\mathbbm{1}_{D_B=\emptyset}\sum_{\substack{\tau\in A_n}}\prod_{D\in \tau}\beta_{|D|}\right)-\prod_{B\in\rho}\beta_{|B|}\right],
\end{multline*}
with,
\begin{multline*}
A_n=\{\tau\in \cP(2n): \tau\text{ doesn't admit a non-crossing} \\ \text{partitioned permutation}\}.
\end{multline*}
\end{lemma}

\begin{proof}
Let $(\pi,\tau)\in\mathcal{B}$ and let $(\sigma^c,\tau,\pi)$ be a crossing limit triple with $\sigma^c$ as in Definition \ref{Definition: definition of canonical sigma}. Let us set $\sigma\vcentcolon = \sigma^c$. Let $B\in\overline{\pi_\tau}$. If $B$ is also a block of $\rho$ then $\overline{T^{\gamma\sigma}(\sigma,\tau)^{\rho}_B}=\overline{T^{\gamma\sigma}(\sigma,\tau)^{\overline{\pi_\tau}}_B}$ which is a tree as $\overline{T^{\gamma\sigma}(\sigma,\tau)^{\overline{\pi_\tau}}}$ is a tree by condition $L_4$. If $B$ is not a block of $\rho$ then the blocks of $\tau$ contained in $B$ are also blocks of $\rho$ and therefore $T^{\gamma\sigma}(\sigma,\tau)^{\rho}_B=T^{\gamma\sigma}(\sigma,\tau)_B$ which is a tree by definition of $\rho$. Thus for any block $B\in\overline{\pi_\tau}$ we have that $\overline{T^{\gamma\sigma}(\sigma,\tau)^{\rho}_B}$ is a tree. The graph $T^{\gamma\sigma}(\sigma,\tau)^{\rho}$ is the same as the graph $T^{\gamma\sigma}(\sigma,\rho)$, so we have proved $\overline{T^{\gamma\sigma}(\sigma,\rho)_B}$ is a tree for any $B\in\overline{\pi_\tau}$. Moreover the partition $\overline{\pi_\tau}$ consist of blocks of $\overline{\pi}$ except by the blocks of $\tau$ that are loops of $T^{\pi}$ which are themselves blocks of $\overline{\pi_\tau}$. Since these blocks are also blocks of $\rho$ then $\overline{\pi_\tau}=\overline{\pi_\rho}$. Now we will construct a non-crossing pairing $\sigma^{new}\leq \rho$ so that $T^{\gamma\sigma^{new}}(\sigma^{new},\rho)$ is a tree and $(\rho,\sigma^{new})\in \PS_{\NC_2}^{loop-free}(m_1,\dots,m_r)$. Moreover by Lemma \ref{Lemma: Sigma is unique is non-crossing partitioned permutation free-loop type} the pairing $\sigma^{new}$ will be unique. One example of this construction can be seen in Figure \ref{Figure: Big example complement}.

On one hand, if $B\in\overline{\pi_\tau}$ is a block such that $T^{\gamma\sigma}(\sigma,\tau)_B$ is a tree then $\rho|_B=\tau|_B$ by definition and therefore if we let $\sigma^{new}_B=\sigma_B$ we get $T^{\gamma\sigma^{new}}(\sigma^{new},\rho)_B$ is a tree. On the other hand, if $B\in\overline{\pi_\tau}$ is a block such that $T^{\gamma\sigma}(\sigma,\tau)_B$ is not a tree then $B$ is also a block of $\rho$. let $D_1,\dots,D_n$ be the set of blocks of $\tau$ contained in $B$. Since $T^{\gamma\sigma}(\sigma,\tau)_B$ is not a tree but $\overline{T^{\gamma\sigma}(\sigma,\tau)^{\rho}_B}$ is, there exist multiple edges joining the same pair of vertices. Assume that $B$ and some connected component of $T^{\gamma\sigma}$ which we denote by $W_1$ are the pair of vertices that have multiple edges. Let $\{V_1,\dots,V_n\}$ be the set of edges connecting $B$ and $W_1$. Each $V_i$ correspond to a block of $\sigma$, let us write $V_i=\{u_i,v_i\}$. By Remark \ref{Remark: Loops blocks are already trees} the edge $e_{u_1}^{\pi}$ is a non-loop of $T^{\pi}$ and hence $e_{u_1}^{\gamma\sigma}$ must be a non-loop of $T^{\gamma\sigma}$ since $\gamma\sigma\leq \pi$. Thus, the edges $e_{u_1}^{\gamma\sigma},e_{v_1}^{\gamma\sigma},\dots,e_{v_n}^{\gamma\sigma},e_{u_n}^{\gamma\sigma}$ are all joining the same pair of vertices of $T^{\gamma\sigma}$ and are all non-loops of $T^{\gamma\sigma}$. Let us assume without loss of generality that all the edges $e_{u_1}^{\gamma\sigma},\dots,e_{u_n}^{\gamma\sigma}$ have the same orientation while all the edges $e_{v_1}^{\gamma\sigma},\dots,e_{v_n}^{\gamma\sigma}$ have the same orientation and opposite to the edges $e_{u_1}^{\gamma\sigma},\dots,e_{u_n}^{\gamma\sigma}$. Let $\sigma|_{W_1}^\prime = \sigma|_{W_1}(u_1,v_1)(u_2,v_2)(u_1,v_2)(u_2,v_1)$. By Lemma \ref{Lemma: Splitting edges produces two new non-crossing pairings} the partition $\gamma|_{W_1}\vee \sigma_{W_1}^\prime$ has two connected components $W_{1,1}$ and $W_{1,2}$ whose union is $W_1$. We let $\sigma^\prime$ to be $\sigma(u_1,v_1)(u_2,v_2)(u_1,v_2)(u_2,v_1)$. Thanks to Lemma \ref{Lemma: Splitting edges produces two new non-crossing pairings} in the graph $T^{\gamma\sigma^\prime}(\sigma^\prime,\tau)_B$ the edges $V_3,\dots,V_n$ go from $B$ to either $W_{1,1}$ or $W_{1,2}$ while the edge $\{u_1,v_2\}$ goes from $B$ to $W_{1,1}$ and the edge $\{u_2,v_1\}$ goes from $B$ to $W_{1,2}$. Moreover by Lemma \ref{Lemma: Splitting edges produces two new non-crossing pairings} part $(iv)$ if $V_1^\prime,\dots,V_{n^\prime}^\prime$ are another pair of edges going from some other block of $\rho$ say $B_2$ to $W_1$ then in the graph $T^{\gamma\sigma^\prime}(\sigma^\prime,\tau)_B$ either all these edges go from $B_2$ to $W_{1,1}$ or all these edges go from $B_2$ to $W_{1,2}$. We conclude that $\overline{T^{\gamma\sigma^\prime}(\sigma^\prime,\rho)_B}$ is still a tree. In the new pairing $\sigma^\prime$ the edges $\{V_1,\dots,V_n\}$ are not joining the same pair of vertices anymore, but rather some of them join $B$ to $W_{1,1}$ while the rest join $B$ to $W_{1,2}$ so we can continue doing the same process for each collection of edges joining the same pair of vertices. At some point we get a pairing, $\sigma^{new}$ which has at most one edge joining two pair of vertices and therefore $T^{\gamma\sigma^{new}}(\sigma^{new},\rho)_B$ will be a tree for any $B\in\overline{\pi_\tau}=\overline{\pi_\rho}$ which proves $T^{\gamma\sigma^{new}}(\sigma^{new},\rho)$ is a tree as desired. Moreover any block of $\sigma^{new}$ that corresponds to a loop of $T^{\gamma\sigma}$ is also a block of $\sigma$ that correspond to a loop of $T^{\gamma\sigma}$ which is a block of $\tau$ and $\rho$. Then $(\rho,\sigma^{new})\in \PS_{\NC_2}^{loop-free}(m_1,\dots,m_r)$.
So given $(\pi,\tau)\in\mathcal{B}$ we have found a unique $$(\rho,\sigma^{new})\in \PS_{\NC_2}^{loop-free}(m_1,\dots,m_r).$$ 
Now we wonder how many pairs $(\pi,\tau)\in\mathcal{B}$ correspond to the same non-crossing partitioned permutation $(\rho,\sigma^{new})$. To answer this first observe that given $(\sigma^{new},\rho)$ we can construct $2^{\#(\sigma^{new})-\#(\rho)}$ distinct partitions $\pi^\prime$ as in Proposition \ref{Proposition: first contribution of alpha}. Since $\rho \leq \overline{\pi}$ then for some of these orientations we have $\pi^\prime \leq \pi$, however as both $\pi$ and $\pi^\prime$ have $m/2+r-2$ blocks then it must be $\pi=\pi^\prime$, i.e. $\pi$ can be any of the $2^{\#(\sigma^{new})-\#(\rho)}$ distinct partitions $\pi$ constructed as in Proposition \ref{Proposition: first contribution of alpha}. 

This determines how many distinct partitions $\pi$ correspond to the same pair $(\rho,\sigma^{new})$, however, it might also be possible that more than one $\tau$ correspond to the same pair $(\rho,\sigma^{new})$. To determine this observe that if $B\in\overline{\pi_\tau}$ is a block of $\overline{\pi_\tau}$ for which the blocks of $\tau$ contained in $B$ are also blocks of $\rho$ then $\tau$ is uniquely determined on these blocks. On the other hand, for a block $B\in \overline{\pi_\tau}$ that is also a block of $\rho$ we have that $T^{\gamma\sigma}(\sigma,\tau)_B$ is not a tree but $\overline{T^{\gamma\sigma}(\sigma,\tau)_B^{\rho}}$ is a tree. 

Observe that regardless of $\sigma|_B$ and $\tau|_B$ we always end up with the same pairing $\sigma^{new}$. The latest means that we are allowed to choose any $\tau|_B$ as long as $T^{\gamma\sigma}(\sigma,\tau)_B$ is not a tree for any $\sigma$. Suppose $B=\{u_1,u_2,\dots,u_n,\ab v_1,\dots,v_n\}$ so that $\sigma^{new}$ has blocks $\{u_i,v_i\}$ and $e_{u_1}^{\pi},\dots,e_{u_n}^{\pi}$ have all the same orientation while $e_{v_1}^{\pi},\dots,e_{v_n}^{\pi}$ have all the same orientation and opposite to $e_{u_1}^{\pi},\dots,e_{u_n}^{\pi}$. Let us identify this block with the block $B=\{1,\dots,2n\}$ so that $u_i$ are the even numbers and $v_i$ are the odd numbers. Then $\tau|_B$ can be any partition of $\cP(2n)$ that respects the parity (so that $\tau$ contains the same number of edges in each orientation) and such that there is no $\sigma\in \cP_2(2n)$ that respects the parity and $T^{\gamma\sigma}(\sigma,\tau)_B$ is a tree. The graph $T^{\gamma\sigma}(\sigma,\tau)_B$ is isomorphic to the graph $\Gamma(\tau|_B,\sigma|_B,\gamma_{2n}^{par})$. Hence, $\tau|_B$ can be any choice as long as it doesn't admit a non-crossing partitioned permutation as Defined in \ref{Definition: Graphs that count the extra terms modified}. In Example \ref{Example: The pairs corresponding to maximal block} we show an explicit computations of all these pairs for the case $r=4$ and $m_1=m_2=m_3=m_4=2$.

Note that at least one block of $\rho$ must be union of blocks of $\tau$ otherwise we would have $T^{\gamma\sigma^{new}}(\tau,\sigma^{new})$ is a tree which contradicts $(\pi,\tau)\in \mathcal{B}$. Moreover the block of $\rho$ that is union of blocks of $\tau$ must be a block of $\overline{\pi_\tau}$ which must be also a block of $\overline{\pi}$. The contribution of a block $B$ of $\rho$ that is not union of blocks of $\tau$ in $\C_{\tau}(\pi)$ is
$$\beta_{|B|},$$
while the contribution of a block $B$ of $\rho$ that is union of blocks of $\tau$ is
$$\prod_{\substack{D\in \tau \\ D\subset B}}\beta_{|D|}.$$

We said before that at any block $B$ of $\rho$ that is union of blocks of $\tau$ we can choose any $\tau|_B$ as long as it doesn't admit a non-crossing partitioned permutation. Let $\tau|_B$ one of these choices, then
\begin{equation}\label{Auxiliar: aux of last proof for crossing part}
\C_{\tau}(\pi)=\prod_{\substack{B\in \rho \\ B \text{ is a block of }\tau}}\beta_{|B|}\prod_{\substack{B\in \rho \\ B \text{ is union of blocks of }\tau}}\prod_{\substack{D\in\tau \\ D\subset B}}\beta_{|D|}.
\end{equation}

Note that if two blocks of $\tau$ are related in the sense of definition \ref{Definition: Equivalence class of blocks of partitioned permutation} then these blocks are contained in the same block of $\overline{\pi}$. Further, at least one of the blocks of $\rho$, $B$, is the union of all the blocks of $\tau$ contained in $B$ and $B$ is a block of $\overline{\pi}$. So this block must be maximal. This means $(\rho,\sigma^{new})$ contains at least one maximal block. If we let $\mathcal{M}$ to be the subset of $\PS_{\NC_2}^{loop-free}(m_1,\dots,m_r)$ that contains at least one maximal block we are proving that $(\rho,\sigma^{new})\in \mathcal{M}$. 

Conversely let $(\rho,\sigma^{new})\in \mathcal{M}$. We let $\pi$ to be one of the $2^{\#(\sigma^{new})-\#(\rho)}$ distinct partitions $\pi$ constructed as in Proposition \ref{Proposition: first contribution of alpha} with $(\rho,\sigma^{new})$. We choose $\tau \leq \rho$ be such that for at least one maximal block $B\in \rho$ the partition $\tau|_B$ doesn't admit a non-crossing partitioned permutation. Then $\#(\pi)-m/2+r-2=0$ and $\tau \leq \rho \leq \overline{\pi}$. Further, there is no $\sigma\leq \tau$ such that $T^{\gamma\sigma}(\sigma,\tau)$ is a tree as $\tau|_{B}$ doesn't admit a non-crossing partitioned permutation. Hence $(\tau,\pi)\in \mathcal{B}$.

So we have proved that the summation over $\mathcal{B}$ can be indexed as a summation over $\mathcal{M}$ just taking care of how may pairs $(\tau,\pi)\in \mathcal{B}$ correspond to the same $(\rho,\sigma^{new})\in \mathcal{M}$. Additionally, in (\ref{Auxiliar: aux of last proof for crossing part}) we express the contribution $\C_\tau(\pi)$ in terms of $(\rho,\sigma^{new})$. Hence

\begin{multline}\label{Aux: auxliar of last proof of non-crossing ones 2}
\sum_{(\tau,\pi)\in \mathcal{B}}\C_{\tau}(\pi)= \\
\sum_{(\rho,\sigma^{new})\in\mathcal{M}}2^{\#(\sigma^{new})-\#(\rho)}\left[ \prod_{\substack{B\in \rho}}\left(\beta_{|B|}+\mathbbm{1}_{D_B=\emptyset}\sum_{\tau\in A_{|B|/2}}\prod_{D\in\tau}\beta_{|D|}\right)-\prod_{\substack{B\in \rho}}\beta_{|B|} \right],
\end{multline}
where the summation over $\tau \in A_{|B|/2}$ is to consider all possible $\tau$ that correspond to the same pair $(\rho,\sigma^{new})$ while the product by $2^{\#(\sigma^{new})-\#(\rho)}$ is to consider all possible $\pi$ that correspond to the same pair $(\rho,\sigma^{new})$. Furthermore, we consider the term $\left(\beta_{|B|}+\mathbbm{1}_{D_B=\emptyset}\sum_{\tau\in A_{|B|/2}}\prod_{D\in\tau}\beta_{|D|}\right)$ because it might be possible that a maximal block $B\in \rho$ is not union of blocks of $\tau$ but rather it is a block of $\tau$ and finally we subtract the second product $\prod_{\substack{B\in \rho}}\beta_{|B|}$ as that would correspond to the case when there is no block of $\rho$ that is union of blocks of $\tau$. To illustrate the validity of (\ref{Aux: auxliar of last proof of non-crossing ones 2}) in Example \ref{Example: The pairs corresponding to maximal block} we compute both sides of (\ref{Aux: auxliar of last proof of non-crossing ones 2}) for the case $r=2$ and $m_1=m_2=m_3=m_4=2$.

We would like to get a sum indexed by the whole set $\PS_{\NC_2}^{loop-free}(m_1,\dots\ab ,m_r)$ instead of just $\mathcal{M}$. This can be be easily done be letting,
\begin{multline*}
\widehat{\K}_{(\rho,\sigma^{new})}=2^{\#(\sigma^{new})-\#(\rho)} \\
\mbox{}\times
\left[\prod_{\substack{B\in\rho \\ |B|=2n}}\left(\beta_{|B|}+\mathbbm{1}_{D_B=\emptyset}\sum_{\substack{\tau\in A_n}}\prod_{D\in \tau}\beta_{|D|}\right)-\prod_{B\in\rho}\beta_{|B|}\right].
\end{multline*}
Then,
\begin{equation*}
\sum_{(\tau,\pi)\in \mathcal{B}}\C_{\tau}(\pi)=\sum_{\substack{(\rho,\sigma^{new})\in \PS_{\NC_2}^{loop-free}(m_1,\dots,m_r)}}\widehat{\K}_{(\rho,\sigma^{new})}.
\end{equation*}
Because whenever $(\sigma^{new},\rho)$ has no maximal blocks then the first and second product cancels out so that we get $0$, otherwise we get the expression that we had before.
\end{proof}

\begin{example}\label{Example: The pairs corresponding to maximal block}
Let $r=4$ and $m_1=m_2=m_3=m_4=2$. The unique non-crossing pairing with a maximal block is
$$(\rho,\sigma^{new})=(1_8,\{1,2\},\{3,4\},\{5,6\},\{7,8\}).$$
Let us determine all pairings $(\pi,\tau)$ that correspond to the same $(\rho,\sigma^{new})$. First, there are $8=2^{\#(\sigma)-\#(\rho)}$ partitions $\pi$. For each $\pi$ we can choose $3$ partitions $\tau$ which is precisely the number of partitions of $\cP(8)$ that doesn't admit a non-crossing partitioned permutation. Thus there must be $24=8\times 3$ pairs $(\pi,\tau)$ that correspond to the same $(\rho,\sigma^{new})$. In the table we enumerate all of these.
\begin{table}[h]
\begin{tabular}{|cccc|}
\hline
$\pi$ & $\tau_1$  & $\tau_2$  & $\tau_3$      \\
\hline
$\{1,3,5,7\}\{2,4,6,8\}$ & $\{1,4,5,8\}\{2,3,6,7\}$ & $\{1,4,6,7\}\{2,3,5,8\}$ & $\{1,3,6,8\}\{2,4,5,7\}$ \\
\hline
$\{1,3,5,8\}\{2,4,6,7\}$ & $\{1,4,5,7\}\{2,3,6,8\}$ & $\{1,4,6,8\}\{2,3,5,7\}$ & $\{1,3,6,7\}\{2,4,5,8\}$ \\
\hline
$\{1,3,6,7\}\{2,4,5,8\}$ & $\{1,4,6,8\}\{2,3,5,7\}$ & $\{1,4,5,7\}\{2,3,6,8\}$ & $\{1,3,5,8\}\{2,4,6,7\}$ \\
\hline
$\{1,3,6,8\}\{2,4,5,7\}$ & $\{1,4,6,7\}\{2,3,5,8\}$ & $\{1,4,5,8\}\{2,3,6,7\}$ & $\{1,3,5,7\}\{2,4,6,8\}$ \\
\hline
$\{1,4,5,7\}\{2,3,6,8\}$ & $\{1,3,5,8\}\{2,4,6,7\}$ & $\{1,3,6,7\}\{2,4,5,8\}$ & $\{1,4,6,8\}\{2,3,5,7\}$ \\
\hline
$\{1,4,5,8\}\{2,3,6,7\}$ & $\{1,3,5,7\}\{2,4,6,8\}$ & $\{1,3,6,8\}\{2,4,5,7\}$ & $\{1,4,6,7\}\{2,3,5,8\}$ \\
\hline
$\{1,4,6,7\}\{2,3,5,8\}$ & $\{1,3,6,8\}\{2,4,5,7\}$ & $\{1,3,5,7\}\{2,4,6,8\}$ & $\{1,4,5,8\}\{2,3,6,7\}$ \\
\hline
$\{1,4,6,8\}\{2,3,5,7\}$ & $\{1,3,6,7\}\{2,4,5,8\}$ & $\{1,3,5,8\}\{2,4,6,7\}$ & $\{1,4,5,7\}\{2,3,6,8\}$ \\    
\hline
\end{tabular}
\end{table}
Let us now compute both sides of (\ref{Aux: auxliar of last proof of non-crossing ones 2}). The left hand side sums over all pairs $(\pi,\tau)$ which in this case there are $24$ of them, each one with contribution $\C_\tau(\pi)=\beta_4^2$ so that the left hand side becomes $24\beta_4^2$. The right hand side just have a single term $(\rho,\sigma^{new})=(1_8,\{1,2\},\{3,4\},\{5,6\},\{7,8\})$. For this term and the unique block of $\rho$ we get
\begin{eqnarray*}
&& 2^{\#(\sigma^{new})-\#(\rho)}\left[ \prod_{\substack{B\in \rho}}\left(\beta_{|B|}+\mathbbm{1}_{D_B=\emptyset}\sum_{\tau\in A_{|B|/2}}\prod_{D\in\tau}\beta_{|D|}\right)-\prod_{\substack{B\in \rho}}\beta_{|B|} \right] \\
&=& 8\left[\beta_8+3\beta_4^2-\beta_8\right]=24\beta_4^2.
\end{eqnarray*}
\end{example}

\begin{figure}
    \centering
    \includegraphics[width=0.75\textwidth]{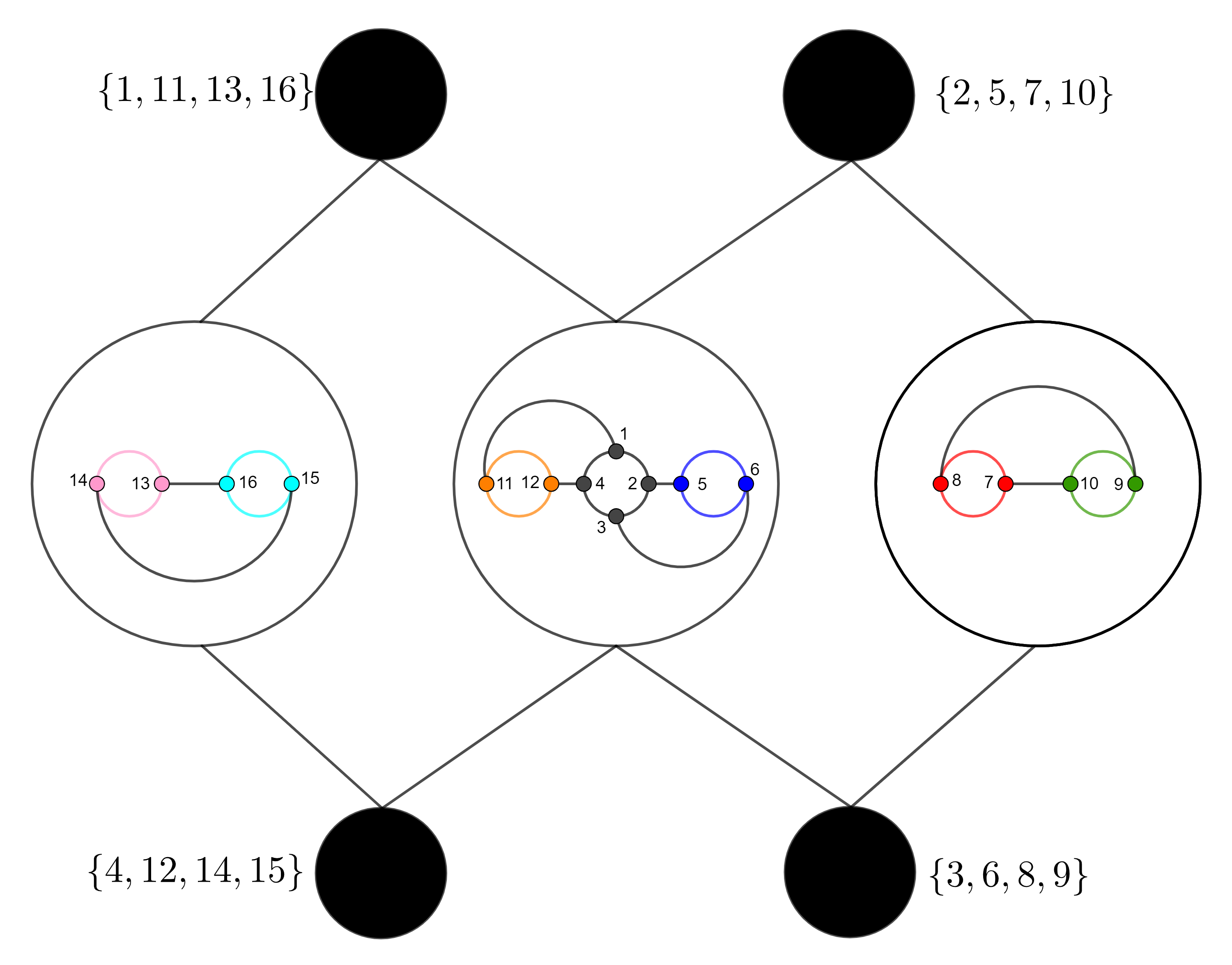}
    \text{A) The graph $T^{\gamma\sigma}(\sigma,\tau)$ of Example \ref{Example: Big example}. Note that for any $B\in\overline{\pi}$} 
    \text{the graph $T^{\gamma\sigma}(\sigma,\tau)_B$ is not a tree.}\\

    \includegraphics[width=0.8\textwidth]{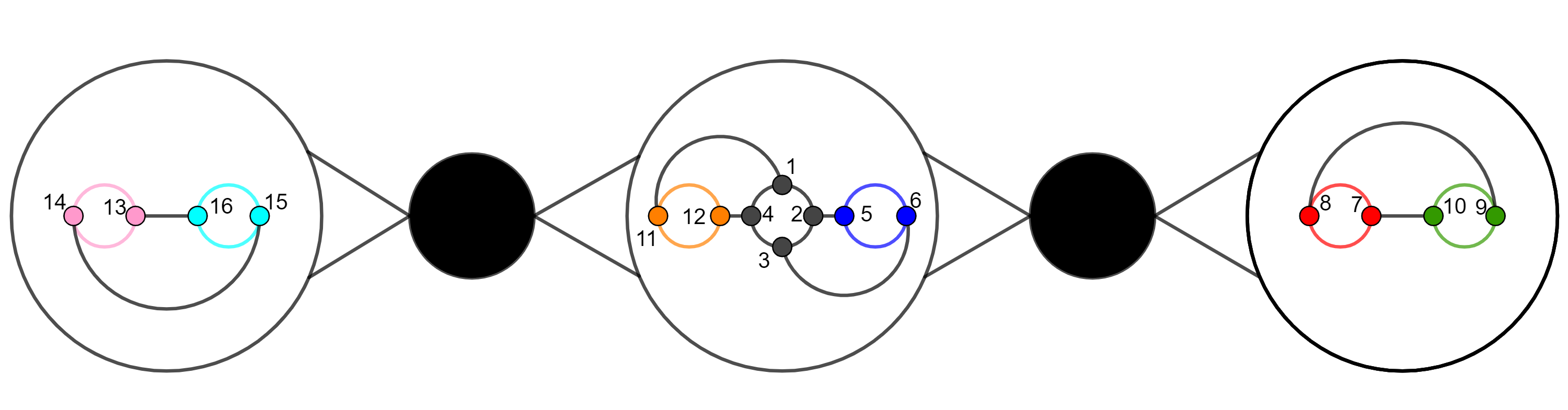}
    \text{B) In this case $\rho=\overline{\pi}$ as checked in Example \ref{Example: Example of rho}.} 
    \text{Here the graph $T^{\gamma\sigma}(\sigma,\tau)^{\overline{\pi}}$ which is the same as $T^{\gamma\sigma}(\sigma,\rho)$.}\\
    
    \includegraphics[width=0.8\textwidth]{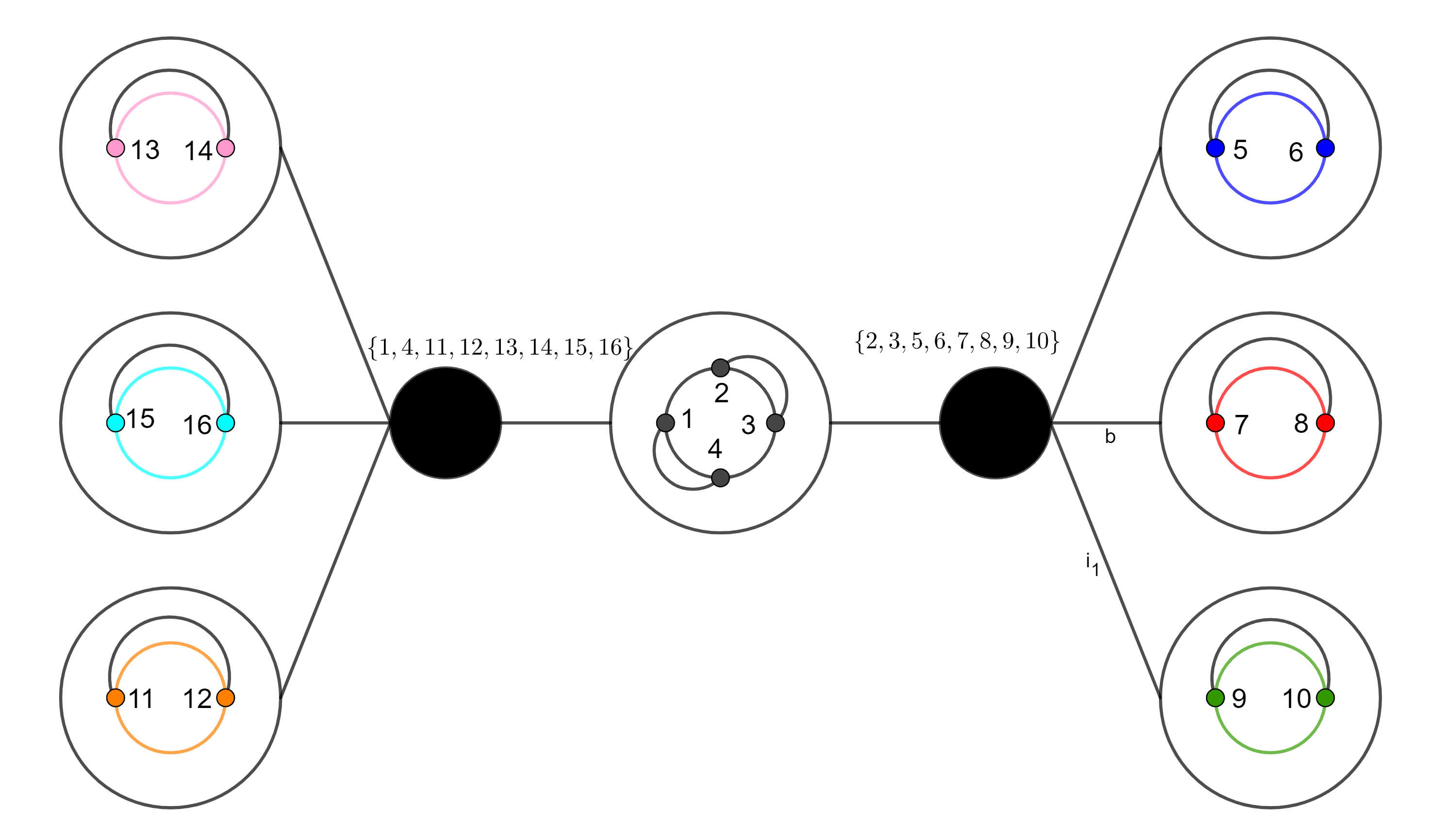}
    \text{C) The resulting pairing $\sigma^{new}=(1,4)(2,3)(5,6)(7,8)(9,10)(11,12)$}
    \text{$(13,14)(15,16)$. Here the graph $T^{\gamma\sigma^{new}}(\sigma^{new},\rho).$}\\
    
    \caption{The graphs of Examples \ref{Example: Big example} and \ref{Example: Example of rho}}
    \label{Figure: Big example complement}
\end{figure}

\subsection{Higher order cumulants}

Let us recall the two Equations,
\begin{equation}\label{Equation: Identity for non-crossing triples}
\sum_{(\tau,\pi)\in \mathcal{A}} =
\sum_{(\tau,\sigma)\in \PS_{NC_2}^{loop-free}(m_1,\dots,m_r)} \widetilde{\K}_{(\tau,\sigma)}
\end{equation}
\begin{equation}\label{Equation: Identity for crossing triples}
\sum_{(\tau,\pi)\in \mathcal{B}}\C_{\tau}(\pi)=\sum_{\substack{(\rho,\sigma^{new})\in \PS_{\NC_2}^{loop-free}(m_1,\dots,m_r)}}\widehat{\K}_{(\rho,\sigma^{new})}.
\end{equation}
With,
$$\widetilde{\K}_{(\tau,\sigma)}=2^{\#(\sigma)-\#(\tau)}\prod_{B\in\tau}\beta_{|B|},$$
and,
\begin{multline*}
\widehat{\K}_{(\rho,\sigma^{new})}=2^{\#(\sigma^{new})-\#(\rho)} \\
\mbox{}\times
\left[\prod_{\substack{B\in\rho \\ |B|=2n}}\left(\beta_{|B|}+\mathbbm{1}_{D_B=\emptyset}\sum_{\substack{\tau\in A_n}}\prod_{D\in \tau}\beta_{|D|}\right)-\prod_{B\in\rho}\beta_{|B|}\right].
\end{multline*}
This motivates the following definition.
\begin{definition}\label{Definition: The pseudo cumulants}
For a non-crossing partitioned permutation $(\tau,\sigma)\in \PS_{\NC_2}^{loop-free}(m_1,\dots,m_r)$ we let $\overline{\K}_{(\tau,\sigma)}$ be defined as
$$\overline{\K}_{(\tau,\sigma)}=\prod_{B\in\tau}\overline{\K}_{(B,\sigma|_B)},$$
with
$$\overline{\K}_{(B,\sigma|_B)}=2^{r-1}\left(\beta_{2r}+\mathbbm{1}_{D_B=\emptyset}\sum_{\substack{\tau\in \cP(2r) \\ \tau\text{ doesn't admit }\\ \text{a non-crossing} \\ \text{partitioned permutation}}}\prod_{D\in\tau}\beta_{|D|}\right),$$
whenever $B$ consists of $r$ cycles of $\sigma$ of size $2$.
We call $\overline{\K}_{(\tau,\sigma)}$ the \textit{pseudo-cumulants} of the Complex Wigner Matrix.
\end{definition}

Thus based in Equations \ref{Equation: Identity for non-crossing triples} and \ref{Equation: Identity for crossing triples} we get the following Lemma.

\begin{lemma}\label{Lemma: Main expression for moments first approach}
For any $m_1,\dots,m_r\in\mathbb{N}$,
\begin{equation}\label{Equation: Main expression for moments first approach}
\alpha_{m_1,\dots,m_r}=\sum_{(\tau,\sigma)\in \PS_{NC_2}^{loop-free}(m_1,\dots,m_r)} \overline{\K}_{(\tau,\sigma)},
\end{equation}
where $\overline{\K}_{(\tau,\sigma)}$ are the pseudo-cumulants of the Complex Wigner Matrix.
\end{lemma}

Up to this point we are pretty close to finish our main goal, finding the free cumulants of the Complex Wigner ensemble. Now we face with two last problems. The first one is writing the moments as a summation indexed by the whole set of non-crossing partitioned permutations rather than $\PS_{NC_2}^{loop-free}(m_1,\dots,m_r)$. This will be simple by extending our summation first to the set of non-crossing partitioned permutations $(\tau,\sigma)$ where $\sigma$ has only cycles of size $1$ and $2$ and then it extends trivially to all the set of non-crossing partitioned permutations. We do this at the beginning of this section. The second problem is that we want the sequence of cumulants to be multiplicative. In the sense that for each block $B\in \tau$ the term $\overline{\K}_{(B,\sigma|_B)}$ depends completely on the number of cycles of $\sigma|_B$ contained in $B$. Although we would like to define the sequence $\overline{\K}_{(\tau,\sigma)}$ as the multiplicative extension of a sequence $\overline{\K}_{n_1,\dots,n_r}$, this is impossible as for each block of $\tau$ the contribution is different depending whether this block contains another block related to it. Note that this problem relies completely on the fact that in the expression for $\overline{\K}_{(B,\sigma|_B)}$ the term containing $\mathbbm{1}_{D_B=\emptyset}$ depends on the block $B$ rather than just in the number of cycles of $\sigma$ contained in $B$. This suggest that if we want to find the higher order free cumulants of any order we need to find another sequence $\K^\prime_{(\tau,\sigma)}$ where for each block $B$ of $\tau$ the term $\K^\prime_{(B,\sigma|_B)}$ depends only on the number of cycles of $\sigma$ contained in $B$. This will be done at the very end of the section. Let us then move with our first goal. Let us recall that by $\PS_{\NC_{2,1}}(m_1,\dots,m_r)$ we mean all the non-crossing partitioned permutations for which any cycle of $\sigma$ has either size $1$ or $2$.

\begin{lemma}\label{Lemma: Cumulants when admit a multipliciative extension}
For a non-crossing partitioned permutation $(\tau,\sigma)\in \PS_{\NC_2}^{loop-free}(m_1,\dots,m_r)$ let $\K^\prime_{(\tau,\sigma)}$ be defined as
$$\K^\prime_{(\tau,\sigma)}=\prod_{B\in\tau}\K^\prime_{(B,\sigma|_B)},$$
with
$$\K^\prime_{(B,\sigma|_B)}=\delta_{r},$$
whenever $B$ consists of $r$ cycles of $\sigma$ of size $2$. Here $\delta_r$ is just a number that depends only on $r$. We extend $\K^\prime_{(\tau,\sigma)}$ to any $(\tau,\sigma)\in \PS_{\NC_{2,1}}(m_1,\dots,m_r)$ by letting,
$$\K^\prime_{(B,\sigma|_B)}=(-1)^{u/2}\frac{u!}{2^{u/2}(u/2)!}\delta_r,$$
whenever the block $B\in \tau$ consist of $v$ cycles of $\sigma$ of size $2$, $u$ cycles of $\sigma$ of size $1$ and $r=u/2+v$ with $(u,v)\neq (2,0)$ and $0$ otherwise. Then,
\begin{eqnarray*}
\sum_{(\tau,\sigma)\in \PS_{\NC_{2,1}}(m_1,\dots,m_r)} \K^\prime_{(\tau,\sigma)}=\sum_{(\tau,\sigma)\in \PS_{\NC_2}^{loop-free}(m_1,\dots,m_r)} \K^\prime_{(\tau,\sigma)}.
\end{eqnarray*}
\end{lemma}
\begin{proof}
Note that the set $\PS_{\NC_{2,1}}(m_1\dots,m_r)$ is the disjoint union of the sets,
$$\PS_{\NC_2}^{loop-free}(m_1,\dots,m_r),$$
$$\PS_{\NC_2}(m_1,\dots,m_r)\setminus \PS_{\NC_2}^{loop-free}(m_1,\dots,m_r),$$
and,
$$\PS_{\NC_{2,1}}^{!}(m_1,\dots,m_r).$$
Where $\PS_{\NC_{2,1}}^{!}(m_1,\dots,m_r)$ is the subset of $\PS_{\NC_{2,1}}(m_1\dots,m_r)$ that contains at least one block of size $1$. Therefore,

\begin{eqnarray*}
\sum_{(\tau,\sigma)\in \PS_{NC_{2,1}}(m_1,\dots,m_r)} \K^\prime_{(\tau,\sigma)}&=&\sum_{(\tau,\sigma)\in \PS_{NC_{2,1}}^{!}(m_1,\dots,m_r)} \K^\prime_{(\tau,\sigma)} \\
&+& \sum_{(\tau,\sigma)\in \PS_{NC_{2}}(m_1,\dots,m_r)\setminus \PS_{\NC_2}^{loop-free}(m_1,\dots,m_r)} \K^\prime_{(\tau,\sigma)} \\
&+&\sum_{(\tau,\sigma)\in \PS_{NC_{2}^{loop-free}}(m_1,\dots,m_r)} \K^\prime_{(\tau,\sigma)}.
\end{eqnarray*}

Let $(\tau,\sigma)\in \PS_{\NC_{2,1}}^{!}(m_1,\dots,m_r)$. We can assume that any block of $\tau$ that is also a block of $\sigma$ must be of size $2$ otherwise $\K^\prime_{(\tau,\sigma)}=0$. Assume $(\tau,\sigma)$ is such that if $\{r,s\}$ is a block of $\sigma$ that is not a block of $\tau$ then $\{r,s\}$ it is either a non-through string or it is a non cutting through string. Let $B$ be a block of $\tau$ that consists of $v$ blocks of $\sigma$ that are of size $2$ and $u$ blocks of $\sigma$ that are of size $1$. If $u$ is odd then $\K^\prime_{(\tau,\sigma)}=0$ as $\K^\prime_{(1_B,\sigma|_B)}=0$ so we assume $u$ is even and positive. We let $r=u/2+v$. 

Let $u=2n$ and let $\{u_1\},\dots,\{u_{2n}\}$ be the blocks of size $1$. Let $0\leq k\leq n$. We choose $2k$ singletons $u_{i_1},\dots,u_{i_{2k}}$ and pair so that we get $k$ blocks all of size $2$. Assume the blocks are of the form $\{u_{i_1},u_{i_2}\},\dots,\{u_{i_{2k-1}},u_{i_{2k}}\}$. Let $\sigma^\prime$ be the permutation $\sigma$ except by the singletons $\{u_{i_1}\},\dots,\{u_{i_{2k}}\}$ which we turn into the pairs $\{u_{i_1},u_{i_2}\},\dots,\ab\{u_{i_{2k-1}},u{i_{2k}}\}$. By Proposition \ref{Proposition: loop edge iff cutting through string} the through strings $\{u_{i_j},u_{i_{i+1}}\}$ are all cutting as removing this block increases the number of connected components of $\gamma\vee\sigma^\prime$ otherwise $(\tau,\sigma)$ wouldn't be a non-crossing partitioned permutation. The partitioned permutation $(\tau,\sigma^\prime)$ is still non-crossing ans it is contained in $$\PS_{\NC_{2,1}}^{!}(m_1,\dots,m_r)\cup (\PS_{\NC_2}(m_1,\dots,m_r) \setminus \PS_{\NC_2}^{loop-free}(m_1,\dots,m_r)).$$
Thus, summing over all $k$ and all choices of singletons $u_{i_1},\dots,u_{i_{2k}}$ gives a sum over all partitioned permutations in
$$\PS_{\NC_{2,1}}^{!}(m_1,\dots,m_r)\cup (\PS_{\NC_2}(m_1,\dots,m_r) \setminus \PS_{\NC_2}^{loop-free}(m_1,\dots,m_r)).$$
The contribution of each choice restricted to the block $B$ is,
$$(-1)^{n-k}\frac{(2n-2k)!}{2^{n-k}(n-k)!}\delta_r.$$

The number of ways of choosing the singletons is $\binom{2n}{2k}$ while the number of ways of pairing them is $|\cP_2(2k)|=\frac{(2k)!}{2^kk!}$. Finally note that,
\begin{eqnarray*}
&&\sum_{k=0}^n \binom{2n}{2k}|\cP_2(2k)|(-1)^{n-k}\frac{(2n-2k)!}{2^{n-k} (n-k)!}\\
&=& 
\sum_{k=0}^n \binom{2n}{2k}\frac{(2k)!}{2^k k!}(-1)^{n-k}\frac{(2n-2k)!}{2^{n-k} (n-k)!} \\
&=& 
\sum_{k=0}^n \frac{(2n)!}{2^n}\frac{1}{(n-k)!k!}(-1)^{n-k} \\
&=& 
\frac{(2n)!}{n!2^n}\sum_{k=0}^n \binom{n}{k}(-1)^{n-k} \\
&=& 
\frac{(2n)!}{n!2^n}(1+(-1))^n=0.
\end{eqnarray*}
Since $\K^\prime_{(\sigma,\tau)}$ is the product of the contribution at each block and at each block we get $0$, then,
\begin{eqnarray*}
&&\sum_{(\tau,\sigma)\in \PS_{NC_{2}}(m_1,\dots,m_r)\setminus \PS_{\NC_2}^{loop-free}(m_1,\dots,m_r)} \K^\prime_{(\tau,\sigma)} \\
&+&\sum_{(\tau,\sigma)\in \PS_{\NC_{2,1}}^{!}(m_1,\dots,m_r)} \K^\prime_{(\tau,\sigma)} \\
&=& 0.
\end{eqnarray*}
\end{proof}

\begin{theorem}\label{Theorem: Cumulants when admit a multipliciative extension 2}
For a non-crossing partitioned permutation $(\tau,\sigma)\in \PS_{\NC_{2,1}}(m_1,\dots,m_r)$ we let $\K^\prime_{(\tau,\sigma)}$ to be defined as in Lemma \ref{Lemma: Cumulants when admit a multipliciative extension}.
If $\alpha^\prime_{m_1,\dots,m_r}$ is a moment sequence that satisfies

\[
\alpha^\prime_{m_1,\dots,m_r}=\sum_{(\tau,\sigma)\in\PS_{\NC_{2,1}}(m_1,\dots,m_r)}\K^\prime_{(\tau,\sigma)},
\]
then the cumulant sequence associated to the moment sequence $\alpha^\prime_{m_1,\dots,m_r}$ is given by
$$\K_{1,\dots,1,2,\dots,2}=(-1)^{u/2}\frac{u!}{2^{u/2}(u/2)!}\delta_r,$$
where in above expression there are $v$ indices that are $2$, $u$ indices that are $1$, $u$ is even, and $r=u/2+v$ with $(u,v)\neq (2,0)$ and $0$ otherwise.
\end{theorem}
\begin{proof}
The proof follows immediately by noting that if we let the cumulants to be defined as before then
$$\sum_{(\tau,\sigma)\in\PS_{\NC}(m_1,\dots,m_r)}\K_{(\tau,\sigma)}=\sum_{(\tau,\sigma)\in\PS_{\NC_{2,1}}(m_1,\dots,m_r)}\K_{(\tau,\sigma)}.$$
The right hand side of last equality is the same as $\sum_{(\tau,\sigma)\in\PS_{\NC_{2,1}}(m_1,\dots,m_r)}\ab\K^\prime_{(\tau,\sigma)}$ which equals $\alpha^\prime_{m_1,\dots,m_r}$. Thus the cumulants satisfy the moment-cumulant relation as required.
\end{proof}

Lemma \ref{Lemma: Cumulants when admit a multipliciative extension} and Theorem \ref{Theorem: Cumulants when admit a multipliciative extension 2} shows that to conclude our goal we are reduced to find and expression similar to Equation \ref{Equation: Main expression for moments first approach}, namely,
\begin{equation}
\alpha_{m_1,\dots,m_r}=\sum_{(\tau,\sigma)\in \PS_{NC_2}^{loop-free}(m_1,\dots,m_r)} \K^\prime_{(\tau,\sigma)}.
\end{equation}
where the term $\K^\prime_{(\tau,\sigma)}$ can be expressed as the multiplicative extension of a sequence $\K^\prime_{n_1,\dots,n_r}$. As pointed out before, the problem with the pseudo cumulants sequence $\overline{\K}_{(\tau,\sigma)}$ is that for each block $B\in \tau$ we define,
$$\overline{\K}_{(B,\sigma|_B)}=2^{r-1}\left(\beta_{2r}+\mathbbm{1}_{D_B=\emptyset}\sum_{\substack{\tau\in \cP(2r) \\ \tau\text{ doesn't admit }\\ \text{a non-crossing} \\ \text{partitioned permutation}}}\prod_{D\in\tau}\beta_{|D|}\right),$$
whenever $B$ consists of $r$ cycles of $\sigma$ of size $2$. The latest means that we are only adding the second term when a block $B$ is maximal. If we get rid of the term $\mathbbm{1}_{D_B=\emptyset}$ then it might be possible that we count more than once the same partition $\tau\in \cP(2r)$ (for instance, see Example \ref{Example: overcounting of bad tau}). To avoid the overcounting we can simply subtract the partitions that have been counted before. This is formalized in the next Lemma.

\begin{example}\label{Example: overcounting of bad tau}
Consider the two non-crossing partitioned permutations $(\tau_1,\sigma_1),(\tau_2,\sigma_2)$ in $\PS_{\NC_2}^{loop-free}(2,2,2,2,2)$ give by,
\begin{eqnarray*}
\sigma_1 &=& (1,4)(2,3)(5,6)(7,8)(9,10) \\
\tau_1 &=& \{1,4\},\{2,3,5,6,7,8,9,10\} \\
\sigma_2 &=& (1,2)(3,4)(5,6)(7,8)(9,10) \\
\tau_2 &=& \{1,2,3,4,5,6,7,8,9,10\}.
\end{eqnarray*}
If we get rid of the term $\mathbbm{1}_{D_B=\emptyset}$ in our formula for the pseudo-cumulants then for $(\tau_1,\sigma_1)$ we are counting the partition,
$$\tau=\{2,6,7,9\},\{3,5,8,10\}$$
which comes from the block $B=\{2,3,5,6,7,8,9,10\}$ of $\tau_1$. The other block $B=\{1,4\}$ of $\tau_1$ determines the other block of $\tau$ so that $\tau=\{1,4\},\{2,6,7,9\},\{3,5,8,10\}$. This partition is one that belongs to $\mathcal{B}$ and contributes $\beta_2\beta_4^2$ in the term $\overline{\K}_{(\tau_1,\sigma_1)}$. On the other hand, for $(\tau_2,\sigma_2)$ if we get rid of the term $\mathbbm{1}_{D_B=\emptyset}$ then we are allowed to choose again $\tau=\{1,4\},\{2,6,7,9\},\{3,5,8,10\}$. This implies we counted twice the same partition but we just wanted it to be counted once which correspond to the case when the unique block of $\tau_2$ is maximal. To get rid of the overcounting we will let $\tau$ to be counted in $(\tau_1,\sigma_1)$ and then subtract from $(\tau_2,\sigma_2)$ the ones that were counted in $(\tau_1,\sigma_1)$. To make this possible observe that $\tau \leq \tau_1 \leq \tau_2$ so for given $(\tau_2,\sigma_2)$ we need to subtract the ones counted for any other partitioned permutation $(\tau_1,\sigma_1)$ where $\tau_1 \leq \tau_2$. We then continue doing this process recursively. We explain this in detail below in Lemma \ref{Lemma: From pseudo to real cumulants}.

\end{example}

\begin{lemma}\label{Lemma: From pseudo to real cumulants}
Let $r\in\mathbb{N}$ and
$$A_{\underbrace{2,\dots,2}_{r\text{ indices}}}=\{(\cU,\delta)\in \PS_{\NC_2}(2,\dots,2): \cU\neq 1_{2r}\}.$$
Let
$$C_{2r}=\sum_{\substack{\tau\in \cP(2r) \\ \tau\text{ doesn't admit }\\ \text{a non-crossing} \\ \text{partitioned permutation}}}\prod_{D\in\tau}\beta_{|D|}.$$
For $(\tau,\sigma)\in \PS_{\NC_2}^{loop-free}(m_1,\dots\ab,m_r)$ let
$$\K^\prime_{(\tau,\sigma)}=\prod_{B\in\tau}\K^\prime_{(B,\sigma|_B)},$$
with $\K^\prime_{(B,\sigma|_B)}$ the recursive sequence given by,
$$\K^\prime_{(B,\sigma|_B)}=2^{r-1}\left(\beta_{2r}+C_{2r}\right)-\sum_{(\cU,\sigma)\in A_{\underbrace{2,\dots ,2}_{r \text{ indices}}}}\prod_{V\in U}\K^\prime_{\underbrace{2,\dots ,2}_{\#(\sigma|_V)\text{ indices}}}-\prod_{V\in U}2^{\#(\sigma)-1}\beta_{|V|},$$
whenever $B$ consists of $r$ cycles of $\sigma$ of size $2$ for $r\geq 2$ and $\K^\prime_{2}=\beta_2$. Then
\begin{equation}\label{Auxiliar: relation between cumulants and pseudo cumulants}
\sum_{(\tau,\sigma)\in \PS_{NC_2}^{loop-free}(m_1,\dots ,m_r)} \overline{\K}_{(\tau,\sigma)}=\sum_{(\tau,\sigma)\in \PS_{NC_2}^{loop-free}(m_1,\dots,m_r)} \K^\prime_{(\tau,\sigma)}.
\end{equation}
Here $\overline{\K}_{(\tau,\sigma)}$ are the pseudo-cumulants of the Complex Wigner Matrix.
\end{lemma}
\begin{proof}
We first prove it for any $r$ and $m_1=\cdots =m_r=2$ then we will extend it to any $m_i$. So let us assume $m_1=\cdots =m_r=2$. We proceed inductively over $r$ defining the cumulants $\K^\prime_{2,\dots,2}$ accordingly so that (\ref{Auxiliar: relation between cumulants and pseudo cumulants}) holds, here $\K^\prime_{2,\dots,2}$ has $r$ indices $2's$. 

Note that for the cases $r=1,2,3$ the term $C_{2r}=0$ and therefore in these cases we simply have
\begin{eqnarray*}
\K^\prime_{2}=\beta_2 \qquad \K^\prime_{2,2}=2\beta_4 \qquad \K^\prime_{2,2,2}=4\beta_6.
\end{eqnarray*}
It is actually well know that above expressions are indeed the cumulants of the complex Wigner ensemble. The case $r=1$ is the Wigner's semicircle law. For $r=2,3$ see \cite{MMPS,MM} respectively. Assume we have defined $\K^\prime_{2,\dots,2}$ for at most $r-1$ indices $2's$. Let us define $\K^\prime_{2,\dots,2}$ for the case when there are $r$ indices $2's$ so that (\ref{Auxiliar: relation between cumulants and pseudo cumulants}) holds. Since we are assuming $m_i=2$ for any $i$ then there is only one non-crossing partitioned permutation that is maximal, that correspond to
$$(\tau,\sigma)=(1_{2r},(1,2)\cdots (r-1,r)).$$
Hence the left hand side of (\ref{Auxiliar: relation between cumulants and pseudo cumulants}) becomes
$$\sum_{(\cU,\delta)\in A_{\underbrace{2,2,2,2}_{r\text{ indices}}}}\prod_{V\in U}2^{\#(\sigma|_V-1)}\beta_{|V|}+2^{r-1}(\beta_{2r}+C_{2r}).$$
On the other hand the right hand side of (\ref{Auxiliar: relation between cumulants and pseudo cumulants}) becomes
$$\sum_{(\cU,\delta)\in A_{\underbrace{2,\dots ,2}_{r\text{ indices}}}}\prod_{V\in U}\K^\prime_{\underbrace{2,\dots ,2}_{\#(\sigma|_V)\text{ indices}}} + \K^\prime_{\underbrace{2,\dots ,2}_{r\text{ indices}}}.$$
To make both expression equal each other we must define
$$\K^\prime_{(B,\sigma|_B)}=2^{r-1}\left(\beta_{2r}+C_{2r}\right)-\sum_{(\cU,\sigma)\in A_{\underbrace{2,\dots ,2}_{r \text{ indices}}}}\prod_{V\in U}\K^\prime_{\underbrace{2,\dots ,2}_{\#(\sigma|_V)\text{ indices}}}-\prod_{V\in U}2^{\#(\sigma)-1}\beta_{|V|},$$
for the block $B=1_{2r}$ that consist of $r$ cycles of size $2$ of $\sigma$ as desired.

Finally let us extend this to any $m_1,\dots,m_r$. First of all, we claim that for any $(\tau,\sigma)\in \PS_{\NC_2}^{loop-free}(m_1,\dots,m_r)$ there exist $(\rho,\sigma^\prime)\in \PS_{\NC_2}^{loop-free}(m_1,\dots,m_r)$ such that any block of $\rho$ which is union of blocks of $\sigma^\prime$ is maximal and $\tau \leq \rho$. Indeed, we first let the blocks of $\tau$ that contain a cutting through string of $\sigma$ be also blocks of $\rho$ (let us recall these blocks have all size $2$). If any other block of $\tau$ is maximal then we are done by letting $(\rho,\sigma^\prime)=(\tau,\sigma)$. Otherwise, there exist blocks $B_1,B_2\in \tau$ and $(a,b),(c,d)\in \sigma$ with $(a,b)\sim (c,d)$ (in the sense of Definition \ref{Definition: Equivalence class of blocks of partitioned permutation}) such that $(a,b)\subset B_1$ and $(c,d)\subset B_2$. Hence $(a,b)$ and $(c,d)$ must be in the same connected component of $\gamma\vee\sigma$. Assume without loss of generality $[a]_{\gamma\sigma}=[c]_{\gamma\sigma}$ and $[b]_{\gamma\sigma}=[d]_{\gamma\sigma}$. We let $\sigma^\prime=\sigma(a,b)(c,d)(a,d)(b,c)$ and $\rho$ to be the one with blocks the same as $\tau$ except $B_1$ and $B_2$ which is merged into the single block $B_1\cup B_2$ of $\rho$. By Lemma \ref{Lemma: Splitting edges produces two new non-crossing pairings} we get $(\rho,\sigma^\prime)\in \PS_{\NC_2}^{loop-free}(m_1,\dots,m_r)$ and the cycles $(a,d)$ and $(b,c)$ of $\sigma^\prime$ are not related anymore. We can continue doing the same process until we have no related cycles so that any block of $\rho$ which is union of blocks of $\sigma^\prime$ is maximal and $\tau \leq \rho$. Let $(\rho,\sigma^\prime)\in \PS_{\NC_2}^{loop-free}(m_1,\dots,m_r)$ as before and let $M$ to be the subset of $\PS_{\NC_2}^{loop-free}(m_1,\dots,m_r)$ that contains only $(\tau,\sigma)$ with $\tau\leq \rho$. Any $(\tau,\sigma)\in M$ can be written as the product of $(\tau_B,\sigma_B)$ where $B\in \rho$ is a block of $\rho$. Moreover at any block $B\in\rho$ by the proved at the beginning we know $\sum_{(\tau_B,\sigma_B)}\overline{\K}_{(\tau_B,\sigma_B)}=\sum_{(\tau_B,\sigma_B)}\K^\prime_{(\tau_B,\sigma_B)}$. Hence
$$\sum_{(\tau,\sigma)\in M}\overline{\K}_{(\tau,\sigma)}=\sum_{(\tau,\sigma)\in M}\K^\prime_{(\tau,\sigma)}.$$
If $M=\PS_{\NC_2}^{loop-free}(m_1,\dots,m_r)$ we are done, otherwise we take $(\tau,\sigma)\in \PS_{\NC_2}^{loop-free}(m_1,\dots,m_r) \setminus M$ and consider $(\rho,\sigma^\prime)$ as before and repeat the process over again until we cover the whole set $\PS_{\NC_2}^{loop-free}(m_1,\dots,m_r)$.
\end{proof}

\begin{lemma}\label{Lemma: From real cumulants in loop-free to all cumulants}
Let $\delta_r$ be the recursive sequence given by $\delta_1=\beta_2$ and
$$\delta_r=2^{r-1}\left(\beta_{2r}+C_{2r}\right)-\sum_{(\cU,\sigma)\in A_{\underbrace{2,\dots ,2}_{r \text{ indices}}}}\prod_{V\in U}\delta_{\#(\sigma|_V)}-\prod_{V\in U}2^{\#(\sigma)-1}\beta_{|V|},$$
for $r\geq 2$. Here $C_{2r}$ and $A_{2,\dots,2}$ are defined as in Lemma \ref{Lemma: From pseudo to real cumulants}. Then, the free cumulants of a Complex Wigner Matrix are given by
$$
\K_{1,\dots,1,2,\dots,2} = (-1)^{u/2}\frac{u!}{2^{u/2}(u/2)!}\delta_r,
$$
where there are $v$ indices that are $2$, $u$ indices that are $1$, $u$ is even, and $r=u/2+v$ with $(u,v)\neq(2,0)$ and $0$ otherwise.
\end{lemma}
\begin{proof}
By definition of the cumulants and Lemma \ref{Lemma: From pseudo to real cumulants} we know,
$$\sum_{(\tau,\sigma)\in \PS_{NC_2}^{loop-free}(m_1,\dots ,m_r)} \overline{\K}_{(\tau,\sigma)}=\sum_{(\tau,\sigma)\in \PS_{NC_2}^{loop-free}(m_1,\dots,m_r)} \K_{(\tau,\sigma)},$$
where $\overline{\K}_{(\tau,\sigma)}$ are the pseudo-cumulants and $\K_{(\tau,\sigma)}$ is the multiplicative extension of $\K_{2,\dots,2}$. The left hand side of last equation is $\alpha_{m_1,\dots,m_r}$ according to Lemma \ref{Lemma: Main expression for moments first approach} while the right-hand side equals to the same sum over the set $\PS_{\NC_{2,1}}(m_1,\dots,m_r)$ according to Lemma \ref{Lemma: Cumulants when admit a multipliciative extension}. From there we conclude by Theorem \ref{Theorem: Cumulants when admit a multipliciative extension 2}.
\end{proof}

Lemma \ref{Lemma: From real cumulants in loop-free to all cumulants} proves our main result \ref{Theorem: Main theorem 1}. On the other hand if $X$ is a GUE then $\beta_2=1$ and any other $\beta_n=0$. Thus only when any block of $\sigma$ is also a block of $\tau$ we have that $\overline{\K}_{(\tau,\sigma)}\neq 0$. Moreover in such a case we have $\overline{\K}_{(\tau,\sigma)}=\beta_2=1$. That correspond to the case when $\sigma\in \NC_2(m_1,\dots,m_r)$, so by Equation \ref{Equation: Main expression for moments first approach} it follows,
$$\alpha_{m_1,\dots,m_r}=|\NC_2(m_1,\dots,m_r)|,$$
which proves our Corollary \ref{Theorem: Main theorem 3}. Finally to prove our result Theorem \ref{Theorem: Main theorem 2} it is enough to appeal to Lemmas \ref{Lemma: Main expression for moments first approach} and \ref{Lemma: From pseudo to real cumulants}.

\subsection{Computation of lower order free cumulants}

Let us finish this work by computing the higher order free cumulants up to order $5$. This will illustrate how our formula works. Let us set $u$ be number of indices that are $1$ and $v$ the number of indices that are $2$. We also let $r=u/2+v$. Let us then compute $C_{2r}$ for $r=1,2,3,4$ first. To complete this calculation we are reduced to find the partitions $\tau\in \cP(2r)$ so that $\tau$ doesn't admit a non-crossing partitioned permutation. Let us note that when $r=1$ the unique partition $\tau\in \cP(2)$ that we are allowed to choose is $\tau=\{1,2\}$ and then only $\sigma=\{1,2\}$ satisfies $\sigma\leq \tau$ and it respects the parity. It is clear that $\Gamma(\tau,\sigma,\gamma_2^{par})$ is a tree so $C_2=0$ in our formula. For $r=2$ the only choices for $\tau$ are $\{1,4\}\{2,3\}$ or $\{1,2,3,4\}$. In the first one we have the only choice $\sigma=(1,4)(2,3)$ resulting in $\Gamma(\tau,\sigma,\gamma_4^{par})$ being a tree. In the second case we can choose $\sigma=(1,2)(3,4)$ resulting in $\Gamma(\tau,\sigma,\gamma_4^{par})$ being a tree. Thus $C_4$ is again $0$. For $r=3$ we can choose,
\begin{eqnarray*}
\tau &=& \{1,4\},\{2,5\},\{3,6\} \\
\tau &=& \{1,6\},\{2,3\},\{4,5\} \\
\tau &=& \{1,6\},\{2,3,4,5\} \\
\tau &=& \{1,4\},\{2,3,5,6\} \\
\tau &=& \{2,3\},\{1,4,5,6\} \\
\tau &=& \{2,5\},\{1,3,4,6\} \\
\tau &=& \{3,6\},\{1,2,4,5\} \\
\tau &=& \{4,5\},\{1,2,3,6\} \\
\tau &=& \{1,2,3,4,5,6\} \\
\end{eqnarray*}
In the first two choices the sigma is unique and $\Gamma(\tau,\sigma,\gamma_6^{par})$ is a tree. In the third choice we take $\sigma=(1,6)(2,5)(3,4)$ so that $\Gamma(\tau,\sigma,\gamma_6^{par})$ is a tree. Similarly for the $4^{th},5^{th},6^{th},7^{th}$ and $8^{th}$ we choose $\sigma=(1,4)(2,3)(5,6),\sigma=(1,4)(2,3)(5,6),\sigma=(1,4)(2,3)(5,6),\sigma=(1,6)\ab (2,5)(3,4)$ and $\sigma=(1,2)(3,6)(4,5)$ respectively so that $\Gamma(\tau,\sigma,\gamma_6^{par})$ is a tree. Finally for the last one we choose $\sigma=(1,2)(3,4)(5,6)$ so that $\Gamma(\tau,\sigma,\gamma_6^{par})$ is a tree. So $C_6=0$. For the case $r=4$ it can be verified there are $3$ partitions $\tau$ for which there is no $\sigma$ so that $\Gamma(\tau,\sigma,\gamma_6^{par})$ is a tree. These are $\tau=\{1,4,5,8\},\{2,3,6,7\},\tau=\{1,4,6,7\},\{2,3,5,8\}$ and $\tau=\{1,3,6,8\},\{2,4,5,7\}$, each one in the sum contributes $\beta_4^2$. Moreover for each of these three partitions we have $C_\tau=\emptyset$, thus $C_8=3\beta4^2$. Now we just need to find the $\delta_r$ sequence which is defined recursively. By hypothesis $\delta_1=\beta_2$ and therefore for the first order cumulants we can only choose one index which has to be $2$, so $r=1$, since $C_2=0$ we have,
$$\K_2=\beta_2.$$

Now we compute $\delta_2$, by definition
$$\delta_2=2\left(\beta_{4}+C_{4}\right)-\sum_{(\cU,\sigma)\in A_{2,2}}\prod_{V\in U}\delta_{\#(\sigma|_V)}-\prod_{V\in U}2^{\#(\sigma)-1}\beta_{|V|}.$$
There are element two elements in $A_{2,2}$ corresponding to $(\{1,3\},\{2,4\},\ab (1,3)(2,4))$ and $(\{1,4\},\{2,3\},(1,3)(2,4))$. Since $\delta_1=\beta_2$ the second summation is $0$ while $C_4=0$, hence $\delta_2=2\beta_4$. For the second order we can only choose two indices, both of them $2$, so $r=2$. Therefore
$$\K_{2,2}=2\beta_4.$$
Now we compute $\delta_3$, by definition
$$\delta_3=4\left(\beta_{6}+C_{6}\right)-\sum_{(\cU,\sigma)\in A_{2,2,2}}\prod_{V\in U}\delta_{\#(\sigma|_V)}-\prod_{V\in U}2^{\#(\sigma)-1}\beta_{|V|}.$$
In this case as in the case $r=2$ the summation becomes $0$ while $C_6=0$, hence $\delta_3=4\beta_6$. For the third order case we can either choose three indices all of them $2$ or two indices that are $1$ and one which is $2$. In the first case we have $r=3$, therefore
$$\K_{2,2,2}=4\beta_6.$$
For the second case we have $r=u/2+v=2$, thus
$$\K_{2,1,1}=-2\beta_4.$$
Finally let us compute $\delta_4$. By definition
$$\delta_4=8\left(\beta_{8}+C_{8}\right)-\sum_{(\cU,\sigma)\in A_{2,2,2,2}}\prod_{V\in U}\delta_{\#(\sigma|_V)}-\prod_{V\in U}2^{\#(\sigma)-1}\beta_{|V|}.$$
In this case the summation again becomes $0$ while $C_8=3\beta_4^2$, hence $\delta_4=8\beta_8+24\beta_4^2.$ For the fourth order case we can choose either $4$ indices all of them $2$, or $2$ indices that are $2$ and two indices that are $1$ or $4$ indices all of them $1$. In the first case we have $r=4$m, therefore
$$\K_{2,2,2,2}=8(\beta_8+3\beta_4^2)=8\beta_8+24\beta_4^2.$$
For the case of two indices that are $2$ and two indices that are $1$ we have $r=u/2+v=1+2=3$, hence
$$\K_{1,1,2,2}=(-1)(4)\frac{2!}{2}\beta_6=-4\beta_6.$$
Finally if all four indices are $2$ then $r=u/2+v=2$, hence 
$$\K_{1,1,1,1}=(-1)^2(2)\frac{4!}{4(2!)}\beta_4=6\beta_4.$$
For the fifth order case we either have $\K_{2,2,2,2,2},\K_{1,1,2,2,2}$ or $\K_{1,1,1,1,2}$. For the second and third case it corresponds to $r=4$ and $r=3$ respectively, thus,
$$\K_{1,1,2,2,2}=(-1)8\frac{2!}{2}(\beta_8+3\beta_4^2)=-8\beta_8-24\beta_4^2,$$
and
$$\K_{1,1,1,1,2}=(-1)^24\frac{4!}{4(2!)}\beta_6=12\beta_6.$$
Finally we do not compute $\K_{2,2,2,2,2}$ as that implies to compute $C_{10}$ which means finding all possible $\tau$ in the case $r=5$, however that should be a very simple problem to implement in a coding language.


\thebottomline
\end{document}